\documentclass[10pt,twocolumn,numbers]{article}

\usepackage{graphicx}

\usepackage{amssymb,amsthm,amsmath,dsfont,bbm,mathtools,stmaryrd}

\usepackage{lineno}

\usepackage[colorlinks=true, citecolor=blue]{hyperref}
\usepackage[linesnumbered,lined,ruled,commentsnumbered]{algorithm2e}
\usepackage{algcompatible}

\usepackage{wrapfig}
\usepackage[numbers,sort&compress]{natbib}
\usepackage{comment}
\usepackage{longtable}
\usepackage{multirow}
\usepackage{bbding}
  
\usepackage[bottom]{footmisc}
\usepackage{booktabs}
\usepackage{pifont}
\usepackage{float}
\usepackage{subfigure}
\usepackage{wrapfig}
\usepackage{framed,xcolor}
\usepackage{newfloat}
\usepackage[xcolor,framemethod=TikZ]{mdframed}
\usepackage{amsthm}
\usepackage{thmtools}
\usepackage{authblk}
\patchcmd{\thebibliography}{\section*{\refname}}{}{}{}
\usepackage{widetext}
\usepackage{tikz}
\usetikzlibrary{arrows,shapes,chains,calc,patterns,decorations.pathreplacing}
\usepackage{stmaryrd}
\usepackage{fancyhdr}
\usepackage{subfigure}
\usepackage{comment}
\def\BibTeX{{\rm B\kern-.05em{\sc i\kern-.025em b}\kern-.08em
		T\kern-.1667em\lower.7ex\hbox{E}\kern-.125emX}}

\renewcommand{\headrulewidth}{2pt}

\newlength\FHoffset

\fancyheadoffset{\FHoffset}

\newlength\FHleft
\newlength\FHright

\newbox\FHline
\setbox\FHline=\hbox{\hsize=\paperwidth%
	\hspace*{\FHleft}%
	\rule{\dimexpr\headwidth-\FHleft-\FHright\relax}{\headrulewidth}\hspace*{\FHright}%
}

\newcommand{\RR}{\mathbb{R}}

\newcommand{\PP}{\mathsf{P}}

\newcommand{\BB}{\mathsf{B}}
\newcommand{\tr}{\mathrm{tr}}
\newcommand{\te}{\mathrm{te}}
\newcommand{\MM}{\mathbf{M}}
\newcommand{\UU}{\mathbf{U}}
\newcommand{\VV}{\mathbf{V}}
\newcommand{\WW}{\mathbf{W}}
\newcommand{\ZZ}{\mathbf{Z}}
\newcommand{\XX}{\mathbf{X}}
\newcommand{\YY}{\mathbf{Y}}
\newcommand{\KK}{\mathsf{K}}
\newcommand{\rank}{\mathsf{rank}}
\newcommand{\trace}{\mathsf{trace}}
\newcommand{\FF}{\mathrm{F}}
\newcommand{\DD}{\mathcal{D}}

\newtheoremstyle{theoremdd}
{\topsep}
{\topsep}
{\itshape}
{0pt}
{\fontfamily{cmss}\selectfont\bfseries}
{.}
{ }
{\thmname{#1}\thmnumber{ #2}\thmnote{ (#3)}}

\theoremstyle{theoremdd}
\newtheorem{theorem}{Theorem}[section]
\newtheorem{lemma}{Lemma}

\newtheorem{corollary}{Corollary}

\mdfdefinestyle{myStyle}{roundcorner=5pt,backgroundcolor=brown!20,linecolor=red!40!black,linewidth=1.5pt}

\usepackage{titlesec}
\titleformat*{\section}{\fontfamily{cmss}\selectfont\large\bfseries\color{red!40!black}}
\titleformat*{\subsection}{\fontfamily{cmss}\selectfont\normalsize\bfseries\color{red!40!black}}
\titleformat*{\subsubsection}{\fontfamily{cmss}\selectfont\normalsize\color{red!40!black}}
\usepackage[left=0.69 in, right=0.69 in, top=1.1 in, bottom=1.3 in]{geometry}
\graphicspath{{./Figures/}}

\renewcommand\abstractname{\fontfamily{cmss}\selectfont\normalsize\bfseries\color{red!40!black}\textbf{Abstract}}

\renewenvironment{abstract}{%
	\centering\small
	\list{}{\leftmargin1.5cm \rightmargin\leftmargin}
	\item\relax
	
	\begin{mdframed}[]
		\item[\hskip\labelsep\scshape\abstractname.]%
	}{%
	\end{mdframed}
	\endlist \par\bigskip
}
\makeatletter
\patchcmd{\@maketitle}{\LARGE \@title}{\fontfamily{cmss}\selectfont\LARGE\color{red!40!black}\@title}{}{}
\makeatother

\graphicspath{{./Figures/}}

\begin{document}

		
		
		\title{Nonlinear Traffic Prediction as a Matrix Completion Problem with Ensemble Learning}
		
			

\author[1]{Wenqing Li}
\author[2]{Chuhan Yang}
\author[1,2]{Saif Eddin Jabari$^{\star}$}

\affil[1]{New York University Abu Dhabi, Saadiyat Island, P.O. Box 129188, Abu Dhabi, U.A.E.}
\affil[2]{New York University Tandon School of Engineering, Brooklyn NY, U.S.A.}

\date{}


\twocolumn[
\begin{@twocolumnfalse}
	
\maketitle	

\begin{abstract}
	This paper addresses the problem of short-term traffic prediction for signalized traffic operations management.  Specifically, we focus on predicting sensor states in high-resolution (second-by-second).  This contrasts with traditional traffic forecasting problems, which have focused on predicting aggregated traffic variables, typically over intervals that are no shorter than 5 minutes.  Our contributions can be summarized as offering three insights: first, we show how the prediction problem can be modeled as a matrix completion problem.  Second, we employ a block-coordinate descent algorithm and demonstrate that the algorithm converges in sub-linear time to a block coordinate-wise optimizer.  This allows us to capitalize on the ``bigness'' of high-resolution data in a computationally feasible way.  Third, we develop an ensemble learning (or adaptive boosting) approach to reduce the training error to within any arbitrary error threshold.  The latter utilizes past days so that the boosting can be interpreted as capturing periodic patterns in the data.  The performance of the proposed method is analyzed theoretically and tested empirically using both simulated data and a real-world high-resolution traffic dataset from Abu Dhabi, UAE.  Our experimental results show that the proposed method outperforms other state-of-the-art algorithms. 
	
	\medskip
	
	\textbf{\fontfamily{cmss}\selectfont\color{red!40!black} Keywords}: Traffic prediction, high-resolution data, signalized intersections, adaptive control, matrix completion, kernel regression, sparse approximation, ensemble learning, adaptive boosting.
\end{abstract}
\bigskip
\end{@twocolumnfalse}
]

		
	
	
	

\section{Introduction}
\label{S:intro}
Governments worldwide expend considerable effort and investment to manage the day-to-day operations of urban networks.  Today's urban road networks are highly complex interacting systems of technologies that include sensory devices, communications technologies, and monitoring tools.  These serve a diverse set of operational objectives, from rapid response to incidents (\emph{safety}) to managing emergencies and special events (\emph{security}) to the day-to-day operation of traffic lights and congestion (\emph{efficiency}).  Adaptive control systems, namely traffic lights \citep{guo2019urban} and ramp meters \citep{shaaban2016literature} are key technologies employed by operators to achieve these objectives.  According to the Federal Highway Administration (FHWA), adaptive signal control technologies can improve travel time by more than 10\% and such improvements can exceed 50\% when they replace outdated controls \citep{asct}.  Decision-making in these systems occurs at a very fine time cadence, e.g., actuated/adaptive traffic controllers update their controls every time a vehicle is detected by one of the sensors in the system.  The control logic employed by adaptive controllers involves coordination among the controllers at different locations (e.g., adjacent intersections or consecutive ramp meters).  Coordination essentially entails one controller informing another to \emph{anticipate} vehicle arrivals, and then producing a coordinated response.  In other words, coordination involves the use of short-term predictions of vehicle arrivals to sensor stations.  In practice, this is done using rudimentary techniques, namely assuming uniform vehicle speeds between sensor locations. But such assumptions are only valid in low-volume traffic and prediction in high-volume traffic is more challenging.

Prediction problems in traffic have focused on applications that do not require the fine temporal resolutions needed for adaptive control. They almost exclusively use data that is aggregated at the 5-15 minute level.  Numerous studies have investigated the impact of data aggregation level on prediction accuracy \citep{ishak2002performance, guo2007data, tan2016short} These studies generally agree that prediction accuracies drop as the aggregation levels get finer. \cite{oh2005exploring} further demonstrated that data-driven techniques like neural networks tend to produce more accurate predictions than conventional statistical models (like time series) as the data aggregation levels get finer (from 10 minutes down to 1 minute). But they argued that finer aggregations are not necessary for the applications they considered. \cite{vlahogianni2011temporal} argued that aggregation may result in bias, even for applications that do not necessarily operate on disaggregated data.  They specifically pointed to non-stationarity, long-run memory, and structural change as features that are lost upon aggregation but that can be of extreme importance, particularly in traffic management.  This is especially the case for traffic signal operations and adaptive control.

The collection and archival of disaggregated traffic data is a growing trend in U.S. cities, but it is also emerging in other parts of the world. The first such system appeared in northern California for highway traffic operations \citep{skabardonis1996880,may2003loop,may2004automatic}. Various traffic operations tools were developed using these data \citep{wu2014using}, two notable examples that require disaggregated data are vehicle re-identification \citep{coifman2002vehicle} and vehicle classification \citep{wang2004dynamic}. Similar systems were subsequently deployed for purposes of monitoring the performance of signalized arterials in real-time, in Texas \citep{balke2005development}, Indiana \citep{smaglik2007event}, and Minnesota \citep{liu2008smart}.	These systems have dubbed the disaggregated sensor data \emph{high-resolution data} or \emph{event-based} data.  These datasets consist of on/off sensor states over time, typically recorded on a second-by-second basis (which is the control cadence).  The datasets also include the status of traffic signal heads (green, orange, red) over time and the signal switch times.  It is notable that these data do not typically correspond to raw sensor states; in the case of inductive loop detectors, the raw sensor states are inductance drops that the controller \emph{interprets} as vehicle presence depending on controller sensitivity settings.  In this paper, we will use the nomenclature ``high-resolution data'' to refer to these types of datasets, pertaining to traffic signal systems.  The literature includes numerous examples of applications that utilize high-resolution data for traffic operations applications, including vehicle classification in interrupted traffic \citep{liu2014length}, detection of over-saturated traffic conditions \citep{wu2010identification}, signal timing optimization \citep{hu2013arterial,hu2013managing}, and estimating red-light running frequencies \citep{chen2017estimation}.

Existing traffic data analysis methods are either model-based or data-driven.  Techniques that are geared towards the estimation of traffic densities or speeds from both fixed and mobile sensors are examples of the former \citep{jabari2012stochastic,jabari2012,jabari2013stochastic,seo2017traffic,zheng2018traffic,jabari2019learning,osorio2014capturing,lu2018probabilistic,osorio2019efficient,ramezani2015dynamics, yildirimoglu2013experienced}.  Data-driven techniques are becoming more popular with the increasing availability of traffic data.  For traffic prediction, data-driven methodologies fall in one of two major categories: parametric approaches and non-parametric approaches.  Time series models dominate the category of parametric techniques. For example, auto-regressive integrated moving average (ARIMA) \citep{kwak2019traffic,kim2001weighted, kumar2015short, williams2003modeling, chen2011short,sun2015stochastic} and vector variants (VARIMA) \citep{stock2001vector,kamarianakis2003forecasting,kamarianakis2005space,kamarianakis2012real,ghosh2009multivariate} have been widely used and demonstrated to be successful.  These models assume a parametric linear relationship between the label and a finite number of past states of the label itself.  Another class of methods combine \emph{partial likelihood inference} with  \emph{generalized linear modeling} (GLM) techniques to generalize both the linear relationship between predictors (past states) and the predictions and the Gaussian noise assumption.  An example of this class of models is the generalized auto-regressive moving average model (GARMA)  \citep{benjamin2003generalized}, but we will refer to them more generally as \emph{time series following generalized linear models}. \cite{fokianos2004partial} demonstrated how the basic ideas of exponential distribution families and monotone link functions used in GLMs generalize directly to the case of time series data.  This readily provides an apparatus for generalizing linear time series to binary time series data (in a manner similar to logistic regression).  To the best of our knowledge, these models have not been applied to traffic prediction problems. We will, nonetheless, test them alongside other techniques in our experiments.

Non-parametric methods, on the other hand, do not assume any functional model forms and are typically data-driven. The basic idea behind non-parametric techniques is that they learn a general form from the historical data and use it to predict future data. Non-parametric methods can be divided into two types: non-parametric regression such as support vector regression (SVR) \citep{wu2004travel,hu2014traffic,jeong2013supervised} and artificial neural networks (ANN) \citep{lv2014traffic,ma2015long,mackenzie2018evaluation,kang2017short,ma2017learning,liu2017short}. Compared with time series models in the traffic prediction literature which typically assume that traffic data vary linearly over time, SVR and ANN techniques can capture nonlinear variations in traffic data. The advantage of SVR models is that they can learn representative features by using various kernels. For this reason, SVR has been successfully applied to predict traffic data such as flow \citep{hu2014traffic, jeong2013supervised}, headway \citep{zhang2013gaussian}, and travel time \citep{wu2004travel,dilip2017sparse,jabari2020sparse}.  \cite{jeong2013supervised} further proposed an online version of SVR, which can efficiently update the model when new data is added. Alternatively, ANNs are among the first non-parametric methods that have been applied to traffic prediction, and there is a wealth of literature on the subject, from simple multilayer perceptrons (MLPs) \citep{lv2014traffic} to more complicated architectures such as recurrent neural networks (RNNs) \citep{ma2015long,mackenzie2018evaluation,kang2017short}, convolutional neural networks (CNNs) \citep{ma2017learning,benkraouda2020traffic,thodi2021incorporating,thodi2021learning}, and even combinations of RNNs and CNNs \citep{liu2017short}. 

To our best of knowledge, short-term prediction using high-resolution data remains an open problem.  As discussed above the accuracy of traditional prediction techniques tend to drop as the temporal aggregation levels get finer.  This paper develops a novel non-parametric traffic prediction method using high-resolution data for signalized traffic operations management applications, e.g., \citep{li2019position,li2020backpressure,li2020decentralized,lin2021pay}.  When using high-resolution data, much larger volumes of data are required in order to observe patterns and produce predictions.  One also improves accuracy with frequent updates, that is, updating the predictions in real-time as opposed to utilizing a parametric model that was fitted off-line.  These two ingredients (large volumes of data and online estimation) are key aspects of the proposed techniques.  The main challenge is, thus, computational/algorithmic in nature and the central theme of this paper is addressing the algorithmic aspects of the problem. To address these challenges, we propose a matrix completion formulation and block coordinate descent method. Our contributions are threefold: First, we present a novel formulation of traffic prediction as a matrix completion problem.  Our formulation extends traditional modeling approaches applied in this context in that the nonlinear dependence of the labels on the inputs can be accommodated with ease. The model can be described as being \emph{transductive} in that we leverage statistical information in the testing data (e.g., correlations in the inputs).  This, in turn, allows for a dynamic implementation that can adapt to streaming high-resolution data. We consider (and leverage), for the first time, the natural sparsity and ``bigness'' of high-resolution traffic data, in a way that is intrinsic to our formulation.  Second, we develop a block coordinate descent technique to solve the matrix completion problem and analytically demonstrate  that the algorithm converges to a block coordinate-wise minimizer (or Nash point).  We also demonstrate analytically that the convergence rate is sub-linear, meaning that each iteration (which only involves a series of algebraic manipulations) produces an order of magnitude reduction in the distance from optimality.  This means that only few iterations are required to solve the problem, allowing for real-time implementation.  Third, the performance of the model is further boosted by forming an ensemble of matrix completion problems using datasets from past days.  The merits of the proposed ensemble learning approach are twofold: training errors can be reduced to arbitrarily small numbers while capturing and exploiting trends in the data from past days.  This also offers interpretability that is usually sacrificed with most large-scale machine learning tools (e.g., deep neural nets).  

The remainder of this paper is organized as follows: Sec.~\ref{S:formulation} formulates the traffic prediction problem as a matrix completion problem.  The block-coordinate descent algorithm is presented in Sec.~\ref{S:BCD} along with all convergence results. Sec.~\ref{S:ensemble} develops the ensemble learning extension, presents an analysis of the training error, and brief analyses of generalization and sample complexity followed by the time complexity of the overall approach.  Sec.~\ref{S:experiments} presents three sets of numerical experiments, which include both simulated and real-world data and Sec.~\ref{S:conc} concludes the paper.

\section{Traffic Prediction as a Matrix Completion Problem: Problem Formulation} \label{S:formulation}
\subsection{Notation and Preliminaries}
Let $\mathbf{x}(t)\in \{0,1\}^n$ represent a set of $n$ network detector states at time $t \in \mathbb{Z}_+$; an element of the vector $\mathbf{x}(t)$ is 0 if the corresponding detector is unoccupied at time $t$ and is equal to 1, otherwise.  We consider as input at time step $t$ the present detector state, and previous states up to a lag of size $L$.  The parameter $L$ depends on temporal correlations in the data, and has been comprehensively discussed in previous work, e.g., \citep{broersen1997abc,wang2002autoregressive}. It plays a similar role to that of the \emph{order} of a vector autoregressive model, where the standard approach involves testing different orders and using \emph{information criteria} (IC) (such as the Akaike Information Criterion, or AIC) to select an ``optimal'' order.  The selected order is the one with the lowest IC score.  IC scores generally consist of two terms, one that represents a training error that decreases with increasing model order, and a penalty term that increases with model order.  It is well known that there does not exist an optimal selection rule that is applicable to all problems \citep{broersen1997abc}.  For systems that exhibit periodic patterns, such as ours (due to traffic lights), the Nyquist-Shannon Theorem suggests a sampling rate that is twice the highest frequency in the signal. The periodicity in the sensor states does not necessarily correspond to the cycle lengths of the traffic signals as movements can be active in different traffic signal phases and due to the varying durations of the phases in adaptive traffic control systems.  In our experiments, we test the predictive accuracy of our approach using different values of $L$ and select the lag $L$ with the best prediction performance.

We represent the input at time $t$ by applying a backshift operator up to order $L$, $\BB_L: \RR^n \rightarrow \RR^{nL}$, which is given by
\begin{equation}
	\BB_L\mathbf{x}(t) = \begin{bmatrix}
		\mathbf{x}(t-L+1) \\ \vdots \\ \mathbf{x}(t)
	\end{bmatrix}. \label{eq_backshift}
\end{equation}
The output of the prediction, performed at time $t$, which we denote by $\mathbf{y}(t) \in\RR^n$, is simply the state of the $n$ network detectors at some prediction horizon $H \in \mathbb{Z}_+$ time steps later, that is $\mathbf{y}(t) = \mathbf{x}(t + H)$. Let $T_{\tr}$ denote the number of training samples. 
The inputs in the training sample cover the interval $t \in \{1,\hdots,T_{\tr}\}$, which we denote by $\{\BB_L\mathbf{x}_{\tr}(t), \mathbf{y}_{\tr}(t)\}_{t=1}^{T_{\tr}}$. As in \eqref{eq_backshift} 
\begin{equation}
	\BB_L\mathbf{x}_{\tr}(t) = \begin{bmatrix}
		\mathbf{x}_{\tr}(t-L+1) \\ \vdots \\ \mathbf{x}_{\tr}(t)
	\end{bmatrix}
\end{equation}
is the input associated with the sample at time $t$ and the corresponding output vector is $\mathbf{y}_{\tr}(t) = \mathbf{x}_{\tr}(t+H) \in\RR^n$. The inputs and outputs at a single time step $t$ are illustrated in Fig. \ref{F:predictionInOut}
\begin{figure}[h!]
	\centering
	\scriptsize
	\tikzstyle{format}=[rectangle,draw,thin,fill=white,font=\small,minimum width=.5cm,minimum height=.5cm]
	\tikzstyle{test}=[diamond,aspect=2,draw,thin,text width=4cm]
	\tikzstyle{point}=[coordinate,on grid,]
	\begin{tikzpicture}[node distance=0.35 cm,
		auto,>=latex',
		thin,
		start chain=going below,
		every join/.style={norm},]
		\coordinate(Z) at(0,0);
		\node[format,xshift=5cm](n0){1};
		\node[format,right of=n0,xshift=.2cm](n1){0};
		\node[format,right of=n1,xshift=.2cm](n2){1};
		\node[format,right of =n2,xshift=.2cm](n3){1};
		\node[format,right of =n3,xshift=.2cm](n4){0};
		\node[format,right of =n4,xshift=.2cm](n5){0};
		\node[format,right of =n5,xshift=.2cm](n6){0};
		\node[format,right of =n6,xshift=.2cm](n7){1};
		\node[format,right of =n7,xshift=0.8cm](n8){1};
		\coordinate[label= right:] (A) at($(n8)+(0.1,0)$);
		\node[right of=n8,xshift=1.cm](n9){\small sensor 1};
		\coordinate[label =  {\small Input $\BB_L\mathbf{x}_{\tr}(t)$}](B) at($(n4)+(0, 0.3)$);
		\coordinate[label = {\small Output $\mathbf{y}_{\tr}(t)$}](C) at($(A)+(0, 0.3)$);
		\node[format,below of=n0,yshift=-.2 cm](n10){0};
		\node[format,right of=n10,xshift=.2cm](n11){1};
		\node[format,right of=n11,xshift=.2cm](n12){0};
		\node[format,right of =n12,xshift=.2cm](n13){1};
		\node[format,right of =n13,xshift=.2cm](n14){0};
		\node[format,right of =n14,xshift=.2cm](n15){0};
		\node[format,right of =n15,xshift=.2cm](n16){1};
		\node[format,right of =n16,xshift=.2cm](n17){0};
		\node[format,below of=n8,yshift=-.2 cm](n18){0};
		\coordinate[label= right:] (D) at($(n18)+(0.1,0)$);
		\node[below of=n9,yshift=-.2 cm](n19){\small sensor 2};
		\node[format,below of=n10,yshift=-.2 cm](n20){1};
		\node[format,right of=n20,xshift=.2cm](n21){0};
		\node[format,right of=n21,xshift=.2cm](n22){1};
		\node[format,right of =n22,xshift=.2cm](n23){0};
		\node[format,right of =n23,xshift=.2cm](n24){1};
		\node[format,right of =n24,xshift=.2cm](n25){0};
		\node[format,right of =n25,xshift=.2cm](n26){0};
		\node[format,right of =n26,xshift=.2cm](n27){0};
		\node[format,below of=n18,yshift=-.2 cm](n28){0};
		\coordinate[label= right:] (E) at($(n28)+(0.1,0)$);
		\node[below of=n19,yshift=-.2 cm](n29){\small sensor 3};
		\coordinate[label= below: $\vdots$,yshift=-.1 cm] (F) at($(n24)+(-0.2,0)$);
		\coordinate[label= below: $\vdots$,yshift=-.1 cm] (I) at($(E)+(-0.12,0)$);
		\coordinate[label= below: $\vdots$,yshift=-.1 cm] (I) at($(E)+(1.2,0)$);      
		\node[format,xshift=5cm,yshift=-2.2cm](n30){1};
		\node[format,right of=n30,xshift=.2cm](n31){1};
		\node[format,right of=n31,xshift=.2cm](n32){0};
		\node[format,right of =n32,xshift=.2cm](n33){0};
		\node[format,right of =n33,xshift=.2cm](n34){0};
		\node[format,right of =n34,xshift=.2cm](n35){1};
		\node[format,right of =n35,xshift=.2cm](n36){0};
		\node[format,right of =n36,xshift=.2cm](n37){1};
		\node[format,right of =n37,xshift=.8cm](n38){0};
		\coordinate[label= right: ] (E) at($(n38)+(0.1,0)$);
		\node[below of=n29,yshift=-0.8cm](n39){\small sensor $n$};
		\draw [thick,decorate,decoration={brace,amplitude=10pt,mirror},below of=n34,yshift=-.5cm] (4.6,-1.7) -- (9.2,-1.7) node[black,midway,yshift=-0.6cm] {\small Lag $L$};
	\end{tikzpicture}
	\caption{An Illustration of Training Inputs and Outputs at a Single Time Step}	
	\label{F:predictionInOut}	
\end{figure}
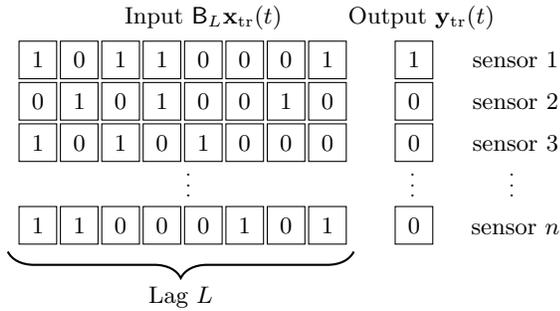

\subsection{Traffic Prediction and Matrix Rank Minimization}
Assuming a linear relationship between the input and output (which will be relaxed below), we may write
\begin{equation}
	\widehat{\mathbf{y}}_{\tr}(t)=\langle \WW,\BB_L\mathbf{x}_{\tr}(t)  \rangle = \WW^{\top}\BB_L\mathbf{x}_{\tr}(t), \label{eq_tr1}
\end{equation}
where $\WW \in \RR^{n \times nL}$ is the (regression) coefficients matrix, $\widehat{\mathbf{y}}_{\tr}(t)$ is an estimate of $\mathbf{y}_{\tr}(t)$, and $\langle \cdot,\cdot \rangle$ is the inner product. Let $T_{\te}$ be the number of testing samples. We denote the set of testing data over the prediction interval $t \in \{T_{\tr}+1,\hdots,T_{\tr} + T_{\te}\}$ by $\{\BB_L \mathbf{x}_{\te}(t), \mathbf{y}_{\te}(t)\}_{t=T_{\tr}+1}^{T_{\tr}+T_{\te}}$ and write their (linear) relationship as
\begin{equation}
	\widehat{\mathbf{y}}_{\te}(t)=\langle \WW,\BB_L\mathbf{x}_{\te}(t) \rangle = \WW^{\top}\BB_L\mathbf{x}_{\te}(t). \label{eq_te1}
\end{equation}

To set the prediction problem up as a matrix completion problem, we first define the input data matrices
\begin{equation}
	\XX_{\tr} \equiv \begin{bmatrix} \BB_L\mathbf{x}_{\tr}(1) & \cdots & \BB_L \mathbf{x}_{\tr}(T_{\tr}) \end{bmatrix} \in \RR^{nL \times T_{\tr}}
\end{equation}
and
\begin{multline}
	\XX_{\te} \equiv \begin{bmatrix} \BB_L\mathbf{x}_{\te}(T_{\tr}+1) & \cdots & \BB_L \mathbf{x}_{\te}(T_{\tr}+T_{\te}) \end{bmatrix} \\ \in \RR^{nL \times T_{\te}}.
\end{multline}
Their corresponding output matrices are defined as
\begin{equation}
	\YY_{\tr} \equiv \begin{bmatrix} \widehat{\mathbf{y}}_{\tr}(1) & \cdots & \widehat{\mathbf{y}}_{\tr}(T_{\tr}) \end{bmatrix} \in \RR^{n \times T_{\tr}}	
\end{equation}
and 
\begin{equation}
	\YY_{\te} \equiv \begin{bmatrix} \widehat{\mathbf{y}}_{\te}(T_{\tr}+1) & \cdots & \widehat{\mathbf{y}}_{\te}(T_{\tr}+T_{\te}) \end{bmatrix} \in \RR^{n \times T_{\te}}.
\end{equation}
We then define the \emph{joint matrix}, which concatenates both the training and the testing data, as
\begin{equation}
	\ZZ \equiv  \begin{bmatrix}
		\YY_{\tr}&\YY_{\te} \\ 
		\XX_{\tr}&\XX_{\te} 
	\end{bmatrix}\in\RR^{(n+nL)\times (T_{\tr}+T_{\te})}. \label{eq_Z1}
\end{equation}
The entries corresponding to $\YY_{\te}$ in $\ZZ$ are unknown and constitute the solution of the prediction problem.  Hence the prediction problem can be cast as a matrix completion problem.  
From \eqref{eq_tr1} and \eqref{eq_te1}, we have that
\begin{equation}
	\begin{bmatrix}
		\YY_{\tr}&\YY_{\te} 
	\end{bmatrix}=\langle \WW, \begin{bmatrix}
		\XX_{\tr}&\XX_{\te}
	\end{bmatrix} \rangle, \label{eq_rel1}
\end{equation}
which implies that the joint matrix $\ZZ$ as defined in \eqref{eq_Z1} has low rank as a result of linear dependencies in the rows of $\ZZ$ implied by \eqref{eq_rel1}.  Thus, removing entries corresponding to $\YY_{\te}$ and completion of the joint matrix $\ZZ$ via rank-minimization techniques is equivalent to directly producing predictions (and bypassing the need to estimate $\WW$). The higher the linear dependence in the matrix $\ZZ$ (represented by its rank), the more parsimonious the resulting prediction, which is a desirable feature from a statistical point of view.

\subsection{Relaxing the Linearity Using Kernels}
We relax the linearity assumption by mapping the inputs to a higher dimensional feature space, via kernel ``basis'' functions, and solve a linear prediction problem in the higher dimensional space \citep{suykens1999least}. In the training stage, the nonlinear relationship can be expressed as
\begin{equation}
	\widehat{\mathbf{y}}_{\tr}(t)=\big\langle \WW,\phi\big(\BB_L\mathbf{x}_{\tr}(t)\big)  \big\rangle = \WW^{\top}\phi\big(\BB_L\mathbf{x}_{\tr}(t)\big), \label{eq_tr2}
\end{equation}
where $\phi: \RR^{nL} \rightarrow \RR^h$ is the nonlinear function that maps the input space to the high-dimensional feature space ($h > nL$) and with slight notation abuse $\WW \in \RR^{n \times h}$ is the regression matrix. Note that the inner product of $\WW$ and $\phi(\mathbf{x})$ can be learned using the \emph{kernel trick} avoiding the need to compute the map $\phi$; see for example \citep{smola2004tutorial,li2017linearity}.  Mapping functions are related to kernels via 
\begin{equation}
	\KK(\mathbf{x}_1, \mathbf{x}_2) = \langle \phi(\mathbf{x}_1), \phi(\mathbf{x}_2) \rangle = \phi(\mathbf{x}_1)^{\top} \phi(\mathbf{x}_2),
\end{equation}
where $\KK$ is a kernel (a weighted distance measure).  We note that not all known kernel functions have known mapping functions (e.g., the widely used Gaussian kernels).  We shall prescribe a kernel and treat the mapping function as an unknown.  Specifically, we shall adopt a radial basis function with periodical patterns (RBFP), which appends a component that captures potential periodic patterns to traditional radial basis functions \citep{lippi2013short}:
\begin{multline}
	\KK_P(\mathbf{x}_1(t_1),\mathbf{x}_2(t_2)) \\= \exp\big(-\gamma\|\BB_L\mathbf{x}_1(t_1)-\BB_L\mathbf{x}_2(t_2)\|_2^2 - \gamma_{\mathrm{p}} d_P(t_1,t_2)^2\big),
	\label{eq_rbfp}
\end{multline}
where $P$ is the period in the data (e.g., one signal cycle), $\gamma$ and $\gamma_{\mathrm{p}}$ are weights associated with the radial basis function and the periodicity, respectively, and
\begin{equation}
	d_P(t_1,t_2) \equiv \min \big\{ |(t_1 - t_2) ~ \mathsf{mod} ~P|, P - |(t_1 - t_2) ~ \mathsf{mod} ~ P| \big\}
\end{equation}
is a temporal distance (modulo periodicity). The testing samples are related, similar to the training samples, as follows:
\begin{equation}
	\widehat{\mathbf{y}}_{\te}(t)=\big\langle \WW,\phi\big(\BB_L\mathbf{x}_{\te}(t) \big)  \big\rangle = \WW^{\top}\phi\big(\BB_L\mathbf{x}_{\te}(t) \big), \label{eq_te2}
\end{equation}
and the \emph{joint matrix} in this kernelized setting is defined as:
\begin{equation}
	\ZZ\equiv  \begin{bmatrix}
		\YY_{\tr}&\YY_{\te} \\ 
		\Phi(\XX_{\tr})&\Phi(\XX_{\te}) 
	\end{bmatrix}\in\RR^{(n+h)\times (T_{\tr}+T_{\te})},
	\label{eq_Z2}
\end{equation}
where $\Phi: \RR^{nL \times t} \rightarrow \RR^{h \times t}$ simply applies $\phi$ to each column of its argument.  That is
\begin{equation}
	\Phi(\XX_{\tr}) \equiv \begin{bmatrix}
		\phi\big(\BB_L \mathbf{x}_{\tr}(1)\big) & \cdots & \phi\big(\BB_L \mathbf{x}_{\tr}(T_{\tr})\big)
	\end{bmatrix}
\end{equation}
and $\Phi(\XX_{\te})$ is defined in a similar way. To further simplify notation, we define the matrices $\Phi_{\tr} \equiv \Phi(\XX_{\tr}) \in \RR^{h \times T_{\tr}}$ and $\Phi_{\te} \equiv \Phi(\XX_{\te})  \in \RR^{h \times T_{\te}}$.  The inner products of these matrices define the kernels that we will use below.  For example we define the kernel $\mathbf{K}_{\tr,\te}$ as:
\begin{widetext}
\begin{equation}
	\mathbf{K}_{\tr,\te} \equiv \langle \Phi(\XX_{\tr}), \Phi(\XX_{\te}) \rangle 
	= \begin{bmatrix} \KK_P\big(\mathbf{x}_{\tr}(1),\mathbf{x}_{\te}(T_{\tr}) \big) & \cdots & \KK_P\big(\mathbf{x}_{\tr}(1),\mathbf{x}_{\te}(T_{\tr}+T_{\te}) \big) \\ \vdots & \ddots & \vdots \\ \KK_P\big(\mathbf{x}_{\tr}(T_{\tr}),\mathbf{x}_{\te}(T_{\tr}) \big) & \cdots & \KK_P\big(\mathbf{x}_{\tr}(T_{\tr}),\mathbf{x}_{\te}(T_{\tr}+T_{\te}) \big) \end{bmatrix}. 
	\label{eq_kernel}
\end{equation}
\end{widetext}

From \eqref{eq_tr2} and \eqref{eq_te2}, we have that
\begin{equation}
	\begin{bmatrix}
		\YY_{\tr}&\YY_{\te} 
	\end{bmatrix}= \langle \WW, \begin{bmatrix}
		\phi(\XX_{\tr})&\phi(\XX_{\te})
	\end{bmatrix} \rangle, \label{eq_rel2}
\end{equation}
which (again) implies that $\ZZ$ is a low-rank matrix and that the prediction problem can be formulated as a low-rank matrix completion problem.

\subsection{The Matrix Completion Problem}
Our matrix completion problem seeks to find an estimate of the matrix $\ZZ$ in \eqref{eq_Z2}, denoted $\widehat{\ZZ}$.  We seek to match $\ZZ$ on the kernelized training and testing inputs $\Phi(\XX_{\tr})$ and $\Phi(\XX_{\te})$, and the training outputs $\YY_{\te}$.  There are many ways to accomplish this; we will invoke the \emph{principle of parsimony}.  This results in a simple solution, one that is easy to interpret and extract insights from.  To this end, our matrix completion problem seeks to find a low-rank approximation $\widehat{\ZZ}$ of the matrix $\ZZ$.  Let $\PP_{\Omega}: \RR^{p \times q} \rightarrow \RR^{p \times q}$ be a binary mask operator:
\begin{equation}
	\big[\PP_{\Omega}(\MM) \big]_{ij} = \mathbbm{1}{\{(i,j) \in \Omega\}} \MM_{ij},
	\label{eq_BinMask}
\end{equation}
where $\mathbbm{1}{\{(i,j) \in \Omega\}}$ is the indicator function, taking the value 1 if the condition $(i,j) \in \Omega$ is true and 0 otherwise. 
We utilize the binary mask to exclude the output testing matrix from $\ZZ$.  That is, we define $\Omega$ as the set of indices corresponding to non-testing outputs, that is
\begin{multline}
	\Omega \equiv \{(i,j): 1 \le i \le n+h, 1 \le j \le T_{\tr}\} \\ \cup \{(i,j) : n+1 \le i \le n+h, T_{\tr}+1 \le j \le T_{\tr}+T_{\te}\}. \label{eq_Omega}
\end{multline}
Consequently,
\begin{equation}
	\PP_{\Omega}(\ZZ) = \begin{bmatrix}
		\YY_{\tr}&\mathbf{0} \\ 
		\Phi(\XX_{\tr})&\Phi(\XX_{\te}) 
	\end{bmatrix}.
	\label{eq_P_omega}
\end{equation}
The matrix completion problem is then formulated as a rank minimization problem:
\begin{equation}
	\widehat{\ZZ} \equiv \underset{\MM \in\RR^{(n+h)\times (T_{\tr}+T_{\te})}}{\arg \min} \big\{ \rank(\MM):  \PP_{\Omega}(\MM - \ZZ)=\mathbf{0} \big\}.
	\label{eq_RMP}
\end{equation}
The solution $\widehat{\ZZ}$ is interpreted as the lowest-rank matrix that matches $\ZZ$ exactly (via the constraint) in the entries corresponding to the training data, $\YY_{\tr}$ and $\Phi(\XX_{\tr})$, and the input testing data, $\Phi(\XX_{\te})$.  
Rank minimization \eqref{eq_RMP} is generally NP-hard but it was demonstrated in \citep{candes2009exact} that \eqref{eq_RMP} can be solved exactly if certain sparsity conditions on $\ZZ$ hold by using convex optimization techniques (via convex relaxation of the rank objective as shown below).  The sparsity condition required is that $\ZZ$ be \emph{sufficiently incoherent}, which is achieved when the number of observed entries in the matrix, $|\Omega|$ is $O(N^{1.2} \rank(\ZZ) \log N)$, where $N \equiv \max \{n+h, T_{\tr} + T_{\te}\}$ in our context.  This ensures faithful reconstruction.  In our context, increasing the number of observations $|\Omega|$ can be accomplished by using longer lags and more training/testing data or more sensors.  A more useful application of this bound is to produce upper bound estimates of the rank of $\widehat{\ZZ}$, which will be needed to set up the problem \eqref{eq_RMCP_SVD_0} below. Henceforth, we will focus on solving a convex relaxation of the problem and refer readers to \citep{candes2009exact,candes2009power} for further information on these bounds.

The main difficulty lies in minimizing the matrix rank.  This can be observed by noting that $\rank(\MM) = \|\sigma(\MM)\|_0$, where $\sigma(\MM)$ is a vector of the singular values of $\MM$ and $\|\cdot\|_0$ is the $\ell_0$ pseudo-norm which counts the number of non-zero elements of its argument.  We relax the objective function via a convex surrogate using the nuclear norm $\|\MM\|_* = \trace(\sqrt{\MM^{\top} \MM}) = \|\sigma(\MM)\|_1$.  (The $\ell_1$ norm in this context is simply the sum of the singular values, since the singular values of a matrix are always non-negative.)  This yields the relaxed problem:
\begin{equation}
	\widehat{\ZZ} \equiv \underset{\MM \in\RR^{(n+h)\times (T_{\tr}+T_{\te})}}{\arg \min} \big\{ \|\MM\|_*:  \PP_{\Omega}(\MM - \ZZ) = \mathbf{0} \big\},
	\label{eq_NNMP}
\end{equation}
which is equivalent to the quadratic optimization problem \cite[Lemma 5.1]{recht2010guaranteed}
\begin{multline}
	\{\widehat{\UU},\widehat{\VV}\} \\ \equiv \underset{\substack{\UU\in\RR^{(n+h)\times r}, \\ \VV\in\RR^{(T_{\tr}+T_{\te})\times r}}}{\arg \min} \big\{ \|\UU\|_{\FF}^2 + \|\VV\|_{\FF}^2: \PP_{\Omega}(\UU\VV^{\top} - \ZZ) = \mathbf{0}  \big\},
	\label{eq_RMCP_SVD_0}
\end{multline}
where $\|\cdot\|_{\FF}$ is the Frobenius norm and $r$ is chosen so that $r \ge \rank(\ZZ)$.  The latter formulation utilizes a bi-linear representation $\MM = \UU\VV^\top$ of the joint matrix and can be solved efficiently using alternating minimization techniques \citep{jain2013low}, an instance of which we shall develop in the next section. The latter formulation \eqref{eq_RMCP_SVD_0} can be expressed as an unconstrained optimization problem via the Lagrangian:
\begin{multline}
	\{\widehat{\UU},\widehat{\VV}\} \\ \equiv \underset{\substack{\UU\in\RR^{(n+h)\times r}, \\ \VV\in\RR^{(T_{\tr}+T_{\te})\times r}}}{\arg \min}  \| \PP_{\Omega}(\UU\VV^{\top} - \ZZ) \|_{\FF}^2  + \mu (\|\UU\|_{\FF}^2 + \|\VV\|_{\FF}^2),
	\label{eq_RMCP_SVD}
\end{multline}
where $\mu > 0$ is (the inverse of) a Lagrange multiplier. Here, we treat $\mu$ as a penalty parameter with the aim of mitigating over-fitting (similar to ridge regression).  There are many techniques that can be employed for the selection of $\mu$; see, e.g., \citep{van2015lecture}.

Another simplification can be achieved by dividing each of the two matrices $\UU$ and $\VV$ into two blocks, a training block and a testing block:
\begin{equation}
	\UU=\begin{bmatrix}
		\UU_{\tr}\\ \UU_{\te}
	\end{bmatrix} \mbox{ and } \VV=\begin{bmatrix}
		\VV_{\tr}\\ \VV_{\te}
	\end{bmatrix},
\end{equation}
where $\UU_{\tr}\in\RR^{n\times r}$, $\VV_{\tr}\in\RR^{T_{\tr}\times r}$, $\UU_{\te}\in\RR^{h\times r}$, and $\VV_{\te}\in\RR^{T_{\te}\times r}$
This decomposition allows us to bypass use of the binary mask $\PP_{\Omega}$, simplifying the formulation further.  We obtain the following optimization problem:
\begin{multline}
	\{\widehat{\UU}_{\tr}, \widehat{\UU}_{\te}, \widehat{\VV}_{\tr}, \widehat{\VV}_{\te}\}  \equiv \underset{\UU_{\tr},\UU_{\te},\VV_{\tr},\VV_{\te}}{\arg \min} \|\UU_{\tr} \VV_{\tr}^{\top} - \YY_{\tr}\|^2_{\FF} \\
	+\|\UU_{\te} \VV_{\tr}^{\top} - \Phi_{\tr} \|^2_{\FF}  + \| \UU_{\te} \VV_{\te}^{\top} -  \Phi_{\te}\|^2_{\FF} \\
	+ \mu(\|\UU_{\tr}\|^2_{\FF} +  \|\UU_{\te}\|^2_{\FF} + \|\VV_{\tr}\|^2_{\FF} + \|\VV_{\te}\|^2_{\FF}).
	\label{eq_KMCP}
\end{multline}
Again, we do not prescribe the mapping functions, but shall prescribe kernels.  For this purpose, we define $\mathbf{K}_{\tr,\tr} \equiv \langle \Phi_{\tr}, \Phi_{\tr} \rangle = \Phi_{\tr}^{\top} \Phi_{\tr}$, $\mathbf{K}_{\tr,\te} \equiv \langle \Phi_{\tr}, \Phi_{\te} \rangle = \Phi_{\tr}^{\top} \Phi_{\te}$, $\mathbf{K}_{\te,\tr} \equiv \langle \Phi_{\te}, \Phi_{\tr} \rangle = \Phi_{\te}^{\top} \Phi_{\tr}$, and $\mathbf{K}_{\te,\te} \equiv \langle \Phi_{\te}, \Phi_{\te} \rangle = \Phi_{\te}^{\top} \Phi_{\te}$.  

To summarize, the key features of the proposed formulation \eqref{eq_KMCP}, as they relate to traffic prediction, are:
\begin{enumerate}
	\item The formulation inherits all the merits of conventional \textit{vector} time series prediction methods, namely, vector auto-regressive models (VARs) \citep{stock2001vector,ghosh2009multivariate}.  Specifically, our approach is capable of capturing and leveraging spatio-temporal dependencies. 
	\item The use of kernels allows us to capture non-linearities (not possible with the VAR models in the literature), which one expects to see in traffic flow patterns, particularly in high-resolution traffic data.
	\item The low rank approximation allows for insights to be easily extracted from the solution. The predictions produced are $\widehat{\YY}_{\te} = \widehat{\UU}_{\tr}\widehat{\VV}_{\te}^{\top}$, where the $n$ rows of $\widehat{\UU}_{\tr}$ capture the relationship between sensors (space) and $r$ latent features, while the $T_{\te}$ rows of $\widehat{\VV}_{\te}$ capture the relationship between the prediction time steps and the $r$ latent features. One can mine these matrices to find spatio-temporal correlations between different sensor and different times steps; see, e.g., \citep{brun2014can}.
\end{enumerate}

Mathematically, the matrix completion problem \eqref{eq_KMCP} is not (in general) convex in all of its block coordinates $\{\UU_{\tr}, \UU_{\te}, \VV_{\tr}, \VV_{\te}\}$ simultaneously but it is convex in each of the blocks separately. (For intuition, readers may consider --as an exercise-- convexity of the function $f(x,y) = (xy-1)^2$).  Such problems are referred to as \emph{block multi-convex} \citep{xu2013block}.  Let  $F(\UU_{\tr},\UU_{\te},\VV_{\tr},\VV_{\te})$ denote the objective function in \eqref{eq_KMCP}:
\begin{multline}
	F(\UU_{\tr},\UU_{\te},\VV_{\tr},\VV_{\te}) \\ \equiv \|\UU_{\tr} \VV_{\tr}^{\top} - \YY_{\tr}\|^2_{\FF} +\|\UU_{\te} \VV_{\tr}^{\top} - \Phi_{\tr} \|^2_{\FF}  + \| \UU_{\te} \VV_{\te}^{\top} -  \Phi_{\te}\|^2_{\FF}
	\\ + \mu(\|\UU_{\tr}\|^2_{\FF} +  \|\UU_{\te}\|^2_{\FF} + \|\VV_{\tr}\|^2_{\FF} + \|\VV_{\te}\|^2_{\FF}).
	\label{eq_objF}
\end{multline} 
We seek a solution $\{\overline{\UU}_{\tr}, \overline{\UU}_{\te}, \overline{\VV}_{\tr}, \overline{\VV}_{\te}\}$ for which the following variational inequalities hold for all $\UU_{\tr} \in \RR^{n \times r}$, all $\UU_{\te} \in \RR^{h \times r}$, all $\VV_{\tr} \in \RR^{T_{\tr} \times r}$, and all $\VV_{\te} \in \RR^{T_{\te} \times r}$:
\begin{equation}
	\big\langle \partial_{\UU_{\tr}}F(\overline{\UU}_{\tr}, \overline{\UU}_{\te},\overline{\VV}_{\tr},\overline{\VV}_{\te}), \UU_{\tr} - \overline{\UU}_{\tr} \big\rangle \ge 0, \label{eq_VI1}
\end{equation}
\begin{equation}
	\big\langle \partial_{\UU_{\te}}F(\overline{\UU}_{\tr}, \overline{\UU}_{\te},\overline{\VV}_{\tr},\overline{\VV}_{\te}), \UU_{\te} - \overline{\UU}_{\te} \big\rangle \ge 0, \label{eq_VI2}
\end{equation}
\begin{equation}
	\big\langle \partial_{\VV_{\tr}}F(\overline{\UU}_{\tr}, \overline{\UU}_{\te},\overline{\VV}_{\tr},\overline{\VV}_{\te}), \VV_{\tr} - \overline{\VV}_{\tr} \big\rangle \ge 0, \label{eq_VI3}
\end{equation}
and
\begin{equation}
	\big\langle \partial_{\VV_{\te}}F(\overline{\UU}_{\tr}, \overline{\UU}_{\te},\overline{\VV}_{\tr},\overline{\VV}_{\te}), \VV_{\te} - \overline{\VV}_{\te} \big\rangle \ge 0, \label{eq_VI4}
\end{equation}
where $\partial_{\UU}$ denotes the partial derivative operator with respect to the matrix $\UU$.  This type of solution is referred to as a \emph{block coordinate-wise minimizer} or a \emph{Nash point} in the sense that one cannot further minimize $F$ by changing any of the four block coordinates separately.

\section{Block Coordinate Descent Algorithm} \label{S:BCD}
\subsection{Block Coordinate Descent and Thresholding}
The proposed block coordinate descent algorithm first updates the two $\UU$ blocks, $\UU_{\tr}$ and $\UU_{\te}$, and then updates the two $\VV$ blocks, $\VV_{\tr}$ and $\VV_{\te}$.  The objective functions of the four sub-problems in iteration $k$ are stated as
\begin{equation}
	F_1^{[k]}(\UU_{\tr}) \equiv \|\UU_{\tr} \widehat{\VV}_{\tr}^{[k-1]\top} - \YY_{\tr}\|^2_{\FF} + 2\mu \| \UU_{\tr} \|_{\FF}^2,
	\label{eq_objF1}
\end{equation}
\begin{multline}
	F_2^{[k]}(\UU_{\te}) \equiv \|\UU_{\te} \widehat{\VV}_{\tr}^{[k-1]\top} - \Phi_{\tr}\|^2_{\FF} + \|\UU_{\te} \widehat{\VV}_{\te}^{[k-1]\top} - \Phi_{\te}\|^2_{\FF} \\ + 2\mu \| \UU_{\te} \|_{\FF}^2,
	\label{eq_objF2}
\end{multline}
\begin{multline}
	F_3^{[k]}(\VV_{\tr}) \equiv \|\widehat{\UU}_{\tr}^{[k]} \VV_{\tr}^{\top} - \YY_{\tr}\|^2_{\FF} + \|\widehat{\UU}_{\te}^{[k]} \VV_{\tr}^{\top} - \Phi_{\tr}\|^2_{\FF} \\+ 2\mu \| \VV_{\tr} \|_{\FF}^2,
	\label{eq_objF3}
\end{multline}
and
\begin{equation}
	F_4^{[k]}(\VV_{\te}) \equiv \|\widehat{\UU}_{\te}^{[k]} \VV_{\te}^{\top} - \Phi_{\te}\|^2_{\FF} + 2\mu \| \VV_{\te} \|_{\FF}^2.
	\label{eq_objF4}
\end{equation}
The factors of two in the regularizers are unnecessary but we use them here to reduce clutter later on.  The updates are obtained in closed form; the updated $\UU$ blocks are given by:
\begin{multline}
	\widehat{\UU}_{\tr}^{[k]} = \underset{ \UU_{\tr} \in \RR^{n \times r}}{\arg \min} ~ F_1^{[k]}(\UU_{\tr}) \\= \YY_{\tr}\widehat{\VV}_{\tr}^{[k-1]} \big( \widehat{\VV}_{\tr}^{[k-1]\top} \widehat{\VV}_{\tr}^{[k-1]} +2\mu\mathbf{I} \big)^{-1} \label{eq_updateUtr}
\end{multline}
and 
\begin{multline}
	\widehat{\UU}_{\te}^{[k]} = \underset{ \UU_{\te} \in \RR^{h \times r} }{\arg \min} ~ F_2^{[k]}(\UU_{\te}) 
	 = (\Phi_{\tr} \widehat{\VV}_{\tr}^{[k-1]} + \Phi_{\te} \widehat{\VV}_{\te}^{[k-1]}) \\
	\cdot (\widehat{\VV}_{\tr}^{[k-1]\top} \widehat{\VV}_{\tr}^{[k-1]} + \widehat{\VV}_{\te}^{[k-1]\top} \widehat{\VV}_{\te}^{[k-1]} + 2\mu\mathbf{I})^{-1}. \label{eq_updateUte}
\end{multline}
The updated $\VV$ blocks are then given by:
\begin{multline}
	\widehat{\VV}_{\tr}^{[k]} = \underset{ \VV_{\tr} \in \RR^{T_{\tr} \times r}}{\arg \min} ~ F_3^{[k]}(\VV_{\tr}) = (\YY_{\tr}^{\top} \widehat{\UU}_{\tr}^{[k]} + \Phi_{\tr}^{\top} \widehat{\UU}_{\te}^{[k]}) \\ \cdot (\widehat{\UU}_{\te}^{[k]\top} \widehat{\UU}_{\te}^{[k]} + \widehat{\UU}_{\tr}^{[k]\top} \widehat{\UU}_{\tr}^{[k]} + 2\mu\mathbf{I})^{-1}
	\label{eq_updateVtr}
\end{multline}
and
\begin{equation}
	\widehat{\VV}_{\te}^{[k]} = \underset{ \VV_{\te} \in \RR^{T_{\te} \times r}}{\arg \min} ~ F_4^{[k]}(\VV_{\te}) = \Phi_{\te}^{\top} \widehat{\UU}_{\te}^{[k]} (\widehat{\UU}_{\te}^{[k]\top} \widehat{\UU}_{\te}^{[k]} + 2\mu\mathbf{I})^{-1}.
	\label{eq_updateVte}
\end{equation}
Since $\Phi_{\tr}$ and $\Phi_{\te}$ are unknown, $\widehat{\UU}_{\te}^{[k]}$ cannot be calculated explicitly.  However, this calculation is not required to produce a prediction: $\widehat{\YY}_{\te} \equiv \widehat{\UU}_{\tr} \widehat{\VV}_{\te}^{\top}$.  To produce this estimate, in each iteration $k$ we need to be able to calculate $\widehat{\UU}_{\tr}^{[k]}$ and $\widehat{\VV}_{\te}^{[k]\top}$.  To calculate these quantities, we need (i) $\widehat{\VV}_{\tr}^{[k-1]}$, (ii) $\Phi_{\te}^{\top}\widehat{\UU}_{\te}^{[k]}$, and (iii) $\widehat{\UU}_{\te}^{[k]\top} \widehat{\UU}_{\te}^{[k]}$.  We start with (iii): assuming that all required calculations have been performed for iteration $k-1$, we have from \eqref{eq_updateUte} that
\begin{multline}
	\widehat{\UU}_{\te}^{[k]\top} \widehat{\UU}_{\te}^{[k]}
	= (\widehat{\VV}_{\tr}^{[k-1]\top} \widehat{\VV}_{\tr}^{[k-1]} + \widehat{\VV}_{\te}^{[k-1]\top} \widehat{\VV}_{\te}^{[k-1]} + \mu\mathbf{I})^{-1} \\ \cdot \Big(\widehat{\VV}_{\tr}^{[k-1]\top} \mathbf{K}_{\tr,\tr} \widehat{\VV}_{\tr}^{[k-1]} + \widehat{\VV}_{\te}^{[k-1]\top} \mathbf{K}_{\te,\tr} \widehat{\VV}_{\tr}^{[k-1]} \\
	+  \widehat{\VV}_{\tr}^{[k-1]\top} \mathbf{K}_{\tr,\te} \widehat{\VV}_{\te}^{[k-1]} + \widehat{\VV}_{\te}^{[k-1]\top} \mathbf{K}_{\te,\te} \widehat{\VV}_{\te}^{[k-1]} \Big) \\ \cdot (\widehat{\VV}_{\tr}^{[k-1]\top} \widehat{\VV}_{\tr}^{[k-1]} + \widehat{\VV}_{\te}^{[k-1]\top} \widehat{\VV}_{\te}^{[k-1]} + \mu\mathbf{I})^{-1}, \label{eq_UteUte}
\end{multline}
which only involves known quantities.  Similarly, for (ii) we have from \eqref{eq_updateUte} that
\begin{multline}
	\Phi_{\te}^{\top} \widehat{\UU}_{\te}^{[k]} =
	(\mathbf{K}_{\te,\tr} \widehat{\VV}_{\tr}^{[k-1]} + \mathbf{K}_{\te,\te} \widehat{\VV}_{\te}^{[k-1]}) \\ \cdot (\widehat{\VV}_{\tr}^{[k-1]\top} \widehat{\VV}_{\tr}^{[k-1]} + \widehat{\VV}_{\te}^{[k-1]\top} \widehat{\VV}_{\te}^{[k-1]} + \mu\mathbf{I})^{-1}, \label{eq_phiU}
\end{multline}
which also only involves known quantities.  For (i), we need to be able to calculate $\Phi_{\tr}^{\top}\widehat{\UU}_{\te}^{[k-1]}$ and $\widehat{\UU}_{\te}^{[k-1]\top} \widehat{\UU}_{\te}^{[k-1]}$. These quantities also only involve known quantities; they are the same as \eqref{eq_phiU} and \eqref{eq_UteUte}, respectively (since $k$ is arbitrary).  Note that this also allows for calculating training estimates $\widehat{\YY}_{\tr}^{[k]} \equiv \widehat{\UU}_{\tr}^{[k]} \widehat{\VV}_{\tr}^{[k]\top}$ in each iteration, which we will need in our ensemble approach presented below.  The steps involved in performing a single update are summarized in Alg. \ref{A:0}.

\begin{algorithm}[h!]
	\small
	\caption{Block Coordinate Descent}
	\label{A:0}
	\KwData{$\YY_{\tr}$, $\mathbf{K}_{\tr,\tr}$, $\mathbf{K}_{\tr,\te}$, $\mathbf{K}_{\te,\tr}$, $\mathbf{K}_{\te,\te}$, $\widehat{\UU}_{\tr}^{[0]}$, $\widehat{\VV}_{\tr}^{[0]}$, $\widehat{\VV}_{\te}^{[0]}$}
	
	\KwResult{$\widehat{\YY}_{\tr}$, $\widehat{\YY}_{\te}$}
	
	\textbf{Initialize}: $k \mapsfrom 1$ 
	
	\While{stopping criterion not met}{
		
		$\widehat{\UU}_{\tr}^{[k]} \mapsfrom \YY_{\tr}\widehat{\VV}_{\tr}^{[k-1]} \big( \widehat{\VV}_{\tr}^{[k-1]\top} \widehat{\VV}_{\tr}^{[k-1]} +2\mu\mathbf{I} \big)^{-1}$ 
		
		$\mathbf{C}_1 \mapsfrom (\widehat{\VV}_{\tr}^{[k-1]\top} \widehat{\VV}_{\tr}^{[k-1]} + \widehat{\VV}_{\te}^{[k-1]\top} \widehat{\VV}_{\te}^{[k-1]} + 2\mu\mathbf{I})^{-1}$
		
		$\mathbf{C}_2 \mapsfrom \widehat{\VV}_{\tr}^{[k-1]\top} \mathbf{K}_{\tr,\tr} \widehat{\VV}_{\tr}^{[k-1]} + \widehat{\VV}_{\te}^{[k-1]\top} \mathbf{K}_{\te,\tr} \widehat{\VV}_{\tr}^{[k-1]}$
		
		$\mathbf{C}_3 \mapsfrom \widehat{\VV}_{\tr}^{[k-1]\top} \mathbf{K}_{\tr,\te} \widehat{\VV}_{\te}^{[k-1]} + \widehat{\VV}_{\te}^{[k-1]\top} \mathbf{K}_{\te,\te} \widehat{\VV}_{\te}^{[k-1]}$
		
		$\Phi_{\te}^{\top} \widehat{\UU}_{\te}^{[k]} \mapsfrom
		(\mathbf{K}_{\te,\tr} \widehat{\VV}_{\tr}^{[k-1]} + \mathbf{K}_{\te,\te} \widehat{\VV}_{\te}^{[k-1]}) \mathbf{C}_1$
		
		$\widehat{\UU}_{\te}^{[k]\top} \widehat{\UU}_{\te}^{[k]} \mapsfrom \mathbf{C}_1 ( \mathbf{C}_2 + \mathbf{C}_3) \mathbf{C}_1$ 
		
		$\mathbf{C}_4 \mapsfrom (\widehat{\UU}_{\te}^{[k]\top} \widehat{\UU}_{\te}^{[k]} + \widehat{\UU}_{\tr}^{[k]\top} \widehat{\UU}_{\tr}^{[k]} + 2\mu\mathbf{I})^{-1}$
		
		$\widehat{\VV}_{\tr}^{[k]} \mapsfrom (\YY_{\tr}^{\top} \widehat{\UU}_{\tr}^{[k]} + \Phi_{\tr}^{\top} \widehat{\UU}_{\te}^{[k]}) \mathbf{C}_4$ 
		
		$\widehat{\VV}_{\te}^{[k]} \mapsfrom \Phi_{\te}^{\top} \widehat{\UU}_{\te}^{[k]} (\widehat{\UU}_{\te}^{[k]\top} \widehat{\UU}_{\te}^{[k]} + 2\mu\mathbf{I})^{-1}$ 
		
		$k \mapsfrom k+1$ 
	}
	$\widehat{\YY}_{\tr} \mapsfrom \widehat{\UU}_{\tr}^{[k]} \widehat{\VV}_{\tr}^{[k]\top}$ and  $\widehat{\YY}_{\te} \mapsfrom \widehat{\UU}_{\tr}^{[k]} \widehat{\VV}_{\te}^{[k]\top}$
\end{algorithm}

\textit{Soft thresholding}:
The predictions $\widehat{\YY}$ produced by Algorithm~\ref{A:0} will produce values that are not necessarily restricted to $\{0,1\}$. For traffic signal operations, whether a vehicle is present or not needs to be decided before the control decisions can be made.  For this reason, we employ the thresholding procedure given in Algorithm~\ref{A:2}. Since our approach is data-driven, it can be applied to prediction problems involving continuous data; in such cases, thresholding is not required.  Our thresholding procedure can simply be described as one that produces cut-offs for each of the $n$ network sensors separately. This resembles the operation performed by intersection controllers, which translate inductance drops at the sensors to on-off signals and this is tuned for each of the sensors separately.  Our thresholding algorithm chooses cut-offs, denoted $\{\tau_j\}_{j=1}^n$ that result in the lowest training error. We note that thresholding as a post-processing step is common in classification and prediction problems in neural networks \citep{maas2013rectifier} and logistic regression based methods \citep{harrell2015regression,goldberg2010transduction} to project non-binary solutions to binary values. 

\begin{algorithm}[h!]
	\small
	\caption{Threshold Learning}
	\label{A:2}
	\KwData{Predicted output matrices $\widehat{\YY}_{\tr}$ and $\widehat{\YY}_{\te}$, true training output $\YY_{\tr}$}
	\KwResult{Set of thresholds, one per row $\{\tau_j\}_{j \ge 1}$}
	\For{each row $j$ of $\widehat{\YY}_{\tr}$}{
		Store row $j$ of $\widehat{\YY}_{\tr}$ in a separate vector $\mathbf{y} \mapsfrom \widehat{\YY}_{j,\tr}$ 
		
		Sort $\mathbf{y}$ in ascending order $\mathbf{y} \mapsfrom \mathsf{sort}(\mathbf{y})$ 
		
		$c_j \mapsfrom \| \mathbf{y} \|_0$ \hspace{.1in} \tcp{number of non-zero elements in $\mathbf{y}$} 
		
		\For{$i \ge c_j$}{
			
			$\tau_{j,i} \mapsfrom \mathbf{y}_i$
			
			$\widehat{\YY}_{j,m,\tr} \mapsfrom \mathbbm{1}\{\widehat{\YY}_{j,m,\tr} \ge \tau_{j,i}\}$ for all $m \ge 1$
			
			Calculate the error $e_{j,i} \mapsfrom \| \widehat{\YY}_{\tr} - \YY_{\tr} \|_0$
		}
		Set $\tau_j \mapsfrom \tau_{j,i^*}$, where $i^* \mapsfrom \arg \min_{i \ge c_j} e_{j,i}$
	}
\end{algorithm}

\subsection{Sublinear Convergence to a Block Coordinate-Wise Minimizer}
In this section, we demonstrate that the block coordinate descent algorithm above converges to a coordinate-wise minimizer (a Nash point). The convergence follows from the strong convexity of the objective functions of the sub-problems as we demonstrate in Lemma~\ref{lem_additivityBounds} below.  We will also prove that the convergence rate is sub-linear. We will require estimates for the Lipschitz bounds on the gradients of our objective functions.  Hence, before stating our results formally and proving them, we will next demonstrate that the sub-problems have strongly convex objective functions and provide Lipschitz bounds on their gradients. 

\medskip

\textit{Strong Convexity}: A function $f: \RR^{p \times q} \rightarrow \RR$ is $\lambda$-strongly convex, for some $\lambda > 0$, if for all $\mathbf{G}_1, \mathbf{G}_2 \in \RR^{p \times q}$
\begin{multline}
	f(\mathbf{G}_1) -  f(\mathbf{G}_2) \ge 
	\big\langle \partial_{\mathbf{G}} f(\mathbf{G}_2), \mathbf{G}_1 - \mathbf{G}_2  \big\rangle \\ + \frac{\lambda}{2} \| \mathbf{G}_1 - \mathbf{G}_2\|_{\FF}^2 \label{eq_strConv1}
\end{multline}
or, equivalently,
\begin{equation}
	\big\langle \partial_{\mathbf{G}} f(\mathbf{G}_1) -  \partial_{\mathbf{G}} f(\mathbf{G}_2), \mathbf{G}_1 - \mathbf{G}_2  \big\rangle \ge \lambda \| \mathbf{G}_1 - \mathbf{G}_2\|_{\FF}^2. \label{eq_strConv2}
\end{equation}
We will use definition \eqref{eq_strConv2} to demonstrate the strong convexity of $F_1^{[k]}$, $F_2^{[k]}$, $F_3^{[k]}$, and $F_4^{[k]}$, defined above, for any $k \ge 1$.  We begin with $F_1^{[k]}$: for any $\UU_1,\UU_2 \in \RR^{n\times r}$, we have that
\begin{multline}
	\big\langle \partial_{\UU}F_1^{[k]}(\UU_1) - \partial_{\UU}F_1^{[k]}(\UU_2), \UU_1 - \UU_2 \big\rangle \\
	= \big\langle 2(\UU_1 - \UU_2)(\widehat{\VV}_{\tr}^{[k-1]\top}\widehat{\VV}_{\tr}^{[k-1]} + \mu \mathbf{I}), \UU_1 - \UU_2\big\rangle \\
	= \trace \big( 2(\widehat{\VV}_{\tr}^{[k-1]\top}\widehat{\VV}_{\tr}^{[k-1]} + \mu \mathbf{I}) (\UU_1 - \UU_2)^{\top} (\UU_1 - \UU_2) \big).
\end{multline}
Since $\mu > 0$, we have that $2\widehat{\VV}_{\tr}^{[k-1]\top}\widehat{\VV}_{\tr}^{[k-1]} + 2\mu \mathbf{I}$ is positive definite with positive eigenvalues.  In particular the smallest eigenvalue, defined as
\begin{equation}
	\underline{\lambda}^{[k]}_1 \equiv \min_{1 \le i \le r}\lambda(2\widehat{\VV}_{\tr}^{[k-1]\top}\widehat{\VV}_{\tr}^{[k-1]} + 2\mu \mathbf{I}) \label{eq_eigen1}
\end{equation}
is positive, i.e., $\underline{\lambda}^{[k]}_1 > 0$, where $\lambda(\MM)$ is a vector of eigenvalues of $\MM$.
It follows immediately that (see, e.g., \citep{fang1994inequalities})
\begin{equation}
	\big\langle \partial_{\UU}F_1^{[k]}(\UU_1) - \partial_{\UU}F_1^{[k]}(\UU_2), \UU_1 - \UU_2 \big\rangle 
	\ge \underline{\lambda}_1^{[k]} \| \UU_1 - \UU_2 \|_{\FF}^2. \label{eq_SC_1}
\end{equation}
Hence, $F_1^{[k]}$ is $\underline{\lambda}^{[k]}_1$-strongly convex.  Similarly,
\begin{equation}
	\big\langle \partial_{\UU}F_2^{[k]}(\UU_1) - \partial_{\UU}F_2^{[k]}(\UU_2), \UU_1 - \UU_2 \big\rangle 
	\ge \underline{\lambda}_2^{[k]} \| \UU_1 - \UU_2 \|_{\FF}^2 \label{eq_SC_2}
\end{equation}
for any $\UU_1,\UU_2 \in \RR^{h\times r}$, where
\begin{equation}
	\underline{\lambda}^{[k]}_2
	\equiv \min_{1 \le i \le r}\lambda(2\widehat{\VV}_{\tr}^{[k-1]\top}\widehat{\VV}_{\tr}^{[k-1]} + 2\widehat{\VV}_{\te}^{[k-1]\top}\widehat{\VV}_{\te}^{[k-1]} + 2\mu \mathbf{I}) \label{eq_eigen2}
\end{equation}
is the smallest eigenvalue of the matrix $2\widehat{\VV}_{\tr}^{[k-1]\top}\widehat{\VV}_{\tr}^{[k-1]} + 2\widehat{\VV}_{\te}^{[k-1]\top}\widehat{\VV}_{\te}^{[k-1]} + 2\mu \mathbf{I}$ and $\underline{\lambda}^{[k]}_2 > 0$.  Hence, $F_2^{[k]}$ is $\underline{\lambda}^{[k]}_2$-strongly convex. It can be similarly shown that $F_3^{[k]}$ and $F_4^{[k]}$ are $\underline{\lambda}^{[k]}_3$-strongly convex and $\underline{\lambda}^{[k]}_4$-strongly convex, respectively, where
\begin{equation}
	\underline{\lambda}^{[k]}_3 \equiv \min_{1 \le i \le r}\lambda(2\widehat{\UU}_{\tr}^{[k]\top}\widehat{\UU}_{\tr}^{[k]} + 2\widehat{\UU}_{\te}^{[k]\top}\widehat{\UU}_{\te}^{[k]} + 2\mu \mathbf{I}) \label{eq_eigen3}
\end{equation}
and
\begin{equation}
	\underline{\lambda}^{[k]}_4 \equiv \min_{1 \le i \le r}\lambda(2\widehat{\UU}_{\te}^{[k]\top}\widehat{\UU}_{\te}^{[k]} + 2\mu \mathbf{I}). \label{eq_eigen4}
\end{equation}

\medskip

\textit{Lipschitz Bounds}: We will now establish that the Lipschitz constants for $\partial_{\UU}F_1^{[k]}$, $\partial_{\UU}F_2^{[k]}$, $\partial_{\VV}F_3^{[k]}$, and $\partial_{\VV}F_4^{[k]}$ are the largest eigenvalues of the matrices above.  For any $\UU_1,\UU_2 \in \RR^{n \times r}$
\begin{multline}
	\big\| \partial_{\UU}F_1^{[k]}(\UU_1) - \partial_{\UU}F_1^{[k]}(\UU_2) \|_{\FF}^2 \\
	= \big\| 2(\UU_1 - \UU_2)(\widehat{\VV}_{\tr}^{[k-1]\top}\widehat{\VV}_{\tr}^{[k-1]} + \mu \mathbf{I}) \big\|_{\FF}^2 \\
	\le \big(\overline{\lambda}_1^{[k]} \big)^2 \| \UU_1 - \UU_2 \|_{\FF}^2, \label{eq_upperBoundLambda1}
\end{multline} 
where 
\begin{equation}
	\overline{\lambda}^{[k]}_1 \equiv \max_{1 \le i \le r}\lambda(2\widehat{\VV}_{\tr}^{[k-1]\top}\widehat{\VV}_{\tr}^{[k-1]} + 2\mu \mathbf{I}). \label{eq_eigenMax1}
\end{equation}
The inequality \eqref{eq_upperBoundLambda1} follows from the bounds in \citep{fang1994inequalities} and we have the Lipschitz condition:
\begin{equation}
	\big\| \partial_{\UU}F_1^{[k]}(\UU_1) - \partial_{\UU}F_1^{[k]}(\UU_2) \|_{\FF}^2 \le \big(\overline{\lambda}^{[k]}_1\big)^2 \| \UU_1 - \UU_2 \|_{\FF}^2. \label{eq_lip1}
\end{equation}
We can similarly establish the Lipschitz conditions
\begin{equation}
	\big\| \partial_{\UU}F_2^{[k]}(\UU_1) - \partial_{\UU}F_2^{[k]}(\UU_2) \|_{\FF}^2 \le \big(\overline{\lambda}^{[k]}_2\big)^2 \| \UU_1 - \UU_2 \|_{\FF}^2, \label{eq_lip2}
\end{equation}
\begin{equation}
	\big\| \partial_{\VV}F_3^{[k]}(\VV_1) - \partial_{\VV}F_3^{[k]}(\VV_2) \|_{\FF}^2 \le \big(\overline{\lambda}^{[k]}_3\big)^2 \| \VV_1 - \VV_2 \|_{\FF}^2, \label{eq_lip3}
\end{equation}
and
\begin{equation}
	\big\| \partial_{\VV}F_4^{[k]}(\VV_1) - \partial_{\VV}F_4^{[k]}(\VV_2) \|_{\FF}^2 \le \big(\overline{\lambda}^{[k]}_4\big)^2 \| \VV_1 - \VV_2 \|_{\FF}^2, \label{eq_lip4}
\end{equation}
where the constants are given by
\begin{equation}
	\overline{\lambda}^{[k]}_2
	\equiv \max_{1 \le i \le r}\lambda(2\widehat{\VV}_{\tr}^{[k-1]\top}\widehat{\VV}_{\tr}^{[k-1]} + 2\widehat{\VV}_{\te}^{[k-1]\top}\widehat{\VV}_{\te}^{[k-1]} + 2\mu \mathbf{I}), \label{eq_eigenMax2}
\end{equation}
\begin{equation}
	\overline{\lambda}^{[k]}_3 \equiv \max_{1 \le i \le r}\lambda(2\widehat{\UU}_{\tr}^{[k]\top}\widehat{\UU}_{\tr}^{[k]} + 2\widehat{\UU}_{\te}^{[k]\top}\widehat{\UU}_{\te}^{[k]} + 2\mu \mathbf{I}) \label{eq_eigenMax3}
\end{equation}
and
\begin{equation}
	\overline{\lambda}^{[k]}_4 \equiv \max_{1 \le i \le r}\lambda(2\widehat{\UU}_{\te}^{[k]\top}\widehat{\UU}_{\te}^{[k]} + 2\mu \mathbf{I}). \label{eq_eigenMax4}
\end{equation}
The smallest eigenvalues, $\{\underline{\lambda}^{[k]}_i\}_{i=1}^4$ are all bounded from below by $2 \mu$ for all $k$, which we need in the sequel.  We can also produce estimates of the largest eigenvalues, $\{\overline{\lambda}^{[k]}_i\}_{i=1}^4$, by appeal to the Perron-Frobenius theorem.  We denote these upper bounds by $L_1 \ge \overline{\lambda}^{[k]}_1$, $L_2 \ge \overline{\lambda}^{[k]}_2$, $L_3 \ge \overline{\lambda}^{[k]}_2$, and $L_4 \ge \overline{\lambda}^{[k]}_4$ (for all $k$) and define the upper bound $L_{\max} \equiv \max\{L_1^2,L_2^2,L_3^2,L_4^2\}$.

\begin{lemma}[Algorithm Convergence]  
	\label{lem_additivityBounds}
	Let $F$, $F_1^{[k]}$, $F_2^{[k]}$, $F_3^{[k]}$, and $F_4^{[k]}$ be as defined in \eqref{eq_objF} and \eqref{eq_objF1}--\eqref{eq_objF4}, and let the block updates $\widehat{\UU}_{\tr}^{[k]}$, $\widehat{\UU}_{\te}^{[k]}$, $\widehat{\VV}_{\tr}^{[k]}$, and $\widehat{\VV}_{\te}^{[k]}$ be as given in \eqref{eq_updateUtr}--\eqref{eq_updateVte}.   Assume that the initial solution $\{\widehat{\UU}_{\tr}^{[0]},\widehat{\UU}_{\te}^{[0]},\widehat{\VV}_{\tr}^{[0]},\widehat{\VV}_{\te}^{[0]}\}$ is such that $F(\widehat{\UU}_{\tr}^{[0]},\widehat{\UU}_{\te}^{[0]},\widehat{\VV}_{\tr}^{[0]},\widehat{\VV}_{\te}^{[0]}) < \infty$.  Then
	\begin{multline}
		\underset{K \rightarrow \infty}{\lim} \sum_{k = 1}^K \Big( \| \widehat{\UU}_{\tr}^{[k-1]} - \widehat{\UU}_{\tr}^{[k]} \|_{\FF}^2 + \| \widehat{\UU}_{\te}^{[k-1]} - \widehat{\UU}_{\te}^{[k]} \|_{\FF}^2
		\\ + \| \widehat{\VV}_{\tr}^{[k-1]} - \widehat{\VV}_{\tr}^{[k]} \|_{\FF}^2 + \| \widehat{\VV}_{\te}^{[k-1]} - \widehat{\VV}_{\te}^{[k]} \|_{\FF}^2 \Big) < \infty.
		\label{eq_globalConvergence}
	\end{multline}
\end{lemma}

\begin{proof}
Noting that $\widehat{\UU}_{\tr}^{[k]}$ and $\widehat{\UU}_{\te}^{[k]}$ minimize $F_1^{[k]}$ and $F_2^{[k]}$, respectively, we readily have that 
\begin{multline}
	2 F(\widehat{\UU}_{\tr}^{[k-1]},\widehat{\UU}_{\te}^{[k-1]},\widehat{\VV}_{\tr}^{[k-1]},\widehat{\VV}_{\te}^{[k-1]}) 
	- F_1^{[k]}(\widehat{\UU}_{\tr}^{[k-1]}) \\
	- F_2^{[k]}(\widehat{\UU}_{\te}^{[k-1]}) - F_3^{[k]}(\widehat{\VV}_{\tr}^{[k-1]}) - F_4^{[k]}(\widehat{\VV}_{\te}^{[k-1]}) \ge 0.
	\label{eq_bound0}
\end{multline}
Similarly, since $\widehat{\VV}_{\tr}^{[k]}$ and $\widehat{\VV}_{\te}^{[k]}$ minimize $F_3^{[k]}$ and $F_4^{[k]}$, respectively, we have that 
\begin{multline}
	2 F(\widehat{\UU}_{\tr}^{[k]},\widehat{\UU}_{\te}^{[k]},\widehat{\VV}_{\tr}^{[k]},\widehat{\VV}_{\te}^{[k]}) - F_1^{[k]}(\widehat{\UU}_{\tr}^{[k]})
	- F_2^{[k]}(\widehat{\UU}_{\te}^{[k]}) \\
	- F_3^{[k]}(\widehat{\VV}_{\tr}^{[k]}) - F_4^{[k]}(\widehat{\VV}_{\te}^{[k]}) \le 0.
	\label{eq_bound01}
\end{multline}
Then, for all $k \ge 1$
\begin{multline}
	F(\widehat{\UU}_{\tr}^{[k-1]},\widehat{\UU}_{\te}^{[k-1]},\widehat{\VV}_{\tr}^{[k-1]},\widehat{\VV}_{\te}^{[k-1]})
	- F(\widehat{\UU}_{\tr}^{[k]},\widehat{\UU}_{\te}^{[k]},\widehat{\VV}_{\tr}^{[k]},\widehat{\VV}_{\te}^{[k]})
	\\ \ge \frac{1}{2} \Big(F_1^{[k]}(\widehat{\UU}_{\tr}^{[k-1]}) - F_1^{[k]}(\widehat{\UU}_{\tr}^{[k]}) + F_2^{[k]}(\widehat{\UU}_{\te}^{[k-1]}) - F_2^{[k]}(\widehat{\UU}_{\te}^{[k]})  \\+ F_3^{[k]}(\widehat{\VV}_{\tr}^{[k-1]}) - F_3^{[k]}(\widehat{\VV}_{\tr}^{[k]}) 
	+ F_4^{[k]}(\widehat{\VV}_{\te}^{[k-1]}) - F_4^{[k]}(\widehat{\VV}_{\te}^{[k]}) \Big). \label{eq_Diffbound}
\end{multline}

Since $F_1^{[k]}$ is strongly convex, we have that
\begin{multline}
	\frac{1}{2} \big( F_1^{[k]}(\widehat{\UU}_{\tr}^{[k-1]}) - F_1^{[k]}(\widehat{\UU}_{\tr}^{[k]}) \big) \\
	\ge \frac{1}{2} \big\langle \partial_{\UU} F_1^{[k]}(\widehat{\UU}_{\tr}^{[k]}) , \widehat{\UU}_{\tr}^{[k-1]} - \widehat{\UU}_{\tr}^{[k]}\big\rangle \\
	+ \frac{\underline{\lambda}_1^{[k]}}{4} \| \widehat{\UU}_{\tr}^{[k-1]} - \widehat{\UU}_{\tr}^{[k]}  \|_{\FF}^2. 
\end{multline}
Since $\partial_{\UU} F_1^{[k]}(\widehat{\UU}_{\tr}^{[k]}) = \mathbf{0}$ and $\underline{\lambda}_1^{[k]} \ge 2\mu$ in accord with the definition \eqref{eq_eigen1}, we have that
\begin{equation}
	\frac{1}{2} \big( F_1^{[k]}(\widehat{\UU}_{\tr}^{[k-1]}) - F_1^{[k]}(\widehat{\UU}_{\tr}^{[k]}) \big) \ge \frac{\mu}{2} \| \widehat{\UU}_{\tr}^{[k-1]} - \widehat{\UU}_{\tr}^{[k]}  \|_{\FF}^2. \label{eq_F1bound}
\end{equation}
Similarly, 
\begin{equation}
	\frac{1}{2} \big( F_2^{[k]}(\widehat{\UU}_{\te}^{[k-1]}) - F_2^{[k]}(\widehat{\UU}_{\te}^{[k]}) \big) \ge \frac{\mu}{2} \| \widehat{\UU}_{\te}^{[k-1]} - \widehat{\UU}_{\te}^{[k]}  \|_{\FF}^2, \label{eq_F2bound}
\end{equation}
\begin{equation}
	\frac{1}{2} \big( F_3^{[k]}(\widehat{\VV}_{\tr}^{[k-1]}) - F_3^{[k]}(\widehat{\VV}_{\tr}^{[k]}) \big) \ge \frac{\mu}{2} \| \widehat{\VV}_{\tr}^{[k-1]} - \widehat{\VV}_{\tr}^{[k]}  \|_{\FF}^2, \label{eq_F3bound}
\end{equation}
and
\begin{equation}
	\frac{1}{2} \big( F_4^{[k]}(\widehat{\VV}_{\te}^{[k-1]}) - F_4^{[k]}(\widehat{\VV}_{\te}^{[k]}) \big) \ge \frac{\mu}{2} \| \widehat{\VV}_{\te}^{[k-1]} - \widehat{\VV}_{\te}^{[k]}  \|_{\FF}^2. \label{eq_F4bound}
\end{equation}

Combining \eqref{eq_Diffbound} with \eqref{eq_F1bound}-\eqref{eq_F4bound} and summing over $k$ from $k=1$ to $k=K$, we get the following inequality:
\begin{multline}
	\frac{2}{\mu}\Big(F(\widehat{\UU}_{\tr}^{[0]},\widehat{\UU}_{\te}^{[0]},\widehat{\VV}_{\tr}^{[0]},\widehat{\VV}_{\te}^{[0]}) - F(\widehat{\UU}_{\tr}^{[K]},\widehat{\UU}_{\te}^{[K]},\widehat{\VV}_{\tr}^{[K]},\widehat{\VV}_{\te}^{[K]}) \Big) \\
	\ge \sum_{k=1}^K \Big( \| \widehat{\UU}_{\tr}^{[k-1]} - \widehat{\UU}_{\tr}^{[k]} \|_{\FF}^2 + \| \widehat{\UU}_{\te}^{[k-1]} - \widehat{\UU}_{\te}^{[k]} \|_{\FF}^2 \\
	+ \| \widehat{\VV}_{\tr}^{[k-1]} - \widehat{\VV}_{\tr}^{[k]} \|_{\FF}^2 + \| \widehat{\VV}_{\te}^{[k-1]} - \widehat{\VV}_{\te}^{[k]} \|_{\FF}^2 \Big).
	\label{eq_objBound}
\end{multline}
Noting that $F(\UU_{\tr},\UU_{\te},\VV_{\tr},\VV_{\te}) \ge 0$ for any $\{\UU_{\tr},\UU_{\te},\VV_{\tr},\VV_{\te} \}$ and taking $K \rightarrow \infty$ completes the proof.
\end{proof}

\begin{corollary}
	\label{corr1}
	The sequence $\{F(\widehat{\UU}_{\tr}^{[k]},\widehat{\UU}_{\te}^{[k]},\widehat{\VV}_{\tr}^{[k]},\widehat{\VV}_{\te}^{[k]})\}_{k \ge 0}$	is non-increasing.  Moreover, there exists a constant $0 < C_0 \le \frac{2}{\mu}F(\widehat{\UU}_{\tr}^{[0]},\widehat{\UU}_{\te}^{[0]},\widehat{\VV}_{\tr}^{[0]},\widehat{\VV}_{\te}^{[0})$ such that
	\begin{multline}
		C_0 \ge \| \widehat{\UU}_{\tr}^{[k-1]} - \widehat{\UU}_{\tr}^{[k]} \|_{\FF}^2 + \| \widehat{\UU}_{\te}^{[k-1]} - \widehat{\UU}_{\te}^{[k]} \|_{\FF}^2 \\
		+ \| \widehat{\VV}_{\tr}^{[k-1]} - \widehat{\VV}_{\tr}^{[k]} \|_{\FF}^2 + \| \widehat{\VV}_{\te}^{[k-1]} - \widehat{\VV}_{\te}^{[k]} \|_{\FF}^2
	\end{multline}
	for any $k \ge 1$.
\end{corollary}
\begin{proof}
The first assertion follows immediately from \eqref{eq_Diffbound} and \eqref{eq_F1bound} - \eqref{eq_F4bound}.  The second assertion follows immediately from the first assertion and \eqref{eq_objBound}.
\end{proof}

Lemma \ref{lem_additivityBounds} implies convergence (in the $\ell_2$ sense) to a limit point, that is, \eqref{eq_globalConvergence} implies that as $k \rightarrow \infty$, $\| \widehat{\UU}_{\te}^{[k-1]} - \widehat{\UU}_{\te}^{[k]} \|_{\FF}^2 \rightarrow 0$, $\| \widehat{\UU}_{\tr}^{[k-1]} - \widehat{\UU}_{\tr}^{[k]} \|_{\FF}^2 \rightarrow 0$, $\| \widehat{\VV}_{\tr}^{[k-1]} - \widehat{\VV}_{\tr}^{[k]} \|_{\FF}^2 \rightarrow 0$, $\| \widehat{\VV}_{\te}^{[k-1]} - \widehat{\VV}_{\te}^{[k]} \|_{\FF}^2 \rightarrow 0$, and that there exists a solution $\{\overline{\UU}_{\tr},\overline{\UU}_{\te},\overline{\VV}_{\tr},\overline{\VV}_{\te}\}$, such that
\begin{multline}
	\| \widehat{\UU}_{\tr}^{[k]} - \overline{\UU}_{\tr} \|_{\FF}^2 + \| \widehat{\UU}_{\te}^{[k]} - \overline{\UU}_{\te} \|_{\FF}^2 \\
	+ \| \widehat{\VV}_{\tr}^{[k]} - \overline{\VV}_{\tr} \|_{\FF}^2 + \| \widehat{\VV}_{\te}^{[k]} - \overline{\VV}_{\te} \|_{\FF}^2 \underset{k \rightarrow \infty}{\longrightarrow} 0.
\end{multline}

We prove that the convergence rate is sub-linear.  To do so, we next state a well-known result (see for example \cite[Lemma 3.5]{beck2013convergence}), which we provide for the sake of completeness.

\begin{lemma}
	\label{lem_sublinear}
	Let $\{q_k\}_{k \ge 0}$ be a non-increasing sequence of positive real numbers and let $0 < \overline{q} < \infty$ and $0 < C < \infty$ be two positive constants.  Define $B \equiv \frac{C}{\overline{q}}$ and assume that (i) $q_0 < \overline{q}$ and (ii) $q_{k-1} - q_k \ge C^{-1}q_{k-1}^2$.  
	Then
	\begin{equation}
		q_k \le \frac{C}{B + k}.
	\end{equation}
\end{lemma}  
\begin{proof}
The second condition implies that
\begin{equation}
	\frac{1}{q_k} - \frac{1}{q_{k-1}} = \frac{q_{k-1} - q_k}{q_{k-1} q_k} \ge \frac{q_{k-1}^2}{C(q_{k-1} q_k)} = \frac{1}{C} \frac{q_{k-1}}{q_k} \ge \frac{1}{C}.
\end{equation}
Then
\begin{equation}
	\frac{1}{q_k} \ge 	\frac{1}{q_0} + \frac{k}{C} \ge \overline{q}^{-1} + \frac{k}{C},
\end{equation}
which completes the proof.
\end{proof}

\begin{lemma}
	\label{lem_ratioBound}
	Assume the conditions of Lemma \ref{lem_additivityBounds} hold.  Assume that $\partial_{\UU} F_1^{[k]}(\widehat{\UU}_{\tr}^{[k-1]}) \ne \mathbf{0}$, $\partial_{\UU} F_2^{[k]}(\widehat{\UU}_{\te}^{[k-1]}) \ne \mathbf{0}$, $\partial_{\VV} F_3^{[k]}(\widehat{\VV}_{\tr}^{[k-1]}) \ne \mathbf{0}$, and $\partial_{\VV} F_4^{[k]}(\widehat{\VV}_{\te}^{[k-1]}) \ne \mathbf{0}$.  Let
	\begin{multline}
		\Sigma^{[k]} \\ \equiv \Big( F_1^{[k]}(\widehat{\UU}_{\tr}^{[k]}) + F_2^{[k]}(\widehat{\UU}_{\te}^{[k]}) + F_3^{[k]}(\widehat{\VV}_{\tr}^{[k]}) + F_4^{[k]}(\widehat{\VV}_{\te}^{[k]}) \Big)^2
	\end{multline}
	and
	\begin{multline}
		\Delta^{[k]} \equiv \big\| \partial_{\UU} F_1^{[k]}(\widehat{\UU}_{\tr}^{[k-1]}) \big\|_{\FF}^2 + \big\| \partial_{\UU} F_2^{[k]}(\widehat{\UU}_{\te}^{[k-1]}) \big\|_{\FF}^2 \\
		+ \big\| \partial_{\VV} F_3^{[k]}(\widehat{\VV}_{\tr}^{[k-1]}) \big\|_{\FF}^2 + \big\| \partial_{\VV} F_4^{[k]}(\widehat{\VV}_{\te}^{[k-1]}) \big\|_{\FF}^2
	\end{multline}
	Then there exists a positive constant $0 < \overline{C}_0 < \infty$, which only depends on the initial solution $(\widehat{\UU}_{\tr}^{[0]},\widehat{\UU}_{\te}^{[0]},\widehat{\VV}_{\tr}^{[0]},\widehat{\VV}_{\te}^{[0]})$ such that
	\begin{equation}
		\overline{C}_0 \ge  1 + \frac{\widetilde{C} \Sigma^{[k]}}{\Delta^{[k]}}
	\end{equation}
	for any positive constant $0 < \widetilde{C} < \infty$ and all $k \ge 1$.
\end{lemma}

\begin{proof}
Define a constant $\widetilde{C}_1$ so that $\widetilde{C}_1 \ge F_1^{[0]}(\widehat{\UU}_{\tr}^{[0]}) + F_2^{[0]}(\widehat{\UU}_{\te}^{[0]}) + F_3^{[0]}(\widehat{\VV}_{\tr}^{[0]}) + F_4^{[0]}(\widehat{\VV}_{\te}^{[0]})$.  Then from \eqref{eq_bound0}, \eqref{eq_bound01}, and Corollary \ref{corr1}, we have that $\widetilde{C}_1 \ge F_1^{[k]}(\widehat{\UU}_{\tr}^{[k]}) + F_2^{[k]}(\widehat{\UU}_{\te}^{[k]}) + F_3^{[k]}(\widehat{\VV}_{\tr}^{[k]}) + F_4^{[k]}(\widehat{\VV}_{\te}^{[k]})$ for all $k \ge 1$. Next, define 
\begin{multline}
	\mathcal{K} \equiv \big\{k: \partial_{\UU} F_1^{[k]}(\widehat{\UU}_{\tr}^{[k-1]}) \ne \mathbf{0}, \partial_{\UU} F_2^{[k]}(\widehat{\UU}_{\te}^{[k-1]}) \ne \mathbf{0}, \\
	\partial_{\VV} F_3^{[k]}(\widehat{\VV}_{\tr}^{[k-1]}) \ne \mathbf{0}, \partial_{\VV} F_4^{[k]}(\widehat{\VV}_{\te}^{[k-1]}) \ne \mathbf{0} \big\}. \label{eq_preconvergence}
\end{multline}
We immediately have that there exists constants $0 < \widetilde{C}_2^{[k]} < \infty$ so that
\begin{multline}
	\widetilde{C}_2^{[k]} \le \big\| \partial_{\UU} F_1^{[k]}(\widehat{\UU}_{\tr}^{[k-1]}) \big\|_{\FF}^2 + \big\| \partial_{\UU} F_2^{[k]}(\widehat{\UU}_{\te}^{[k-1]}) \big\|_{\FF}^2 \\
	+ \big\| \partial_{\VV} F_3^{[k]}(\widehat{\VV}_{\tr}^{[k-1]}) \big\|_{\FF}^2 + \big\| \partial_{\VV} F_4^{[k]}(\widehat{\VV}_{\te}^{[k-1]}) \big\|_{\FF}^2
\end{multline}
for all $k \in \mathcal{K}$. Define the lower bound
\begin{equation}
	\widetilde{C}_2 \equiv \underset{k \in \mathcal{K}}{ \inf} ~ \widetilde{C}_2^{[k]}.
\end{equation} 
Letting $\overline{C}_0 = 1 + \widetilde{C}_2^{-1}\widetilde{C}\widetilde{C}_1^2$ completes the proof.	
\end{proof}

The set $\mathcal{K}$ defined in \eqref{eq_preconvergence} is the set of iteration indices before convergence is achieved.  
It follows immediately from Lemma \ref{lem_ratioBound} that before the algorithm converges (i.e., for $k \in \mathcal{K}$) there exists a positive constant $0 < \overline{C}_0 < \infty$ such that
\begin{multline}
	\overline{C}_0 \Big(\big\| \partial_{\UU} F_1^{[k]}(\widehat{\UU}_{\tr}^{[k-1]}) \big\|_{\FF}^2 + \big\| \partial_{\UU} F_2^{[k]}(\widehat{\UU}_{\te}^{[k-1]}) \big\|_{\FF}^2 \\
	+ \big\| \partial_{\VV} F_3^{[k]}(\widehat{\VV}_{\tr}^{[k-1]}) \big\|_{\FF}^2 + \big\| \partial_{\VV} F_4^{[k]}(\widehat{\VV}_{\te}^{[k-1]}) \big\|_{\FF}^2 \Big) \\
	\ge \big\| \partial_{\UU} F_1^{[k]}(\widehat{\UU}_{\tr}^{[k-1]}) \big\|_{\FF}^2 + \big\| \partial_{\UU} F_2^{[k]}(\widehat{\UU}_{\te}^{[k-1]}) \big\|_{\FF}^2 \\
	+ \big\| \partial_{\VV} F_3^{[k]}(\widehat{\VV}_{\tr}^{[k-1]}) \big\|_{\FF}^2 + \big\| \partial_{\VV} F_4^{[k]}(\widehat{\VV}_{\te}^{[k-1]}) \big\|_{\FF}^2 \\
	+ \widetilde{C} \Big( F_1^{[k]}(\widehat{\UU}_{\tr}^{[k]}) + F_2^{[k]}(\widehat{\UU}_{\te}^{[k]}) + F_3^{[k]}(\widehat{\VV}_{\tr}^{[k]}) + F_4^{[k]}(\widehat{\VV}_{\te}^{[k]}) \Big)^2. \label{eq_inequality}
\end{multline}

We next prove our first main result related to the speed of convergence of Algorithm \ref{A:0}.  Specifically, we prove that the distance from the final solution shrinks in inverse proportion to the number of iterations, that is, few iterations are required to achieve convergence and the number of iterations can be specified beforehand.

\begin{theorem}[Sub-linear Convergence Rate]
	\label{thm_SCR}
	Let $F$, $F_1^{[k]}$, $F_2^{[k]}$, $F_3^{[k]}$, and $F_4^{[k]}$ be as defined in \eqref{eq_objF} and \eqref{eq_objF1}--\eqref{eq_objF4}, and let the block updates $\widehat{\UU}_{\tr}^{[k]}$, $\widehat{\UU}_{\te}^{[k]}$, $\widehat{\VV}_{\tr}^{[k]}$, and $\widehat{\VV}_{\te}^{[k]}$ be as given in \eqref{eq_updateUtr}--\eqref{eq_updateVte}. Then, there exists two positive constants $0 < B < \infty$ and $0 < C < \infty$ such that, for any $k \ge 0$,
	\begin{multline}
		\Big| F(\widehat{\UU}_{\tr}^{[k]},\widehat{\UU}_{\te}^{[k]},\widehat{\VV}_{\tr}^{[k]},\widehat{\VV}_{\te}^{[k]}) - F(\overline{\UU}_{\tr},\overline{\UU}_{\te},\overline{\VV}_{\tr},\overline{\VV}_{\te}) \Big| \\
		\le \frac{C}{B+k}.
	\end{multline}
\end{theorem}

\begin{proof}
The case $k = 0$ is trivial so we will focus on $k \ge 1$. 
From \eqref{eq_Diffbound}, \eqref{eq_F1bound} - \eqref{eq_F4bound}, and the Lipschitz bounds \eqref{eq_lip1} - \eqref{eq_lip2}, we have that
\begin{multline}
	F(\widehat{\UU}_{\tr}^{[k-1]},\widehat{\UU}_{\te}^{[k-1]},\widehat{\VV}_{\tr}^{[k-1]},\widehat{\VV}_{\te}^{[k-1]})
	- F(\widehat{\UU}_{\tr}^{[k]},\widehat{\UU}_{\te}^{[k]},\widehat{\VV}_{\tr}^{[k]},\widehat{\VV}_{\te}^{[k]}) \\ 
	\ge \frac{\mu}{2L_{\max}} \Big( \big\| \partial_{\UU}F_1^{[k]}(\widehat{\UU}_{\tr}^{[k-1]}) 
	- \partial_{\UU}F_1^{[k]}(\widehat{\UU}_{\tr}^{[k]}) \big\|_{\FF}^2 \\
	+ \big\| \partial_{\UU}F_2^{[k]}(\widehat{\UU}_{\te}^{[k-1]}) - \partial_{\UU}F_2^{[k]}(\widehat{\UU}_{\te}^{[k]}) \big\|_{\FF}^2 \\
	+ \big\| \partial_{\VV}F_3^{[k]}(\widehat{\VV}_{\tr}^{[k-1]}) - \partial_{\VV}F_3^{[k]}(\widehat{\VV}_{\tr}^{[k]}) \big\|_{\FF}^2  \\
	+ \big\| \partial_{\VV}F_4^{[k]}(\widehat{\VV}_{\te}^{[k-1]})
	- \partial_{\VV}F_4^{[k]}(\widehat{\VV}_{\te}^{[k]}) \big\|_{\FF}^2\Big) \\
	= \frac{\mu}{2L_{\max}} \Big( \big\| \partial_{\UU}F_1^{[k]}(\widehat{\UU}_{\tr}^{[k-1]}) \big\|_{\FF}^2 
	+ \big\| \partial_{\UU}F_2^{[k]}(\widehat{\UU}_{\te}^{[k-1]}) \big\|_{\FF}^2 \\
	+ \big\| \partial_{\VV}F_3^{[k]}(\widehat{\VV}_{\tr}^{[k-1]}) \big\|_{\FF}^2 
	+ \big\| \partial_{\VV}F_4^{[k]}(\widehat{\VV}_{\te}^{[k-1]}) \big\|_{\FF}^2\Big).
	\label{eq_bound1}
\end{multline}
By convexity of $F_1^{[k]}$ and the Cauchy-Schwartz inequality, we have that
\begin{multline}
	F_1^{[k]}(\widehat{\UU}_{\tr}^{[k-1]}) -  F_1^{[k]}(\widehat{\UU}_{\tr}^{[k]}) 
	\le \big\langle \partial_{\UU} F_1^{[k]}(\widehat{\UU}_{\tr}^{[k-1]}), \widehat{\UU}_{\tr}^{[k-1]} - \widehat{\UU}_{\tr}^{[k]}  \big\rangle \\
	\le \big\| \partial_{\UU} F_1^{[k]}(\widehat{\UU}_{\tr}^{[k-1]}) \big\|_{\FF} \big\| \widehat{\UU}_{\tr}^{[k-1]} - \widehat{\UU}_{\tr}^{[k]} \big\|_{\FF}.
	\label{eq_bound2}
\end{multline}
Then
\begin{equation}
	\big\| \partial_{\UU} F_1^{[k]}(\widehat{\UU}_{\tr}^{[k-1]}) \big\|_{\FF}^2 \ge \frac{1}{C_0} \big( F_1^{[k]}(\widehat{\UU}_{\tr}^{[k-1]}) -  F_1^{[k]}(\widehat{\UU}_{\tr}^{[k]}) \big)^2,
	\label{eq_bound3}
\end{equation}
where $C_0$ is the positive constant of Corollary \ref{corr1}. We can write similar bounds for $F_2^{[k]}$, $F_3^{[k]}$, and $F_2^{[k]}$ to obtain
\begin{multline}
	\frac{1}{4C_0}\Big(F_1^{[k]}(\widehat{\UU}_{\tr}^{[k-1]}) + F_2^{[k]}(\widehat{\UU}_{\te}^{[k-1]}) + F_3^{[k]}(\widehat{\VV}_{\tr}^{[k-1]}) \\
	+ F_4^{[k]}(\widehat{\VV}_{\te}^{[k-1]}) 
	- F_1^{[k]}(\widehat{\UU}_{\tr}^{[k]}) - F_2^{[k]}(\widehat{\UU}_{\te}^{[k]}) - F_3^{[k]}(\widehat{\VV}_{\tr}^{[k]}) \\
	- F_4^{[k]}(\widehat{\VV}_{\te}^{[k]}) \Big)^2 
	\le \big\| \partial_{\UU} F_1^{[k]}(\widehat{\UU}_{\tr}^{[k-1]}) \big\|_{\FF}^2 + \big\| \partial_{\UU} F_2^{[k]}(\widehat{\UU}_{\te}^{[k-1]}) \big\|_{\FF}^2  \\
	+ \big\| \partial_{\UU} F_3^{[k]}(\widehat{\VV}_{\tr}^{[k-1]}) \big\|_{\FF}^2 + \big\| \partial_{\UU} F_4^{[k]}(\widehat{\VV}_{\te}^{[k-1]}) \big\|_{\FF}^2,
	\label{eq_bound4}
\end{multline}
where the left hand side follows from the triangle inequality.  Applying the reverse triangle inequality, we get
\begin{multline}
	\frac{1}{4C_0}\Big(F_1^{[k]}(\widehat{\UU}_{\tr}^{[k-1]}) + F_2^{[k]}(\widehat{\UU}_{\te}^{[k-1]}) + F_3^{[k]}(\widehat{\VV}_{\tr}^{[k-1]}) \\
	+ F_4^{[k]}(\widehat{\VV}_{\te}^{[k-1]}) \Big)^2 
	\le \frac{1}{4C_0}\Big(F_1^{[k]}(\widehat{\UU}_{\tr}^{[k]}) + F_2^{[k]}(\widehat{\UU}_{\te}^{[k]}) \\
	+ F_3^{[k]}(\widehat{\VV}_{\tr}^{[k]}) + F_4^{[k]}(\widehat{\VV}_{\te}^{[k]}) \Big)^2 
	+ \big\| \partial_{\UU} F_1^{[k]}(\widehat{\UU}_{\tr}^{[k-1]}) \big\|_{\FF}^2  \\
	+ \big\| \partial_{\UU} F_2^{[k]}(\widehat{\UU}_{\te}^{[k-1]}) \big\|_{\FF}^2    
	+ \big\| \partial_{\UU} F_3^{[k]}(\widehat{\VV}_{\tr}^{[k-1]}) \big\|_{\FF}^2 \\
	+ \big\| \partial_{\UU} F_4^{[k]}(\widehat{\VV}_{\te}^{[k-1]}) \big\|_{\FF}^2.
	\label{eq_bound5}
\end{multline}
It can be easily shown that 
\begin{multline}
	F(\widehat{\UU}_{\tr}^{[k-1]},\widehat{\UU}_{\te}^{[k-1]},\widehat{\VV}_{\tr}^{[k-1]},\widehat{\VV}_{\te}^{[k-1]}) \\
	\le F_1^{[k]}(\widehat{\UU}_{\tr}^{[k-1]}) + F_2^{[k]}(\widehat{\UU}_{\te}^{[k-1]}) + F_3^{[k]}(\widehat{\VV}_{\tr}^{[k-1]}) \\
	+ F_4^{[k]}(\widehat{\VV}_{\te}^{[k-1]}).
	\label{eq_bound6}
\end{multline}
Hence,
\begin{multline}
	\frac{1}{4C_0}\Big(F(\widehat{\UU}_{\tr}^{[k-1]},\widehat{\UU}_{\te}^{[k-1]},\widehat{\VV}_{\tr}^{[k-1]},\widehat{\VV}_{\te}^{[k-1]}) \Big)^2 \\
	\le \frac{1}{4C_0}\Big(F_1^{[k]}(\widehat{\UU}_{\tr}^{[k]}) + F_2^{[k]}(\widehat{\UU}_{\te}^{[k]}) 
	+ F_3^{[k]}(\widehat{\VV}_{\tr}^{[k]}) 
	+ F_4^{[k]}(\widehat{\VV}_{\te}^{[k]}) \Big)^2 \\
	+ \big\| \partial_{\UU} F_1^{[k]}(\widehat{\UU}_{\tr}^{[k-1]}) \big\|_{\FF}^2 
	+ \big\| \partial_{\UU} F_2^{[k]}(\widehat{\UU}_{\te}^{[k-1]}) \big\|_{\FF}^2 \\ 
	+ \big\| \partial_{\UU} F_3^{[k]}(\widehat{\VV}_{\tr}^{[k-1]}) \big\|_{\FF}^2 
	+ \big\| \partial_{\UU} F_4^{[k]}(\widehat{\VV}_{\te}^{[k-1]}) \big\|_{\FF}^2.
	\label{eq_bound7}
\end{multline}
From \eqref{eq_inequality} (and Lemma \ref{lem_ratioBound}) we have that there exists a positive constant $\overline{C}_0$ which only depends on the initial solution so that \eqref{eq_bound7} implies that
\begin{multline}
	\frac{1}{4C_0\overline{C}_0}\Big(F(\widehat{\UU}_{\tr}^{[k-1]},\widehat{\UU}_{\te}^{[k-1]},\widehat{\VV}_{\tr}^{[k-1]},\widehat{\VV}_{\te}^{[k-1]}) \Big)^2 \\
	\le \big\| \partial_{\UU} F_1^{[k]}(\widehat{\UU}_{\tr}^{[k-1]}) \big\|_{\FF}^2 + \big\| \partial_{\UU} F_2^{[k]}(\widehat{\UU}_{\te}^{[k-1]}) \big\|_{\FF}^2  \\
	+ \big\| \partial_{\UU} F_3^{[k]}(\widehat{\VV}_{\tr}^{[k-1]}) \big\|_{\FF}^2 + \big\| \partial_{\UU} F_4^{[k]}(\widehat{\VV}_{\te}^{[k-1]}) \big\|_{\FF}^2.
	\label{eq_bound8}
\end{multline}
Combining this with \eqref{eq_bound1}, we get
\begin{multline}
	F(\widehat{\UU}_{\tr}^{[k-1]},\widehat{\UU}_{\te}^{[k-1]},\widehat{\VV}_{\tr}^{[k-1]},\widehat{\VV}_{\te}^{[k-1]}) 
	- F(\widehat{\UU}_{\tr}^{[k]},\widehat{\UU}_{\te}^{[k]},\widehat{\VV}_{\tr}^{[k]},\widehat{\VV}_{\te}^{[k]}) \\
	\ge \frac{\mu}{8L_{\max}C_0 \overline{C}_0} \Big( F(\widehat{\UU}_{\tr}^{[k-1]},\widehat{\UU}_{\te}^{[k-1]},\widehat{\VV}_{\tr}^{[k-1]},\widehat{\VV}_{\te}^{[k-1]}) \Big)^2.
	\label{eq_bound9}
\end{multline}
Since $F(\widehat{\UU}_{\tr}^{[k]},\widehat{\UU}_{\te}^{[k]},\widehat{\VV}_{\tr}^{[k]},\widehat{\VV}_{\te}^{[k]}) \ge F(\overline{\UU}_{\tr},\overline{\UU}_{\te},\overline{\VV}_{\tr},\overline{\VV}_{\te})$ for all $k$, \eqref{eq_bound9} implies that
\begin{multline}
	F(\widehat{\UU}_{\tr}^{[k-1]},\widehat{\UU}_{\te}^{[k-1]},\widehat{\VV}_{\tr}^{[k-1]},\widehat{\VV}_{\te}^{[k-1]}) 
	- F(\widehat{\UU}_{\tr}^{[k]},\widehat{\UU}_{\te}^{[k]},\widehat{\VV}_{\tr}^{[k]},\widehat{\VV}_{\te}^{[k]}) \\
	\ge \frac{\mu}{8L_{\max}C_0 \overline{C}_0} \Big( F(\widehat{\UU}_{\tr}^{[k-1]},\widehat{\UU}_{\te}^{[k-1]},\widehat{\VV}_{\tr}^{[k-1]},\widehat{\VV}_{\te}^{[k-1]})  \\
	- F(\overline{\UU}_{\tr},\overline{\UU}_{\te},\overline{\VV}_{\tr},\overline{\VV}_{\te}) \Big)^2.
	\label{eq_bound10}
\end{multline}
We next add and subtract $F(\overline{\UU}_{\tr},\overline{\UU}_{\te},\overline{\VV}_{\tr},\overline{\VV}_{\te})$ on the left hand side, and we invoke Lemma \eqref{lem_sublinear} with 
\begin{equation}
	q_k \equiv F(\widehat{\UU}_{\tr}^{[k]},\widehat{\UU}_{\te}^{[k]},\widehat{\VV}_{\tr}^{[k]},\widehat{\VV}_{\te}^{[k]}) - F(\overline{\UU}_{\tr},\overline{\UU}_{\te},\overline{\VV}_{\tr},\overline{\VV}_{\te})
\end{equation}
and note that $\{q_k\}_{k \ge 0}$ is a non-negative sequence by the first assertion of Corollary \ref{corr1}.  Letting $B = 16L_{\max} \overline{C}_0 \mu^{-2}$ and $C = 8 L_{\max} C_0 \overline{C}_0 \mu^{-1}$ completes the proof.
\end{proof}

The result in Theorem \ref{thm_SCR} suggests that one can determine a number of iterations as a stopping criterion for Algorithm \ref{A:0}.  More importantly, the theorem says that this number need not be large in order to get close to the limit.  The next theorem provides our second main result.  It demonstrates that the limiting solution $\{\overline{\UU}_{\tr},\overline{\UU}_{\te},\overline{\VV}_{\tr},\overline{\VV}_{\te}\}$ implied by Lemma \ref{lem_additivityBounds} is a block coordinate-wise minimizer of the matrix completion problem \eqref{eq_KMCP}.

\begin{theorem}[Block Corrdinate-Wise Minimizer]
	\label{thm_BCW_Min}
	Under the assumptions of Lemma \ref{lem_additivityBounds}, the variational inequalities \eqref{eq_VI1}-\eqref{eq_VI4} hold for all $\UU_{\tr} \in \RR^{n \times r}$, all $\UU_{\te} \in \RR^{h \times r}$, all $\VV_{\tr} \in \RR^{T_{\tr} \times r}$, and all $\VV_{\te} \in \RR^{T_{\te} \times r}$.
\end{theorem}

\begin{proof}
In the limit, we have that
\begin{equation}
	\overline{\UU}_{\tr} = \underset{\UU_{\tr} \in  \RR^{n \times r}}{\arg \min} ~ \|\UU_{\tr} \overline{\VV}_{\tr}^{\top} - \YY_{\tr}\|^2_{\FF} + 2\mu \| \UU_{\tr} \|_{\FF}^2.
\end{equation}
Hence, for all $\UU_{\tr} \in \RR^{n \times r}$
\begin{equation}
	F(\UU_{\tr}, \overline{\UU}_{\te},\overline{\VV}_{\tr},\overline{\VV}_{\te}) \ge F(\overline{\UU}_{\tr}, \overline{\UU}_{\te},\overline{\VV}_{\tr},\overline{\VV}_{\te}). \label{eq_inequality1}
\end{equation}
In particular, consider the matrix $\overline{\UU}_{\tr} + h (\UU_{\tr} - \overline{\UU}_{\tr})$, where $h$ is a scalar and $\UU_{\tr}$ is any $\RR^{n \times r}$ matrix.  From \eqref{eq_inequality1}, we have for each $h > 0$ that
\begin{multline}
	\frac{1}{h} \Big( F\big(\overline{\UU}_{\tr} + h (\UU_{\tr} - \overline{\UU}_{\tr}), \overline{\UU}_{\te},\overline{\VV}_{\tr},\overline{\VV}_{\te}\big) \\
	- F(\overline{\UU}_{\tr}, \overline{\UU}_{\te},\overline{\VV}_{\tr},\overline{\VV}_{\te}) \Big) \ge 0
\end{multline}
is true for all $\UU_{\tr} \in \RR^{n \times r}$.  
Since $F$ is continuous in all block coordinates, upon taking $h \rightarrow 0$ the left-hand side converges to the directional partial derivative of $F$ along $\UU_{\tr} - \overline{\UU}_{\tr}$ and since the inequality is true for all $h$, it is true in the limit.  Then, since the directional derivative is simply the inner product of the derivative and the direction, it follows that the first variational inequality,
\begin{equation}
	\big\langle \partial_{\UU_{\tr}}F(\overline{\UU}_{\tr}, \overline{\UU}_{\te},\overline{\VV}_{\tr},\overline{\VV}_{\te}), \UU_{\tr} - \overline{\UU}_{\tr} \big\rangle \ge 0,
\end{equation}
is true for all $\UU_{\tr} \in \RR^{n \times r}$.  The same reasoning can be used to demonstrate that the variational inequalities \eqref{eq_VI2}, \eqref{eq_VI3}, and \eqref{eq_VI4} also hold for $\overline{\UU}_{\te}$, $\overline{\VV}_{\tr}$, and $\overline{\VV}_{\te}$, respectively.
\end{proof}

The interpretation of the solution as a Nash point (an equilibrium) can be seen immediately upon examining the inequalities
\begin{equation}
	F(\UU_{\tr}, \overline{\UU}_{\te},\overline{\VV}_{\tr},\overline{\VV}_{\te}) \ge F(\overline{\UU}_{\tr}, \overline{\UU}_{\te},\overline{\VV}_{\tr},\overline{\VV}_{\te}), \label{eq_nash1}
\end{equation}
\begin{equation}
	F(\overline{\UU}_{\tr},\UU_{\te},\overline{\VV}_{\tr},\overline{\VV}_{\te}) \ge F(\overline{\UU}_{\tr}, \overline{\UU}_{\te},\overline{\VV}_{\tr},\overline{\VV}_{\te}), \label{eq_nash2}
\end{equation}
\begin{equation}
	F(\overline{\UU}_{\tr},\overline{\UU}_{\te},\VV_{\tr},\overline{\VV}_{\te}) \ge F(\overline{\UU}_{\tr}, \overline{\UU}_{\te},\overline{\VV}_{\tr},\overline{\VV}_{\te}), \label{eq_nash3}
\end{equation}
and
\begin{equation}
	F(\overline{\UU}_{\tr},\overline{\UU}_{\te},\overline{\VV}_{\tr},\VV_{\te}) \ge F(\overline{\UU}_{\tr}, \overline{\UU}_{\te},\overline{\VV}_{\tr},\overline{\VV}_{\te}) \label{eq_nash4}
\end{equation}
for all $\UU_{\tr} \in \RR^{n \times r}$ $\UU_{\te} \in \RR^{h \times r}$, $\VV_{\tr} \in \RR^{T_{\tr} \times r}$, and $\VV_{\te} \in \RR^{T_{\te} \times r}$.  The inequalities \eqref{eq_nash2}, \eqref{eq_nash3}, and \eqref{eq_nash4} can be established in the same way that \eqref{eq_nash1} was established in the proof of Theorem \ref{thm_BCW_Min}.

\section{Ensemble Learning, Sample Complexity, and Time Complexity} \label{S:ensemble}
\subsection{Incorporating Day-to-Day Historical Patterns}
Ensemble learning, a.k.a. \emph{adaptive boosting}, is a technique that is used to enhance the performance of a learning task by combining the predictions of an ensemble of ``weak learners''. The weak learners can include a variety of algorithms designed to perform the same prediction task.  The prediction performance of the individual algorithms can be weak (hence the nomenclature ``weak learners'') but the predictions produced by ensemble learning, which are weighted sums of the individual predictions, are \emph{guaranteed} to have lower errors on the training data.  Our proposed ensemble learning implementation considers data from past days as the ensemble.  In other words, each of our weak learners solves a matrix completion problem (using Algorithms \ref{A:0} and \ref{A:2}) with training data covering the same time periods but from different days.  In this way, we can interpret our ensemble learning approach as one that \emph{learns and utilizes day-to-day historical patterns}. 

Without loss of generality, let $\DD$ be set of indices representing the weak learners; each element of $\DD$ corresponds to a day in the past with lower indices representing more recent days, in particular, $d=1$ corresponds to the `present'.  The set $\DD$ can be chosen to include 4-5 weeks of past data and either week days or week ends are chosen based on whether the present day is a week day or week end.  Let $\theta(t,d)$ denote the weight associated with time step $t$ in ensemble $d$.  The joint matrix is now given by
\begin{equation}
	\ZZ(\Theta)
	\equiv 
	\begin{bmatrix}
		\YY_{|\DD|,\tr} & \cdots & \YY_{1,\tr} & \YY_{\te} \\ 
		\Theta_{|\DD|} \odot \Phi(\XX_{|\DD|,\tr}) & \cdots & \Theta_1 \odot \Phi(\XX_{1,\tr}) & \Phi_{\te} 
	\end{bmatrix},
	\label{eq_Z3}
\end{equation}
where $\YY_{d,\tr} \equiv [\mathbf{y}_{d,\tr}(1) \cdots \mathbf{y}_{d,\tr}(T_{\tr})]$ and $\XX_{d,\tr} \equiv [\BB_L\mathbf{x}_{d,\tr}(1) \cdots \BB_L\mathbf{x}_{d,\tr}(T_{\tr})]$ are the output and input matrices, respectively, corresponding to day $d$,  $\Phi$ applies the (unknown) mapping function to each of the columns of its argument, 
\begin{equation}
	\Theta_d \equiv \begin{bmatrix} \theta(1,d)\mathbf{e} \cdots \theta(T_{\tr},d)\mathbf{e} \end{bmatrix} \in  \RR^{h \times T_{\tr}}
\end{equation}
is a matrix of weights corresponding to day $d$, $\mathbf{e}$ is vector of 1s of dimension $h$, and $\odot$ is the component-wise (or Hadamard) product.  Given $\{ \theta(t,d) \}_{1 \le t \le T_{\tr}, d \in \DD}$, the prediction problem is simply a matrix completion problem, which is solved by block coordinate descent.  Here, the training data are given by the augmented matrices 
\begin{equation}
	\widetilde{\YY}_{\tr} \equiv [\YY_{|\DD|,\tr} \cdots \YY_{1,\tr}] \in \RR^{n \times |\DD|T_{\tr}}
\end{equation}
and 
\begin{multline}
	\widetilde{\Phi}_{\tr}(\Theta) \equiv \begin{bmatrix}\Theta_{|\DD|} \odot \Phi(\XX_{|\DD|,\tr}) & \cdots & \Theta_1 \odot \Phi(\XX_{1,\tr}) \end{bmatrix} \\
	\in \RR^{h \times |\DD|T_{\tr}}.	
\end{multline}
The joint matrix is given by
\begin{equation}
	\ZZ(\Theta) \equiv 
	\begin{bmatrix}
		\widetilde{\YY}_{\tr} & \YY_{\te} \\ 
		\widetilde{\Phi}_{\tr}(\Theta) & \Phi_{\te} 
	\end{bmatrix} \in \RR^{(n+h)\times (|\DD|T_{\tr}+T_{\te})}.
	\label{eq_Z4}
\end{equation}

The overall prediction algorithm is depicted in Fig. \ref{F:1}.   
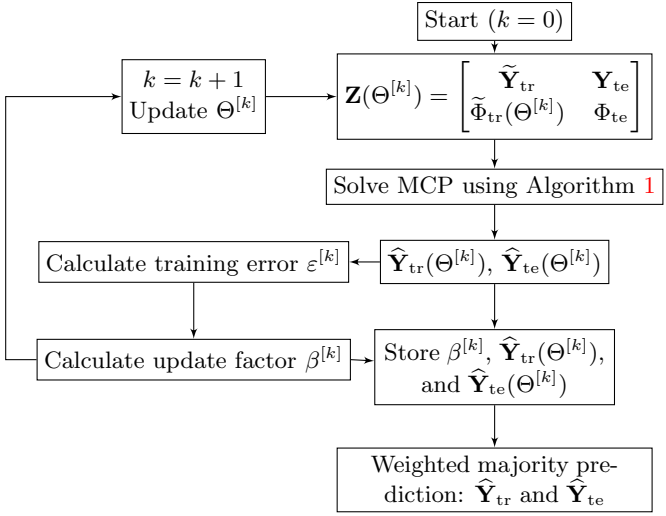
\begin{figure}[htbp]
	\scriptsize
	\tikzstyle{format}=[rectangle,draw,thin,fill=white,align=center]
	\begin{center}
		\begin{tikzpicture}[node distance=1cm,
			auto,>=latex',
			thin,
			start chain=going below,
			every join/.style={norm},]
			\node[format](A){\small Start ($k = 0$)};
			\node[format,below of=A] (n0){\small $\ZZ(\Theta^{[k]}) = \begin{bmatrix}
					\widetilde{\YY}_{\tr} & \YY_{\te} \\ 
					\widetilde{\Phi}_{\tr}(\Theta^{[k]}) & \Phi_{\te}
				\end{bmatrix}$};
			\draw[->] (A.south) -- (n0);
			\node[format,below of=n0,yshift=-.2cm] (n2){\small Solve MCP using Algorithm \ref{A:0}};
			\node[format,below of=n2] (n3){\small $\widehat{\YY}_{\tr}(\Theta^{[k]})$, $\widehat{\YY}_{\te}(\Theta^{[k]})$};
			\draw[->] (n0.south) -- (n2);
			\draw[->] (n2.south) -- (n3);
			\node[format,left of=n3, xshift=-3cm] (n4){\small Calculate training error $\varepsilon^{[k]}$};
			\draw[->] (n3.west) -- (n4);
			\node[format,below of=n4,yshift=-.3cm] (n5){\small Calculate update factor $\beta^{[k]}$};
			\draw[->] (n4.south) -- (n5);
			\node[format,left of=n0, xshift=-3cm] (n6){\small $\begin{matrix} k = k+1 \\ \mathrm{Update}~ \Theta^{[k]} \end{matrix}$};
			\draw[->] (n5.west) -- (-6.5,-4.5) -- (-6.5,-1) -- (n6);
			\draw[->] (n6.east) -- (n0);
			\node[format,below of=n3,text width=3cm,yshift=-.4cm](n7){\small Store $\beta^{[k]}$, $\widehat{\YY}_{\tr}(\Theta^{[k]})$, and $\widehat{\YY}_{\te}(\Theta^{[k]})$};;
			\draw[->] (n3.south) -- (n7);
			\draw[->] (n5.east) -- (n7);
			\node[format,below of=n7,yshift=-.5cm,text width=4cm](n8){\small Weighted majority prediction: $\widehat{\YY}_{\tr}$ and $\widehat{\YY}_{\te}$};
			\draw[->] (n7.south) -- (n8);
		\end{tikzpicture}
	\end{center}
	\caption{An Illustration of the Overall Prediction Algorithm with Adaptive Boosting}\label{F:1}
\end{figure}
The procedure begins with an initial set of weights, which can chosen in a variety of ways, e.g., equal weights: $\theta^{[0]}(t,d) \propto 1$ for all $t \in \{1,\hdots,T_{\tr}\}$ and $d \in \DD$, or weights that favor more recent days: $\theta^{[0]}(t,d) \propto d^{-1}$.  (In our experiments, we use the former.)  We then solve the matrix completion problem using Algorithm \ref{A:0} and use the results to calculate the normalized training error, which is given as
\begin{equation}
	\varepsilon^{[k]} = \frac{\sum_{(t,d)} \theta^{[k]}(t,d) \mathbbm{1}\{ \widehat{\mathbf{y}}_{d,\tr}^{[k]}(t) \ne \mathbf{y}_{d,\tr}(t) \}}{\sum_{(t,d)} \theta^{[k]}(t,d)}, \label{eq_error}
\end{equation}
where $\widehat{\mathbf{y}}_{d,\tr}^{[k]}(t)$ is the training output produced by Algorithm \ref{A:0} using the weights determined in iteration $k$, $\Theta^{[k]}$.  This error metric is bounded between 0 and 1, $\varepsilon^{[k]} = 0$ indicates that $\widehat{\mathbf{y}}_{d,\tr}^{[k]}(t) = \mathbf{y}_{d,\tr}(t)$ for all $(t,d)$ pairs and, at the other extreme, $\varepsilon^{[k]} = 1$ indicates that $\widehat{\mathbf{y}}_{d,\tr}^{[k]}(t) = \mathbf{y}_{d,\tr}(t)$ for none of the $(t,d)$ pairs.  The normalized error is used to calculate an \emph{update factor} as follows:
\begin{equation}
	\beta^{[k]} = \log \frac{1 - \varepsilon^{[k]}}{\varepsilon^{[k]}},
	\label{eq_upFactor}
\end{equation}
which is used to update the weight as
\begin{equation}
	\theta^{[k+1]}(t,d) = \theta^{[k]}(t,d) \exp \big(\beta^{[k]} \mathbbm{1}\{ \widehat{\mathbf{y}}_{d,\tr}^{[k]}(t) \ne \mathbf{y}_{d,\tr}(t) \} \big).\label{eq_weightUpdate}
\end{equation}
The logarithm in \eqref{eq_upFactor} is used to mitigate potential computational instabilities due to large values.  As $\varepsilon^{[k]} \rightarrow 0$, $\beta^{[k]} \rightarrow \infty$ and the limit corresponds to the case where $\mathbbm{1}\{ \widehat{\mathbf{y}}_{d,\tr}^{[k]}(t) \ne \mathbf{y}_{d,\tr}(t) \} = 0$ for all $(t,d)$ pairs (i.e., a perfect match) so that the exponential function on the right-hand side of \eqref{eq_weightUpdate} is 1.  This means that in the case of a perfect match, the weights do not change.  On the other hand, as $\varepsilon^{[k]} \rightarrow 1$, $\beta^{[k]} \rightarrow -\infty$ and $\theta^{[k+1]}(t,d) \rightarrow 0$ in this case.

Finally, the test predictions produced over the $K$ iterations are combined to produce the final prediction. Let $\alpha^{[k]}$ denote the weight assigned to the prediction produced in iteration $k$ and let it be defined as follows:
\begin{equation}
	\alpha^{[k]} \equiv \frac{\beta^{[k]}}{\sum_{j = 1}^K \beta^{[j]}}.
\end{equation}
The combined predictions are given by
\begin{equation}
	\widehat{\YY}_{\tr} = \sum_{k} \alpha^{[k]} \widehat{\YY}_{\tr}(\Theta^{[k]})
	\label{eq_weightedMajority_tr}
\end{equation}
and
\begin{equation}
	\widehat{\YY}_{\te} = \sum_{k} \alpha^{[k]} \widehat{\YY}_{\te}(\Theta^{[k]}).
	\label{eq_weightedMajority_te}
\end{equation}
This is followed by a thresholding step to translate the predictions to labels in $\{0,1\}$.  The prediction algorithm is summarized in Algorithm \ref{A:1} below.  Theorem \ref{thm_CR_Boosting} provides a bound on the training error (expressed as the number of mis-matched columns in $\widehat{\YY}_{\tr}$).  The error bound given in the theorem illustrates the fast reduction in error with number of iterations.  The bound we give is a specialization of the well known result in \cite[Theorem 6]{freund1997decision} to our context.

\begin{algorithm}[h!]
	\caption{Prediction with Ensemble Learning}
	\label{A:1}
	\KwData{Joint matrix $\mathbf{Z}(\Theta^{[0]})$ and $K$ (maximum number of iterations)}
	\KwResult{$\widehat{\YY}_{\tr}$ and $\widehat{\YY}_{\te}$}
	
	\textbf{Initialize}: $k \mapsfrom 0$ and initial weights $\Theta^{[0]} \equiv [\Theta^{[0]}_{|\DD|} \cdots \Theta^{[0]}_1]$
	
	\While{$k < K$}{
		$\widetilde{\Phi}_{\tr}(\Theta^{[k]}) \mapsfrom [\Theta_{|\DD|}^{[k]} \odot \Phi(\XX_{|\DD|,\tr})  \cdots \Theta_1^{[k]} \odot \Phi(\XX_{1,\tr}) ]$ 
		
		Calculate $\widehat{\YY}_{\tr}(\Theta^{[k]})$ and $\widehat{\YY}_{\te}(\Theta^{[k]})$ using Algorithm \ref{A:0} 
		
		Calculate $\varepsilon^{[k]}$ and $\beta^{[k]}$ using \eqref{eq_error} and \eqref{eq_upFactor}, respectively 
		
		$k \mapsfrom k + 1$ and update $\Theta^{[k]}$ using \eqref{eq_weightUpdate} 
	}
	
	Calculate the weighted majority predictions using \eqref{eq_weightedMajority_tr}-\eqref{eq_weightedMajority_te} and apply thresholding (Algorithm \ref{A:2}).
\end{algorithm}

\begin{theorem}[Training Error of Algorithm \ref{A:1}]
	\label{thm_CR_Boosting}
	Assume that $\theta^{[0]}(t,d) = 1$ for all $(t,d)$-pairs and let $\boldsymbol{\tau} \in \RR^n$ denote a vector of thresholds.  Let 
	\begin{equation}
		\epsilon(K) \equiv \frac{\big| \{ (t,d): \widehat{\mathbf{y}}_{d,\tr}^{[K]}(t) \ne \mathbf{y}_{d,\tr}(t) \} \big|}{|\DD|T_{\tr}} \label{trainingError}
	\end{equation}
	denote the training sample error at the end of step $K$.  Then
	\begin{equation}
		\epsilon(K) \le 2^K  \prod_{k=0}^K \big(1 - \varepsilon^{[k]}\big)^{\frac{n - \|\boldsymbol{\tau}\|_1}{n}} \big(\varepsilon^{[k]} \big)^{^{\frac{\|\boldsymbol{\tau}\|_1}{n}}}. \label{trainingErrorBound}
	\end{equation}
\end{theorem}

\begin{proof}
Let $(t,d)$ be such that $\mathbbm{1}\{ \widehat{\mathbf{y}}_{d,\tr}(t) \ne  \mathbf{y}_{d,\tr}(t)\} = 1$.  Since $\sum_{k=0}^K \alpha^{[k]} = 1$ we have that $\mathbf{y}_{d,\tr}(t) = \sum_{k=0}^K \alpha^{[k]} \mathbf{y}_{d,\tr}(t)$, thus
\begin{equation}
	\sum_{k=0}^K \alpha^{[k]} \big\| \widehat{\mathbf{y}}_{d,\tr}^{[k]}(t)  - \mathbf{y}_{d,\tr}(t) \big\|_1 \ge  \|\boldsymbol{\tau}\|_1.
\end{equation}
Multiplying both sides by $n^{-1}\sum_{j=0}^K \beta^{[j]}$, we get the inequality
\begin{equation}
	\sum_{k=0}^K \frac{\beta^{[k]}}{n} \big\| \widehat{\mathbf{y}}_{d,\tr}^{[k]}(t)  - \mathbf{y}_{d,\tr}(t) \big\|_1 \ge  \frac{1}{n}\sum_{j=0}^K \beta^{[j]}\|\boldsymbol{\tau}\|_1.
\end{equation}
Since $\mathbbm{1}\{ \widehat{\mathbf{y}}_{d,\tr}^{[k]}(t) \ne \mathbf{y}_{d,\tr}(t) \} \ge n^{-1} \big\| \widehat{\mathbf{y}}_{d,\tr}^{[k]}(t)  - \mathbf{y}_{d,\tr}(t) \big\|_1$ for all $k$, we have that
\begin{equation}
	\sum_{k=0}^K \beta^{[k]} \mathbbm{1}\{ \widehat{\mathbf{y}}_{d,\tr}^{[k]}(t) \ne \mathbf{y}_{d,\tr}(t) \} \ge \frac{1}{n}\sum_{j=0}^K \beta^{[j]}\|\boldsymbol{\tau}\|_1. \label{eq_in1}
\end{equation}
Define $\mathcal{M}^{[k]} \equiv \{(t,d): \widehat{\mathbf{y}}_{d,\tr}^{[k]}(t) \ne \mathbf{y}_{d,\tr}(t)\}$. It follows from \eqref{eq_in1} that
\begin{multline}
	\sum_{(t,d)} \theta^{[K+1]}(t,d) \\
	\ge \sum_{(t,d) \in \mathcal{M}^{[K]}} \theta^{[0]}(t,d) \exp \Big(\sum_{k=0}^K \beta^{[k]} n^{-1} \|\boldsymbol{\tau}\|_1 \Big) \\
	= |\DD|T_{\tr} \epsilon(K) \prod_{k=0}^K \exp(\beta^{[k]} n^{-1} \|\boldsymbol{\tau}\|_1), \label{eq_in2}
\end{multline}
where $|\DD|T_{\tr} \epsilon(K) = \sum_{(t,d) \in \mathcal{M}^{[K]}} \theta^{[0]}(t,d)$ follows from the initialization assumption. Next, we have that
\begin{multline}
	\sum_{(t,d)} \theta^{[K+1]}(t,d) \\ = \sum_{(t,d)} \theta^{[K]}(t,d) \Big(\frac{1-\varepsilon^{[K]}}{\varepsilon^{[K]}}\Big)^{\mathbbm{1}\{ \widehat{\mathbf{y}}_{d,\tr}^{[K]}(t) \ne \mathbf{y}_{d,\tr}(t) \}}
	\\ = \sum_{(t,d)} \theta^{[K]}(t,d) \bigg(1 - \Big(1 - \frac{1-\varepsilon^{[K]}}{\varepsilon^{[K]}}\Big)  \mathbbm{1}\{ \widehat{\mathbf{y}}_{d,\tr}^{[K]}(t) \ne \mathbf{y}_{d,\tr}(t) \} \bigg)
	\\ = \sum_{(t,d)}  \theta^{[K]}(t,d) - \sum_{(t,d)}  \theta^{[K]}(t,d) \varepsilon^{[K]} \Big( 1 - \frac{1-\varepsilon^{[K]}}{\varepsilon^{[K]}} \Big) \\
	= \sum_{(t,d)}  2 \theta^{[K]}(t,d) (1 - \varepsilon^{[K]}),
\end{multline}
where the second to last equality follows from \eqref{eq_error}.  This implies that
\begin{equation}
	\sum_{(t,d)} \theta^{[K+1]}(t,d) = 2^K |\DD|T_{\tr} \prod_{k=0}^K (1 - \varepsilon^{[k]}). \label{eq_in3}
\end{equation}
Combining \eqref{eq_in2} with \eqref{eq_in3} and utilizing the definition of $\beta^{[k]}$, \eqref{eq_upFactor}, completes the proof.
\end{proof}

\subsection{Generalization and Training Sample Complexity}
Theorem~\ref{thm_CR_Boosting} provides a bound on the \emph{in-sample} error.  To study how this generalizes to \emph{out-of-sample} data, we can think of $(\widetilde{\Phi}_{\tr}(\Theta),\widetilde{\YY}_{\tr})$ as random samples of sensor states.  The in-sample error $\epsilon(K)$ is representative of the true error insofar as the sample $(\widetilde{\Phi}_{\tr}(\Theta),\widetilde{\YY}_{\tr})$ is representative of the true underlying phenomenon that generates the data. Let $\epsilon_{\mathrm{true}}$ denote the true error, i.e., over the entire (unknown) data generating process (not just the sample).  A fundamental result by \cite{vapniknature} states that, for any tolerance threshold $0<\delta<1$, we have with probability $1-\delta$ that
\begin{equation}
	\epsilon_{\mathrm{true}} \le \epsilon(K) + \sqrt{ \frac{ \mathsf{d}_{\mathrm{VC}}^{\mathrm{AB}} \big(\log_e(2|\DD|T_{\tr} / \mathsf{d}_{\mathrm{VC}}^{\mathrm{AB}}) + 1\big) - \log_e(\delta / 4) }{|\DD|T_{\tr}} },
\end{equation}
where $\mathsf{d}_{\mathrm{VC}}^{\mathrm{AB}}$ is the \emph{Vapnik-Cheronenkis (VC) dimension} of the set of solutions of the adaptive boosting algorithm (Algorithm~\ref{A:1}).  The second term (under the square root) can be made arbitrarily small by appropriate choice of the ``sample size'' $|\DD|T_{\tr}$ (provided that $\mathsf{d}_{\mathrm{VC}}^{\mathrm{AB}}$ is finite).  This says that the true error $\epsilon_{\mathrm{true}}$ can be made as arbitrarily small as the training error $\epsilon(K)$, which can be made arbitrarily small by increasing $K$ in accord with Theorem~\ref{thm_CR_Boosting}.  A similar bound can be stated for the matrix completion problem providing guidance into the selection of $T_{\tr}$.  Let $\mathsf{d}_{\mathrm{VC}}^{\mathrm{BCD}}$ denote the VC-dimension of the matrix completion problem.  Theorem 8 in \citep{freund1997decision} demonstrates how $\mathsf{d}_{\mathrm{VC}}^{\mathrm{AB}}$ can be determined from $\mathsf{d}_{\mathrm{VC}}^{\mathrm{BCD}}$.  Determination of the VC-dimension of our matrix completion problem, $\mathsf{d}_{\mathrm{VC}}^{\mathrm{BCD}}$, is beyond the scope of this paper.  However, we note that it has been demonstrated that $\mathsf{d}_{\mathrm{VC}}^{\mathrm{BCD}}$ is finite for matrix completion problems that aim to minimize rank; we refer to \citep{srebro2004generalization} for deeper analytical insights into generalization errors for low-rank matrix completion problems.  In our experiments, we compare training and testing errors experimentally using a real-world dataset.

\subsection{Time Complexity Analysis}
In each iteration, the block-coordinate descent algorithm (Algorithm \ref{A:0}) solves four least-squared (LS) problems but the solutions are given in closed form. The complexity of calculating $\widehat{\UU}_{\tr}^{[k]}$, using \eqref{eq_updateUtr}, is $O(r^2T_{\tr})$ (the complexity of matrix multiplication and inversion of a symmetric positive definite matrix).  The complexity of updating $\widehat{\VV}_{\tr}^{[k]}$, using \eqref{eq_updateVtr}, is $O(\max\{r^3,T_{\tr}^2r,T_{\tr}T_{\te}r\})$, which is typically equal to $O(T_{\tr}^2r)$ as $T_{\tr}>T_{\te}>r$ will be the case in most settings. Similarly, updating $\widehat{\VV}_{\te}^{[k]}$ has a time complexity of $O(\max\{r^3,T_{\te}^2r,T_{\tr}T_{\te}r\})$, which is typically $O(T_{\tr}T_{\te}r)$. Therefore, the complexity in each iteration of Algorithm \ref{A:0} is $O(\max\{r^3,T_{\tr}^2r,T_{\te}^2r,T_{\tr}T_{\te}r\}) = O(T_{\tr}^2r)$. Let $K_{\mathrm{BCD}}$ denote a the number of block-coordinate descent iterations, which can be determined a priori in accord with Theorem \ref{thm_SCR}. Then, the overall complexity of Algorithm \ref{A:0} is $O(K_{\mathrm{BCD}}T_{\tr}^2r)$. 

The analysis above implies that the time complexity of a single iteration of Algorithm \ref{A:1} is $O(K_{\mathrm{BCD}}|\DD|^2T_{\tr}^2r)$.  The time complexity of the soft-thresholding algorithm is $O(|\DD|^2T_{\tr}^2n)$.  The training error bound given in Theorem \ref{thm_CR_Boosting} can be used to estimate a number of iterations needed to get the training error to within a pre-specified error bound.  Let $K_{\mathrm{AB}}$ be the number of iterations of Algorithm \ref{A:1} to be performed.  The overall complexity of our proposed approach is then $O(K_{\mathrm{AB}}K_{\mathrm{BCD}}T_{\tr}^2r + |\DD|^2T_{\tr}^2n)$. In most cases, we will have that $r > n$ so that the total complexity can be simply be stated as $O(K_{\mathrm{AB}}K_{\mathrm{BCD}}T_{\tr}^2r)$.  One can further reduce the computational complexity of Algorithm \ref{A:0} by utilizing more efficient matrix multiplication and inversion techniques, our time complexity bounds assume standard matrix algebra techniques are used.

\section{Experimental Results}
\label{S:experiments}
\subsection{Network Descriptions and Performance Metrics} \label{s_networks}
\emph{Simulated Data: } We test the proposed methods using both reproducible toy examples with simulated data and a real-work high-resolution dataset.  For the former, we generated synthetic data using a microscopic simulation model of a hypothetical network of four 4-leg intersections, depicted in Fig.~\ref{F:net}.  The model was developed using the open-source simulator SUMO (Simulation of Urban MObility) \citep{krajzewicz2012recent}. 
\begin{figure}[htbp]
	\centerline{\includegraphics[scale=0.55]{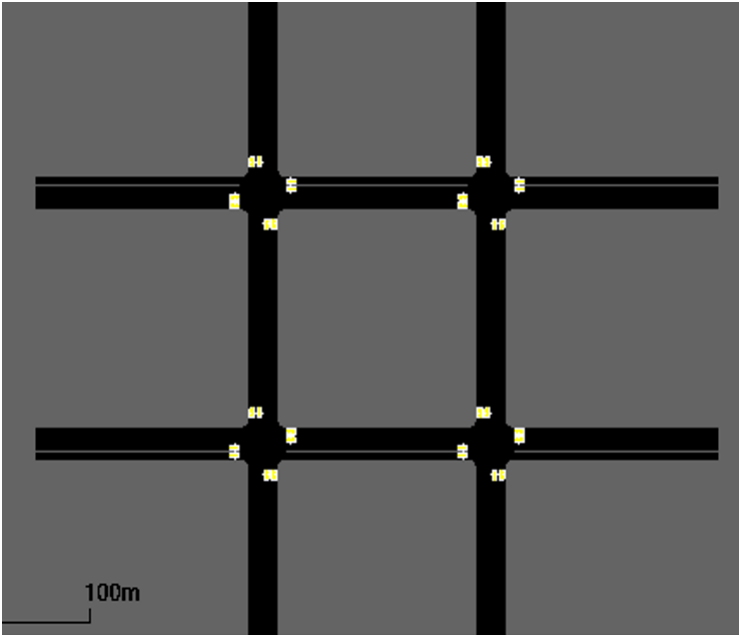}}
	\caption{Layout of the Synthetic Network} \label{F:net}		
\end{figure}
The intersections are operated using a fixed-time controller with a cycle length of 90 seconds.  All approaches are 200 meters long and have two lanes with point sensors capable of recording high-resolution data located at the stop lines. We, thus, have 32 sensors in total. The simulation time horizon is 2 days (172,000 seconds). The average occupancy of these 32 sensors over the two day period is around 15\%, hourly averaged occupancy profiles of two of the sensors are shown in Fig.~\ref{F:occ}.  The experiments and results are presented in Sec.~\ref{ss_SUMO}.
\begin{figure}[h!]
	\centerline{\includegraphics[scale=.80]{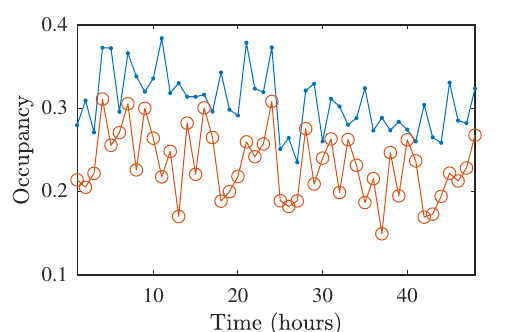}}
	\caption{Occupancy Profiles of Two Sensors} \label{F:occ}		
\end{figure}

\emph{Real-World Dataset:} The second dataset is a real-world dataset provided by the Abu Dhabi Department of Transportation.  The dataset was obtained for Al Zahiyah in downtown Abu Dhabi and consists of two corridors with eleven intersections as shown in Fig.~\ref{F:4}.  
\begin{figure}[h]
	\centerline{\includegraphics[scale=0.25]{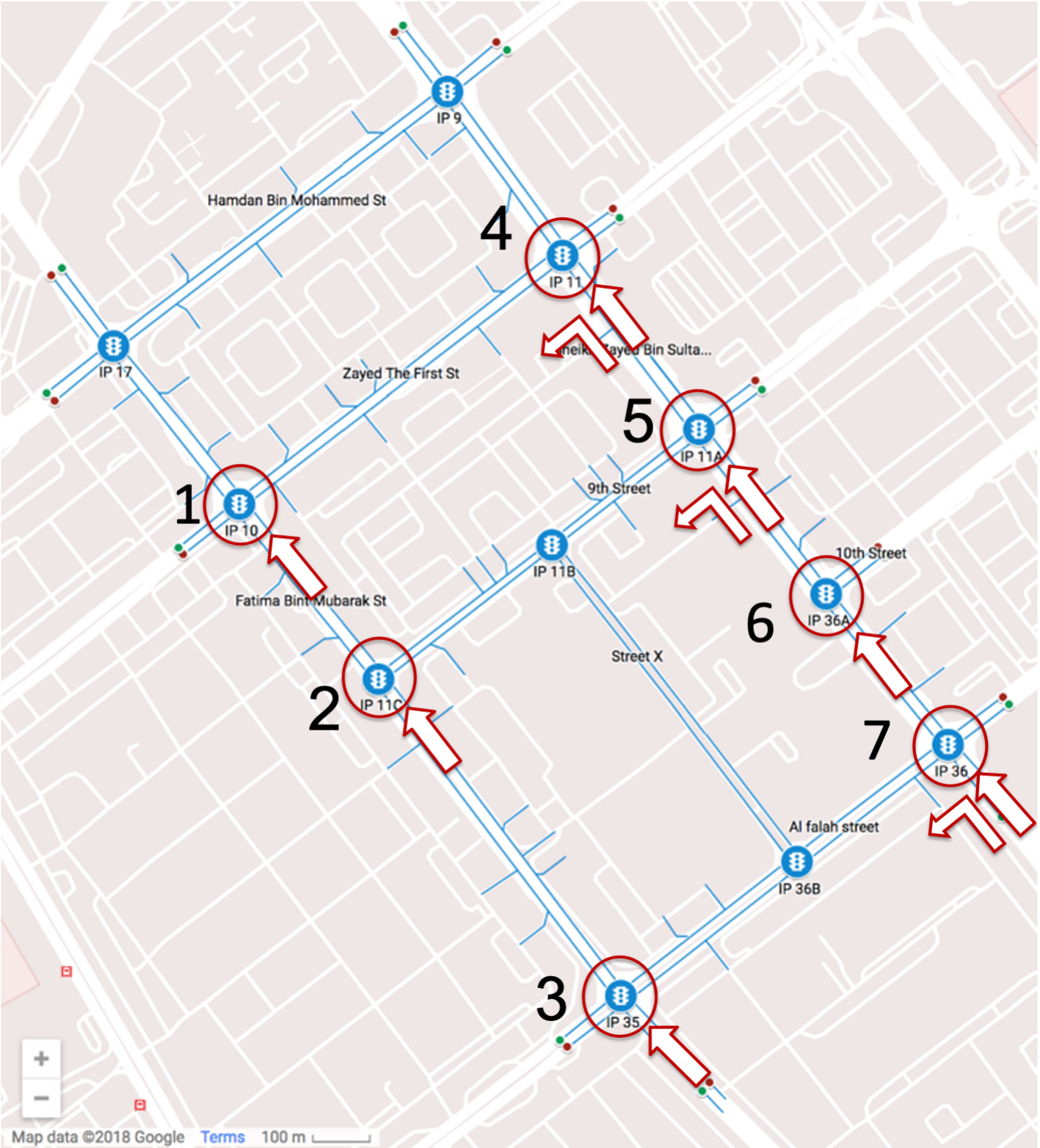}}
	\caption{Layout of the Downtown Abu Dhabi Network}		
	\label{F:4}
\end{figure}
These intersections, currently operated by a commercial adaptive controller, are located in the central business district in the city and are known to have heavy queuing during peak demand periods.  For our experiments, we utilize second-by-second data obtained from point sensors along the directions that are highlighted in the figure for seven adjacent intersections. The sensors used in this study are those that the controller uses to make signal extension decisions. 
Each of the intersections has three or four through lanes in each direction with one sensor in each lane.  Each of the intersections also has two left-turn lanes approaching the stop-line.  The sensor stations used are located nearer the stop-lines of the intersections for intersections 4, 5, and 7, and the datasets for these sensors include the left-turn lanes in our experiments (as illustrated in Fig.~\ref{F:4}).  The total number of sensors in these experiments is 31. A schematic of one of the intersections is provided in Fig.~\ref{F:IP11}.  The sensors used are those in the advanced position (not the stop-line sensors). 
\begin{figure}[h]
	\centerline{\includegraphics[scale=0.65]{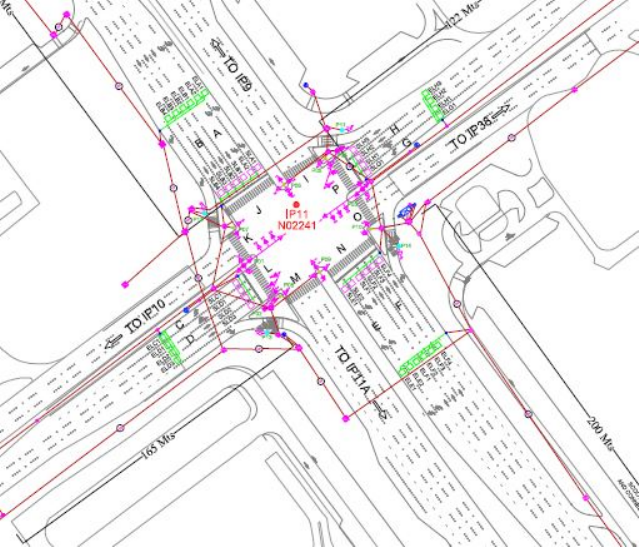}}
	\caption{A Schematic of One of the Intersections Used in the Experiment}
	\label{F:IP11}		
\end{figure}

\textit{Performance Metrics}: We employ two performance metrics in our comparisons, the first is a traditional mean absolute error (MAE):
\begin{equation}
	\epsilon_{\mathrm{MAE}} = \frac{1}{nT_{\te}} \sum_{t=1}^{T_{\te}} \big\| \widehat{\boldsymbol{y}}_{\te}(t) - \boldsymbol{y}_{\te}(t) \big\|_1,
\end{equation}
where $\widehat{\boldsymbol{y}}_{\te}(t)$ and $\boldsymbol{y}_{\te}(t)$ are columns $t \in \{1,\hdots,T_{\te}\}$ of $\widehat{\YY}_{\te}$ and $\YY_{\te}$, respectively.  The second metric is the Skorokhod $\mathrm{M}_1$ metric:
\begin{equation}
	d_{\mathrm{M}_1}(\widehat{\YY}_{\te},\YY_{\te})\equiv \underset{\substack{(\widehat{Y},\widehat{y}) \in \Pi(\widehat{\YY}_{\te}), \\ (Y,y) \in \Pi(\YY_{\te})}}{\textrm{inf}} \max \big\{\| \widehat{Y} - Y \|_{\infty}, \| \widehat{y} - y \|_{\infty} \big\},
	\label{eq:21}
\end{equation}
where $\|\cdot\|_{\infty}$ is the uniform norm and $\Pi(\YY)$ is the set of parametric representations of the rows of  $\YY$. The Skorokhod $\mathrm{M}_1$ metric is a particularly suitable metric for processes with jumps, more specifically, it allows for comparisons between processes with jumps and those free of jumps.  We refer to \citep{whitt2002stochastic} for a more detailed description of the metric.  In our context, it is chosen for its ability to compare high-resolution signals. 

Finally, we set $\mu=0.01$ unless otherwise specified and run all tests on a 2.7 GHz Intel Core i7 Processor with 16 GB of RAM.

\subsection{Small Illustrative Example}
\label{smallex}
Our first example utilizes a single sensor station form the real-world dataset to illustrate the proposed techniques.  The three sensors located at intersection 2 in Fig.~\ref{F:4} are used in this example, and we denote the occupancies of the three sensors at time $t$ by $s_1(t)$, $s_2(t)$, and $s_3(t)$ for lanes 1, 2, and 3, respectively.  We set the lag to $L=3$ seconds and the prediction horizon to $H=2$ seconds.  We wish to predict the occupancies of sensor 2 in time steps $t+12, \hdots,t+16$ utilizing past data from all three sensors.  The structure of the joint matrix $\ZZ$ in \eqref{eq_Z1} (without kernels) for this problem is
\begin{widetext}
\begin{equation}
	\ZZ = \begin{bmatrix}
		\YY_{\tr} = \begin{bmatrix} s_2(t+2)&\cdots & s_2(t+11) \end{bmatrix} & \YY_{\te}= \begin{bmatrix}s_2(t+12)&\cdots & s_2(t+16)  \end{bmatrix}\\
		\XX_{\tr} = \begin{bmatrix} s_3(t)&\cdots & s_3(t+9) \\  s_2(t)&\cdots & s_2(t+9) \\s_1(t)&\cdots & s_1(t+9) \\
			s_3(t-1)&\cdots & s_3(t+8) \\  s_2(t-1)&\cdots & s_2(t+8) \\s_1(t-1)&\cdots & s_1(t+8)\\
			s_3(t-2)&\cdots & s_3(t+7) \\  s_2(t-2)&\cdots & s_2(t+7) \\s_1(t-2)&\cdots & s_1(t+7)\end{bmatrix} 
		& \XX_{\te} = \begin{bmatrix} s_3(t+10)&\cdots & s_3(t+14) \\  s_2(t+10)&\cdots & s_2(t+14) \\s_1(t+10)&\cdots & s_1(t+14) \\
			s_3(t+9)&\cdots & s_3(t+14) \\  s_2(t+9)&\cdots & s_2(t+13) \\s_1(t+9)&\cdots & s_1(t+13)\\
			s_3(t+8)&\cdots & s_3(t+12) \\  s_2(t+8)&\cdots & s_2(t+12) \\s_1(t+8)&\cdots & s_1(t+12)\end{bmatrix}\\
	\end{bmatrix}.
	\label{eq_Smallex}
\end{equation}
\end{widetext}
The entries of this matrix are provided in Fig.~\ref{fig_Smallex}, and color coded according to whether the occupancy is 0 or 1.  
\begin{figure}[h!]
	\centerline{\includegraphics[width=0.45\textwidth]{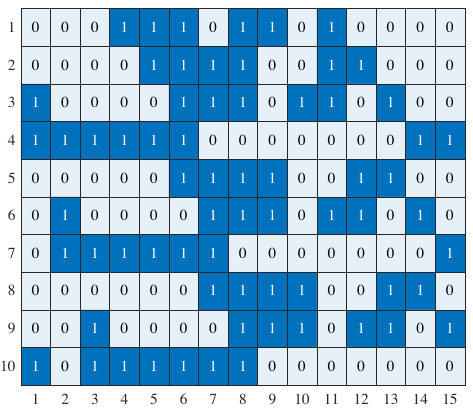}}
	\caption{Entries of $\ZZ$ for the Small Illustrative Example}
	\label{fig_Smallex}		
\end{figure}
Note the difference in the occupancy profiles corresponding to sensors 1 and 2 ($s_1$ and $s_2$) from times $t-1$ to $t+13$ (rows 6 and 7 in Fig.~\ref{fig_Smallex}). Over the entire 15 second period, the aggregated occupancies are the same but the traffic patters are different.

To solve the prediction problem, we set $\mu \equiv 10^{-6}$ and $r=1$ and list the training results of the (unkernalized) problem without adaptive boosting in Table~\ref{t_PROI} and those with adaptive boosting in Table~\ref{t_boost_te}.  The corresponding testing results are summarized in Tables \ref{t_PROI_b} and \ref{t_boost_tr}, respectively. The convergence criterion used is $\epsilon^{\mathrm{rel}}(k) < 10^{-4}$, where $\epsilon^{\mathrm{rel}}(k)$ is the relative error in iteration $k$, calculated as
\begin{multline}
	\epsilon^{\mathrm{rel}}(k) \equiv \max \bigg\{ \frac{\|\UU_{\tr}^{[k]} - \UU_{\tr}^{[k-1]}\|_{\FF}}{\| \UU_{\tr}^{[k-1]} \|_{\FF}}, \frac{\|\VV_{\tr}^{[k]} - \VV_{\tr}^{[k-1]}\|_{\FF}}{\| \VV_{\tr}^{[k-1]}  \|_{\FF}} \\, \frac{\|\VV_{\te}^{[k]} - \VV_{\te}^{[k-1]}\|_{\FF}}{\| \VV_{\te}^{[k-1]} \|_{\FF}} \bigg\}.
\end{multline}

\begin{table*}[h!]
	\caption{Training Results Without Adaptive Boosting (G.T. = Ground Truth)}
	\small
	\centering
	\begin{tabular}{|c|c|cccccccccc|}
		\hline
		& $k$ & $t+2$ & $t+3$ & $t+4$ & $t+5$ & $t+6$ & $t+7$ & $t+8$ & $t+9$ & $t+10$ & $t+11$ \\ \hline 
		\multirow{4}{*}{$\widehat{\YY}_{\tr}^{[k]}$} & 1 &  -0.0915 &   0.0884   & 0.1208  &  0.8264   & 0.8682 &   0.7644 &  -0.1001  &  0.5218  &  0.8142 &   0.0580
		\\
		& 5&  0.0209   & 0.0535  &  0.1454 &   1.0242   & 1.0542  &  0.9922  &  0.1287  &  0.9451  &  0.9304   & 0.0597  \\
		& 10 &    0.0208  &  0.0531  &  0.1439 &   1.0224   & 1.0520   & 0.9908   & 0.1278  &  0.9452   & 0.9330    &0.0662
		\\
		& end &0.0205&    0.0523&    0.1421 &   {1.0218} &   {1.0513}  &  {0.9908}   & 0.1261  &  {0.9464}  & { 0.9354} &   0.0686
		\\ \hline
		&G.T.&     0   & 0   & 0  & 1   & 1  &  1  &  0  & 1  &  1 &  0\\
		\hline
	\end{tabular}
	\label{t_PROI}
\end{table*}

\begin{table*}[ht!]
	\caption{Training Results With Adaptive Boosting (A.T. = After Thresholding, G.T. = Ground Truth)}
	\small
	\centering
	\begin{tabular}{|c|c|cccccccccc|}
		\hline
		& $k$ & $t+2$ & $t+3$ & $t+4$ & $t+5$ & $t+6$ & $t+7$ & $t+8$ & $t+9$ & $t+10$ & $t+11$ \\ \hline 
		\multirow{4}{*}{$\widehat{\YY}_{\tr}(\Theta^{[k]})$} & 1 &  0.0205&    0.0523&    0.1421 &   1.0218 &   1.0513 &  0.9908  & 0.1261  &  0.9464  & 0.9354 &   0.0686
		\\
		& 2&  0.0222&    0.0568 &   0.1545&   1.0192 & 1.0482 & 0.9865 &    0.1408&    0.9469 &    0.9336 &    0.0765  \\
		& 3 &    0.0247  &  0.0635  &  0.1655   & 1.0136  &  1.0432   & 0.9798  &  0.1804  &  0.9486  &  0.9324   & 0.0900
		\\
		& 4 & 0.0272  &  0.0711&    0.1670  &  1.0045  &  1.0371    &0.9715  &  0.2112  &  0.9527  &  0.9349  &  0.1104
		\\ \hline
		$\widehat{\YY}_{\tr}$ & \eqref{eq_weightedMajority_tr} & 0.0229   & 0.0587   & 0.1531   & 1.0169   & 1.0466  &  0.9844  &  0.1581   & 0.9479  &  0.9341 &   0.0814
		\\ \hline
		&A.T.&     0   & 0   & 0  & 1   & 1  &  1  &  0  & 1  &  1 &  0\\
		\hline
		&G.T.&     0   & 0   & 0  & 1   & 1  &  1  &  0  & 1  &  1 &  0\\
		\hline
	\end{tabular}
	\label{t_boost_te}
\end{table*}

\begin{table}[h!]
	\caption{Testing Results Without Adaptive Boosting (A.T. = After Thresholding, G.T. = Ground Truth)}
	\small
	\centering
	\begin{tabular}{|c|c|ccccc|}
		\hline
		& $k$ & $t+12$ & $t+13$ & $t+14$ & $t+15$ & $t+16$ \\ \hline 
		\multirow{4}{*}{$\widehat{\YY}_{\te}^{[k]}$} & 1 & 0.4694 & 0.4454 & 0.4194 & 0.2675 & 0.2434
		\\
		& 5& 0.1098 & 0.1289 & 0.1143 & 0.0635 & 0.0770  \\
		& 10 & 0.2360 & 0.2351  & 0.1254  & 0.0651 & 0.0391
		\\
		& end & 0.2613 & 0.2626 & 0.1418 & 0.0730 & 0.0407
		\\ \hline
		&A.T. & 1   & 1  &  0  & 0  &  0
		\\ \hline
		&G.T.  & 1   & 0 &  0  & 0  &  0 
		\\ \hline
	\end{tabular}
	\label{t_PROI_b}
\end{table}

Note from Table~\ref{t_PROI} that beyond iteration $k=5$, changes in the results are small.  This illustrates the fast convergence rate of the algorithm. Convergence of the Block-Coordinate Descent (BCD) algorithm is illustrated in Fig.~\ref{f_conv1} for this example.
\begin{figure}[htbp]
	\centerline{\includegraphics[scale=0.8]{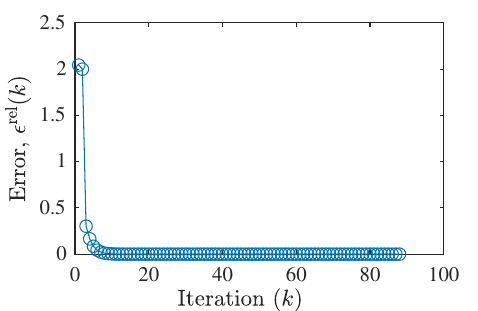}}
	\caption{Convergence of BCD}
	\label{f_conv1}		
\end{figure}
We also see that the algorithm misclassified the sensor state at time $t+13$ in the testing sample in Table~\ref{t_PROI_b}.  The boosting procedure corrects this as depicted in Tables \ref{t_boost_te} and \ref{t_boost_tr} below.  Note that the iterations in the tables below are the \emph{outer} iterations of the boosting procedure: the first outer iteration produces the results observed above (without boosting).

\begin{table}[h!]
	\caption{Testing Results With Adaptive Boosting (A.T. = After Thresholding, G.T. = Ground Truth)}
	\small
	\centering
	\begin{tabular}{|c|c|ccccc|}
		\hline
		& $k$ & $t+12$ & $t+13$ & $t+14$ & $t+15$ & $t+16$ \\ \hline 
		\multirow{4}{*}{$\widehat{\YY}_{\te}^{[k]}$} & 1 & 0.2613 & 0.2626 & 0.1418  & 0.0730 & 0.0407
		\\
		& 2 &0.2583 & 0.2472 & 0.1293  & 0.0639 & 0.0345  \\
		& 3 & 0.2546 & 0.2246 & 0.1120  & 0.0522 & 0.0268
		\\
		& 4 & 0.2609 & 0.2068 & 0.0984 & 0.0437 & 0.0211
		\\ \hline
		$\widehat{\YY}_{\te}$ & \eqref{eq_weightedMajority_te} & 0.1781 & 0.1597 & 0.0811  & 0.0389 & 0.0204
		\\ \hline
		&A.T. & 1   & 0  &  0  & 0  &  0
		\\ \hline
		&G.T.  & 1   & 0 &  0  & 0  &  0 
		\\ \hline
	\end{tabular}
	\label{t_boost_tr}
\end{table}

\subsection{Simulation Experiments}
\label{ss_SUMO}
\textit{Description of Experiments}: We test the proposed method using the simulated data described in Sec.~\ref{s_networks}.  Specifically, we test the approach under varying horizons and lags, $H \in \{1,10,60,120\}$ and $L \in \{10,30,60,120\}$ seconds. The historical data used for boosting consists of $|\DD| = 2$ consecutive days in each experiment, and the number of training and testing time steps are $T_{\tr} = 540$ and $T_{\te} = 60$ seconds, respectively.

\textit{Impact of Choice of Lag ($L$) and Horizon ($H$)}: Since all entries in both $\widehat{\YY}_{\te}$ and $\YY_{\te}$ are binary, we have that $\epsilon_{\mathrm{MAE}} \in [0,1]$ and $d_{\mathrm{M}_1}(\widehat{\YY}_{\te},\YY_{\te}) \in [0,1]$.  We can hence measure \emph{accuracy} using $1- \epsilon_{\mathrm{MAE}}$ and $1 - d_{\mathrm{M}_1}$, where we have dropped the arguments from the latter where no confusion may arise.  Mean accuracy and standard deviations (calculated over the 32 sensors) are summarized in  Table~\ref{table_2} and Table~\ref{table_3} using different lags and prediction horizons.  The entries in bold are those with the highest accuracy for each prediction horizon.  
\begin{table*}[h!]
	\caption{Testing Accuracy Using $\mathrm{MAE}$ ( $1 - \epsilon_{\mathrm{MAE}}: \mathrm{Mean}\pm \mathrm{Std}$)}
	\small
	\centering
	\begin{tabular}{|c|ccc|}
		\hline
		$|\DD| \times L$ & $H=1$  & $H=10$ & $H=60$  
		\\
		\hline
		{$2\times30$}             &  {0.9395} $\pm$ {0.0340}        & {0.9317} $\pm$ {0.0304}   & {0.9268} $\pm$ {0.0268}    
		\\
		$2\times60$ &\textbf{0.9522} $\pm$ \textbf{0.0211} &\textbf{0.9378} $\pm$ \textbf{0.0313} & \textbf{0.9325} $\pm$ \textbf{0.0125}
		\\
		$2\times120$&0.9448 $\pm$ 0.0311&0.9319 $\pm$ 0.0455&0.9178 $\pm$ 0.0761
		\\
		\hline
	\end{tabular}
	\label{table_2}
\end{table*}
\begin{table*}[h!]
	\caption{Testing Accuracy Using the Skorokhod $\mathrm{M}_1$ Metric ($1 - d_{\mathrm{M}_1}: \mathrm{Mean} \pm \mathrm{Std}$)}
	\small
	\centering
	\begin{tabular}{|c|ccc|}
		\hline 
		$|\DD| \times L$ & $H=1$  & $H=10$ & $H=60$ 
		\\
		\hline
		{$2\times30$}             &  \textbf{0.9557} $\pm$ \textbf{0.0520}        & {0.9259} $\pm$ {0.0429}   & {0.9111} $\pm$ {0.0457}   
		\\
		$2\times60$ &0.9464 $\pm$ {0.0385} &\textbf{0.9378} $\pm$\textbf{0.0313} & \textbf{0.9269} $\pm$ \textbf{0.0381}
		\\
		$2\times120$&0.9387 $\pm$ 0.0401&0.9292 $\pm$ 0.0349&0.9155 $\pm$ 0.0977
		\\
		\hline
	\end{tabular}
	\label{table_3}
\end{table*}
\begin{table*}[h!]
	\caption{Paired $t$-test for Different Lags and Horizons (Significance Level = 0.05)}
	\small
	\centering
	\begin{tabular}{|c|c|ccc|}
		\hline
		& $(H,|\DD| \times L)$    & $(1,2 \times30)$   & $(1,2 \times60)$ & $(1,2 \times120)$ 
		\\
		\hline
		\multirow{3}{*}{$(H,|\DD| \times L)$}& $(1,2 \times30)$               &0  &1 & 1
		\\
		& $(1,2 \times60)$      & 1 &0 & 1
		\\
		& $(1,2 \times120)$     &1  &1 & 0
		\\
		\hline
		& $(H,|\DD| \times L)$  & $(1,2 \times30)$  & $(10,2 \times30)$   & $(60,2 \times30)$  
		\\
		\hline
		\multirow{3}{*}{$(H,|\DD| \times L)$} & $(1,2 \times30)$            &    0     & 1     &1  
		\\
		& $(10,2 \times30)$            & 1       & 0 & 1 
		\\
		& $(60,2 \times30)$  &     1   & 1 &0 
		\\
		\hline
	\end{tabular}
	\label{table_ttest}
\end{table*}

The average testing accuracy in Table~\ref{table_2} is no lower than 91.78\% and reaches 95.22\%, while the greatest standard deviation does not exceed 0.0761. The worst case prediction accuracy (when $|\DD| \times L \times H = 2 \times 120 \times 60$) is greater than 76.9\% with probability 97.5\% (calculated as 0.9178 - 1.96$\times$0.0761 corresponding to an accuracy that is 1.96 standard deviations below the mean).  In other words, 97.5\% of the cases have an accuracy that exceeds 76.9\%.  However, with a proper choice of the lag parameter (in this case $L = 60$ seconds), the worst case accuracy is greater than 87.6\% with probability 97.5\% (when $|\DD| \times L \times H = 2 \times 60 \times 10$). 
We see the same results when measuring accuracy using $1 - d_{\mathrm{M}_1}$ in Table~\ref{table_3}.  The worst case prediction accuracy (when $|\DD| \times L \times H = 2 \times 120 \times 60$) is greater than 72.4\% with probability 97.5\% but with a proper choice of the lag parameter, the worse case accuracy is greater than 85.2\% with probability 97.5\%.
For both accuracy metrics used in Tables \ref{table_2} and \ref{table_3}, we observe that the accuracy tends to decrease as $H$ gets larger.  However, increasing the lag $L$ is not observed to improve the accuracy, a lag of $L = 60$ seconds seems to be the best choice in our experiment.

We further perform paired $t$-tests for different lags and horizons to check whether $H$ and $L$ play a significant role in the resulting predictions.  We summarize the results in Table~\ref{table_ttest}, where $H$ is fixed and $L$ is varied in the top part of the table, and $H$ is varied and $L$ is fixed in the bottom part of the table.  
The test result is either 1 or 0 indicating whether a hypothesis that the two models being compared are the same can be rejected (1) or not (0) at a 0.05 significance level.  The results consistently suggest that the difference is significant in all cases when either $L$ is varied or $H$ is.

We provide an illustration of the resulting predictions for all 32 sensors in Figures \ref{ex} and \ref{npeak}, for a high-demand scenario and a low demand scenario, respectively.  
\begin{figure}[h!]
	\centering
	\subfigure[Results Before Thresholding]{\includegraphics[width=0.4\textwidth]{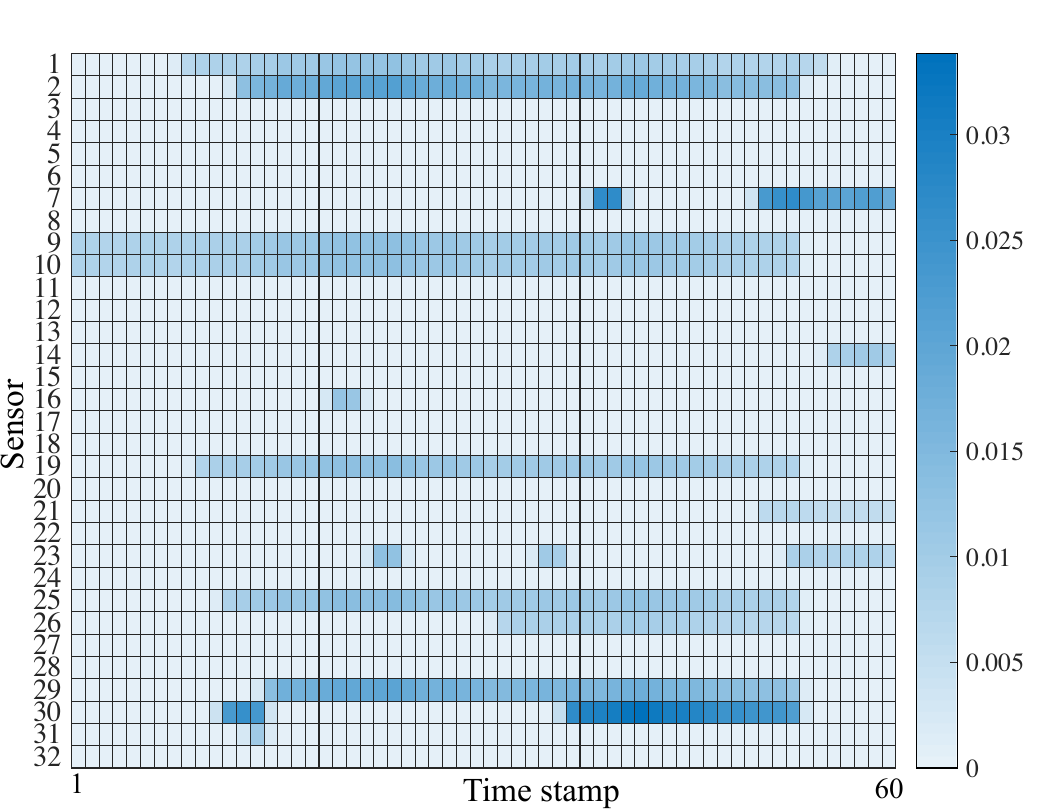}\label{ex1}}
	\subfigure[Results After Thresholding]{\includegraphics[width=0.4\textwidth]{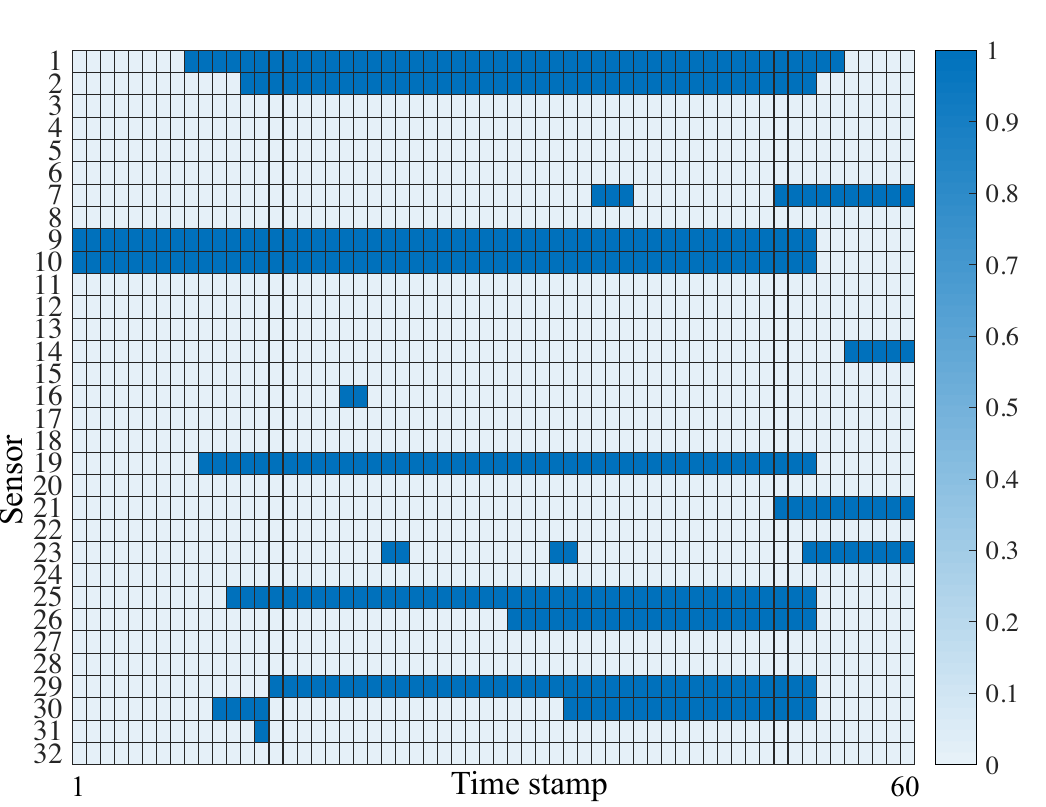}\label{ex2}}
	\subfigure[Ground Truth]{\includegraphics[width=0.4\textwidth]{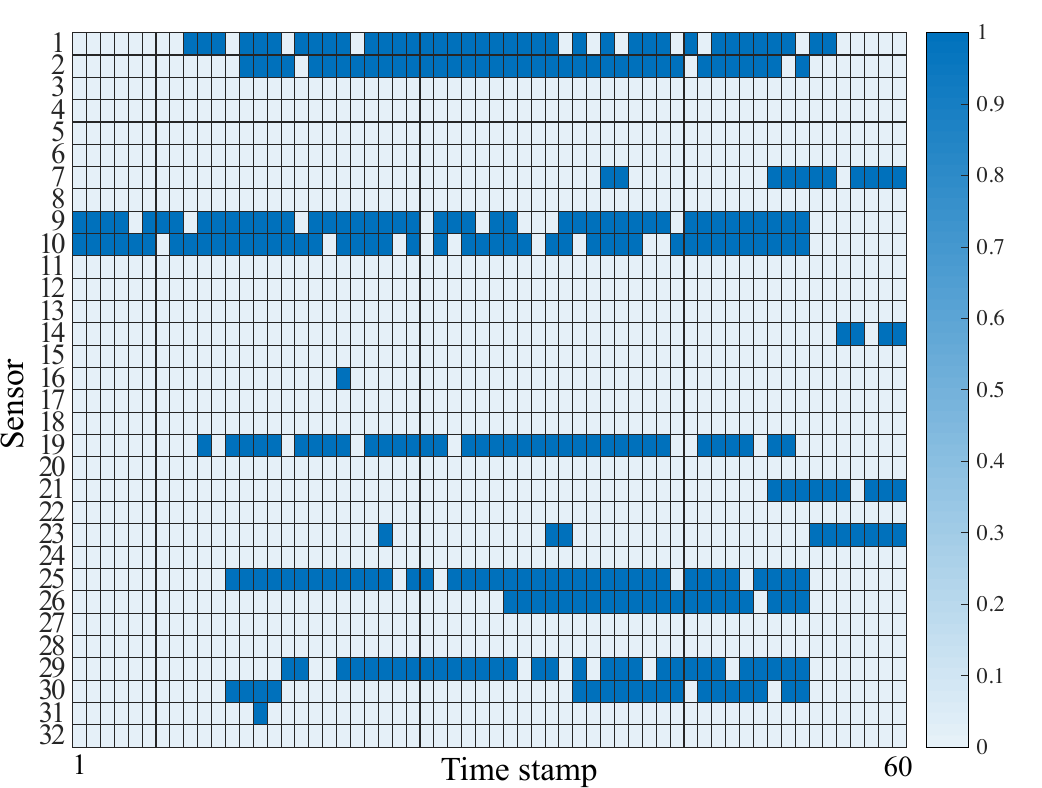}\label{ex3}}
	\caption{Example Comparison Under High Demands}
	\label{ex}		
\end{figure}
\begin{figure}[h!]
	\centering
	\subfigure[Results Before Thresholding]{\includegraphics[width=0.4\textwidth]{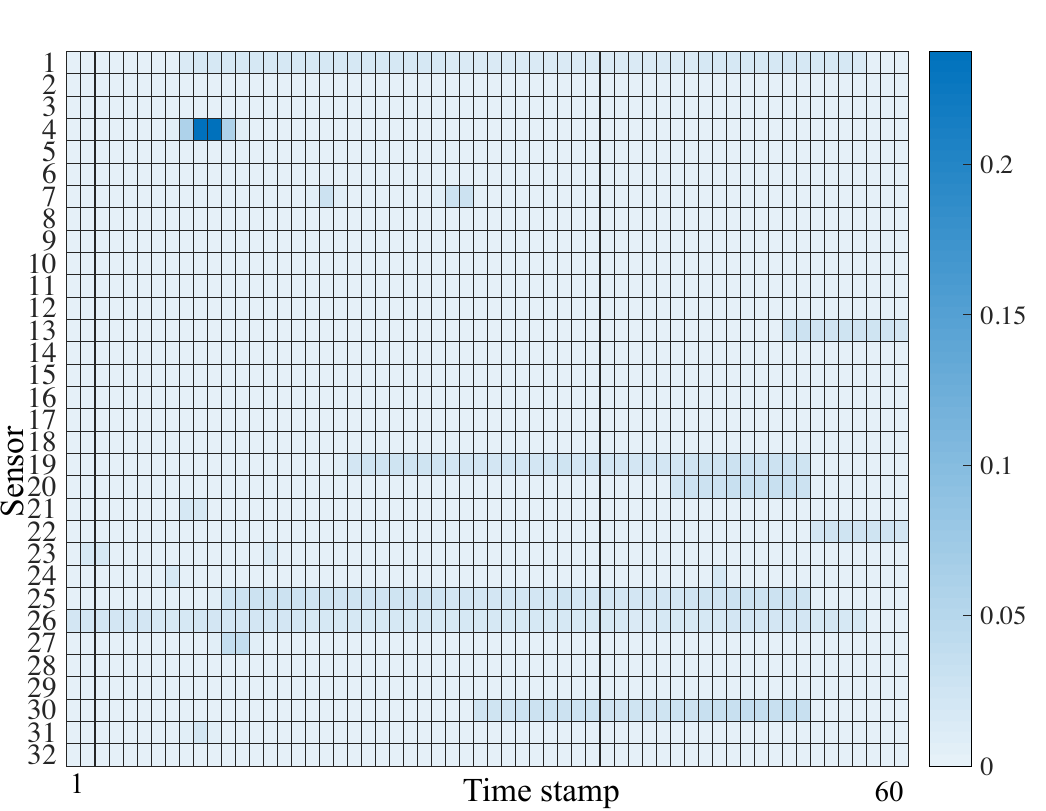}\label{np1}}
	\subfigure[Results After Thresholding]{\includegraphics[width=0.4\textwidth]{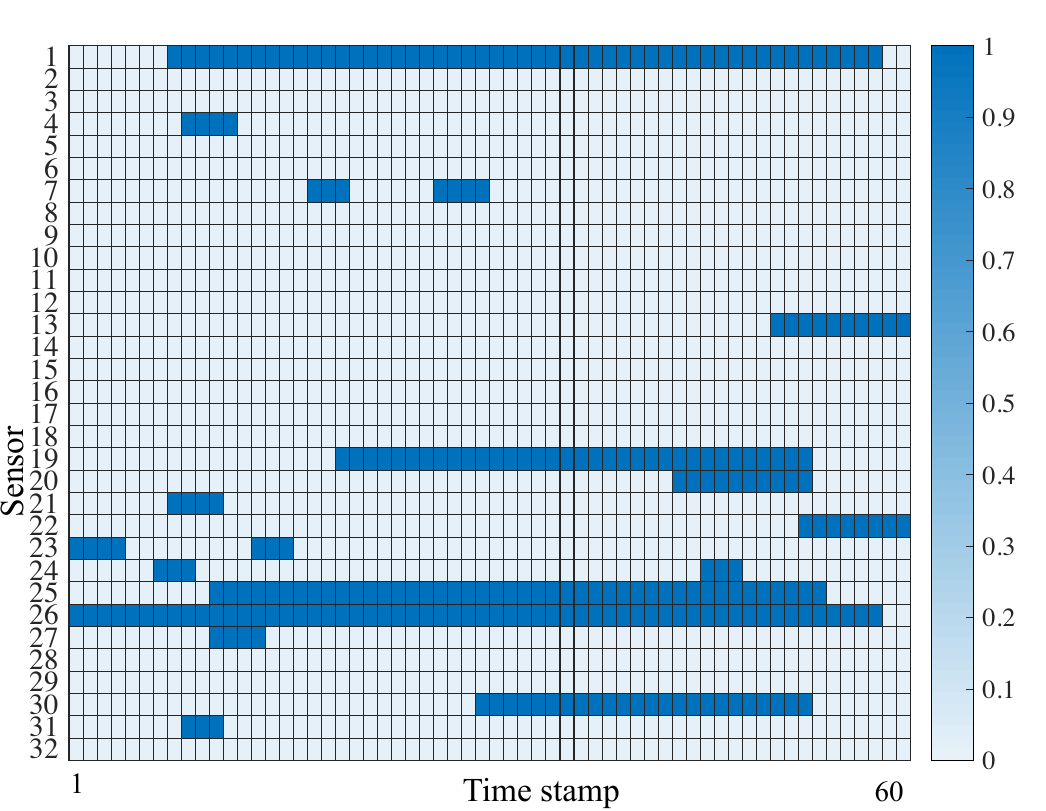}\label{np2}}
	\subfigure[Ground Truth]{\includegraphics[width=0.4\textwidth]{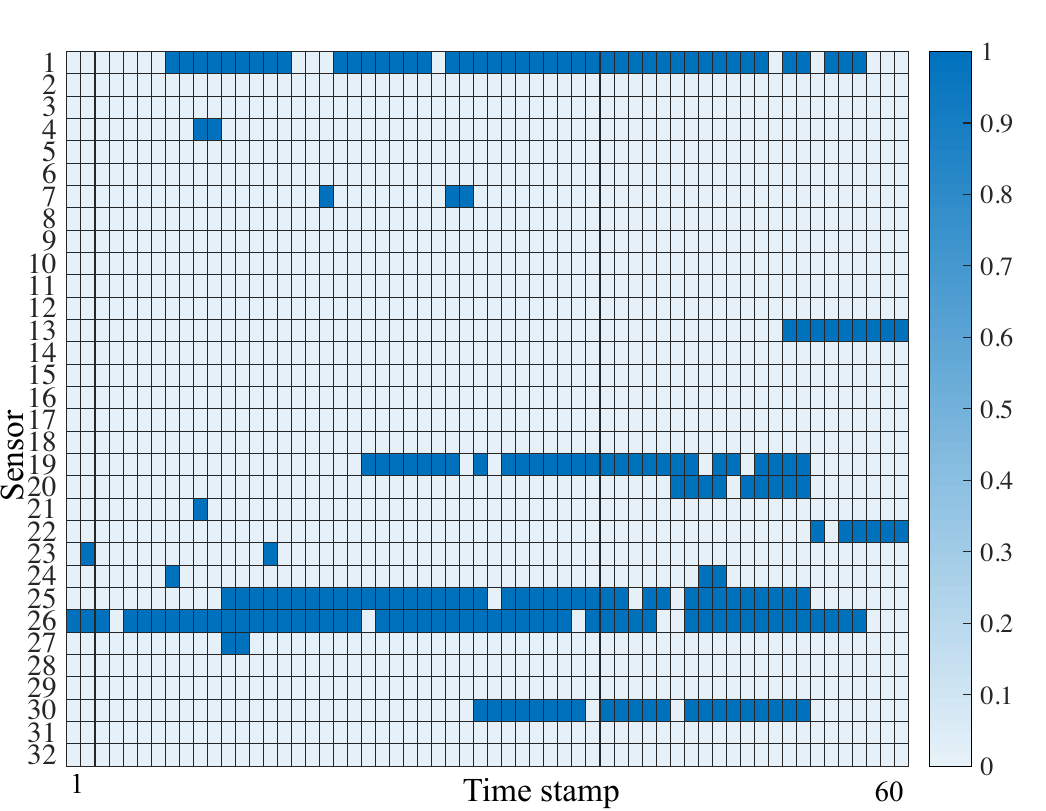}\label{np3}}
	\caption{Example Comparison Under Low Demands}
	\label{npeak}		
\end{figure}
These are the resulting 10 second forecasts produced over a 60-second horizon. The figures illustrate the results before thresholding, after thresholding, and include the ground truth for comparison.  These figures provide an illustration of how our method captures the sensor state dynamics.

\textit{Impact of Rank Parameter ($r$) and Algorithm Convergence}: We next investigate the performance of our approach for different choices of the input $r$ in \eqref{eq_RMCP_SVD_0} and \eqref{eq_KMCP}.  We do so for the case $L = 60$ and $H = 10$ seconds.  The objective values achieved when the algorithm converges for different values of $r$ are depicted in Fig.~\ref{conver}.  We see that the lowest objective value is achieved when $r = 60$ indicating a highly sparse matrix (an order of magnitude smaller than both $nh \sim |\DD|Ln = 3,840$ and $|\DD|T_{\tr} + T_{\te} = 1,140$). Fig.~\ref{conv2} further illustrates the sublinear convergence rate when $r=15$.

\begin{figure}[htbp]
	\centerline{\includegraphics[scale=0.8]{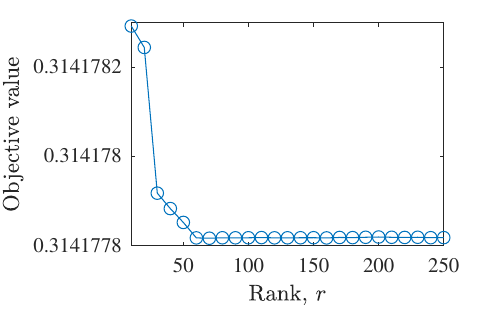}}
	\caption{Relationship Between $r$ and Objective Function Value at Convergence}
	\label{conver}		
\end{figure}

\begin{figure}[htbp]
	\centerline{\includegraphics[scale=0.85]{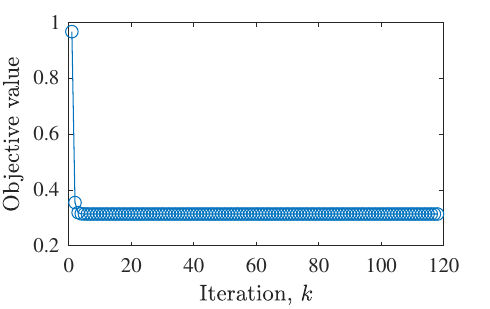}}
	\caption{Sublinear Convergence Rate}
	\label{conv2}		
\end{figure}

\textit{Comparison with Other Techniques}: We benchmark our approach against the following prediction techniques:
\begin{enumerate}
	\item Time series following generalized linear models (TS-GLM): We use the framework described in \citep{fokianos2004partial} to generalize a linear vector auto-regressive time series to binary data. We implement the method using the MATLAB implementation of the \textsf{glmnet} package\footnote{Available at: \url{http://www.stanford.edu/\~hastie/glmnet_matlab/}}\citep{friedman2010regularization}, using the \textsf{binomial} family option, used to perform logistic regression, i.e., employing a logit link function for binary outcomes.  The \textsf{glmnet} package solves a regularized problem with a penalty parameter ($\lambda$).  In our experiments, we fine-tune $\lambda$ by means of a grid search. Finally, we employ thresholding with an 0.5 cutoff to translate the outcomes to 0s and 1s (typical for logistic regression).
	\item Support vector regression (SVR): We choose the SVR model with the standard RBF kernels for comparison. The parameters in SVR, i.e., the error penalty parameter ($C$) and the margin bound parameter ($\epsilon$) are tuned during the training phase by means of a grid search. We also employ a 0.5 cutoff to translate the outputs into binary 0-1 variables.  We utilize the package \textsf{sklearn} in python for implementation.
	
	\item Recurrent neural networks (RNN): RNNs are powerful tools for time series. In our experiments, we employ a long short-term memory (LSTM) architecture \citep{hochreiter1997long}, similar to the neural network architecture used in \citep{mackenzie2018evaluation} to predict traffic flows; we use four layers (input, LSTM layer, fully-connected layer and output). The parameters, i.e., number of hidden units ($N_{\mathrm{H}}$) and learning rate ($\alpha$), are also fine-tuned using a grid search.
\end{enumerate}	

Moreover, for all the benchmark algorithms, the lag and horizon parameters, $L$ and $H$, used in the proposed algorithm are optimized (using a grid search).  The best performance is obtained at $L=60$ and $H=1$. We summarize the selected parameters of these three techniques in Table~\ref{table_p}
\begin{table}[h!]
	\caption{Parameter Settings of Three Benchmark Algorithms.}
	\small
	\centering
	\begin{tabular}{|c|c|}
		\hline
		Method & Parameters
		\\ \hline
		TS-GLM & \textsf{binomial}, \textsf{10-fold} $\lambda$=`$\lambda_{1se}$',  $L=60$, $H=1$  \\
		SVR & $C=0.5$, $\epsilon=0.01$,  $L=60$, $H=1$ \\
		RNN & $N_{\mathrm{H}}=50$, $\alpha=0.1$, $L=60$, $H=1$ \\
		\hline
	\end{tabular}
	\label{table_p}
\end{table}
, and we illustrate the performance of the three benchmark algorithms under varying parameter settings in Fig.~\ref{F:comparison}.
\begin{figure}[h!]
	\centering
	\subfigure[TS-GLM]{\includegraphics[scale=0.4]{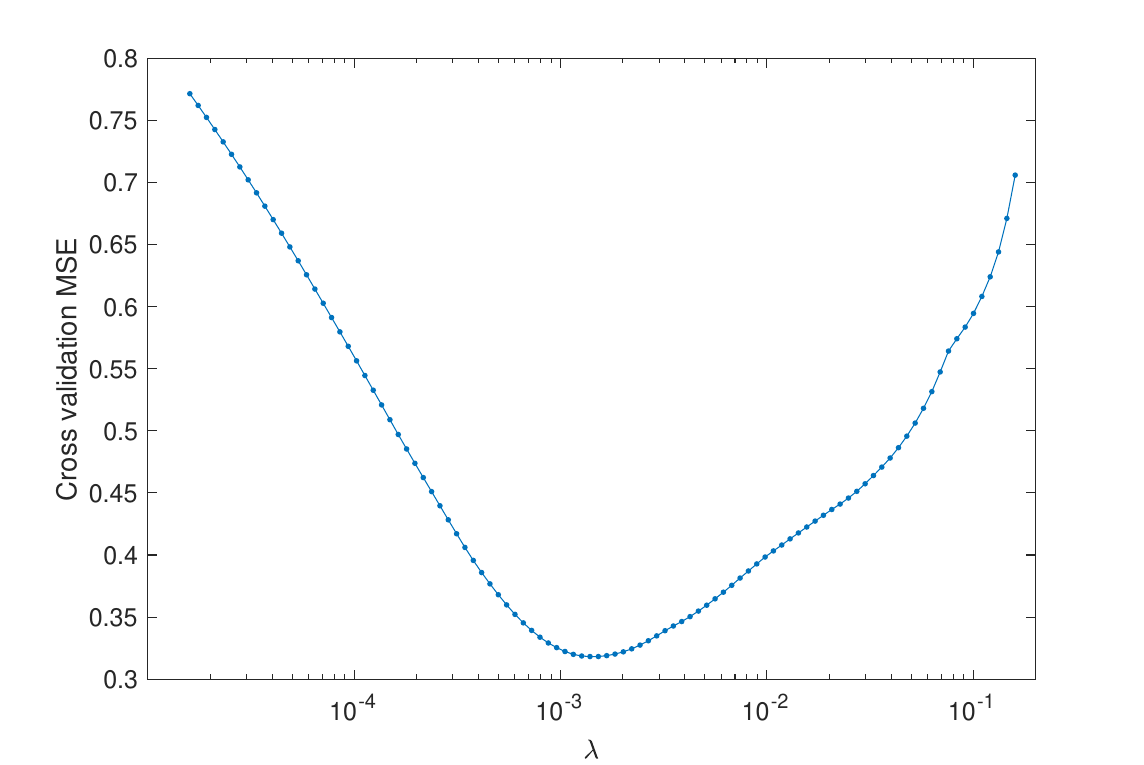}\label{F:comparison-1}} 
	
	\subfigure[RNN]{\includegraphics[scale=0.6]{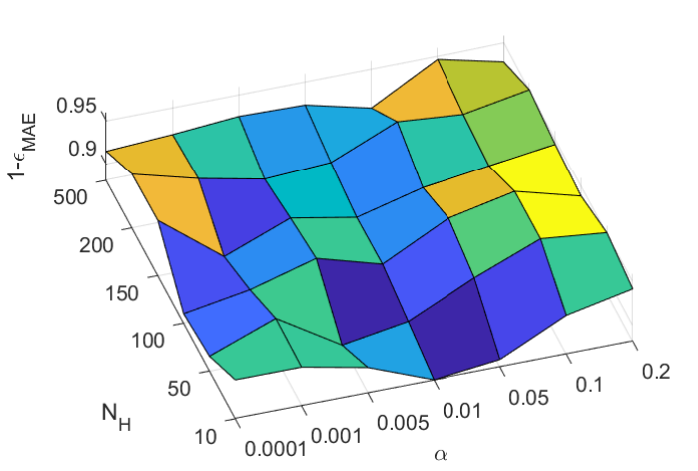}\label{F:comparison-2}}
	
	\subfigure[SVR]{\includegraphics[scale=0.6]{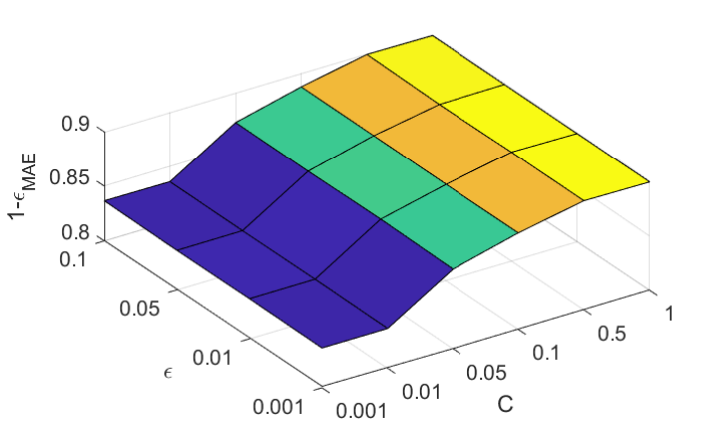}\label{F:comparison-3}}
	\caption{Prediction Performance  with Varying Parameter Settings\label{F:comparison}}		
\end{figure}

We compare these three techniques with two implementations of the proposed approach, the first employs block coordinate decent (BCD) without boosting, Algorithm~\ref{A:2}, and the boosted-BCD (B-BCD) in Algorithm~\ref{A:1}.  The accuracy results are summarized in Table~\ref{table_com}.
\begin{table}[htbp]
	\caption{Accuracy Comparisons}
	\small
	\centering
	\begin{tabular}{|c|cc|}
		\hline
		Method & $1-d_{\mathrm{M}_1}$ & $1 - \epsilon_{\mathrm{MAE}}$
		\\ \hline
		BCD            & 0.9101 $\pm$ 0.0219            &0.9289 $\pm$ 0.0215                                                                     \\
		B-BCD &0.9557 $\pm$ 0.0520&0.9522 $\pm$ 0.0211 \\
		TS-GLM & 0.8519 $\pm$ 0.0266 & 0.9002 $\pm$    0.0373 \\
		SVR & 0.8754 $\pm$ 0.0697& 0.8915 $\pm$ 0.0319\\
		RNN  &0.9473 $\pm$ 0.0315&0.9652 $\pm$ 0.0456\\
		\hline
	\end{tabular}
	\label{table_com}
\end{table}
The results indicate that our B-BCD outperforms the other methods under both metrics, matched only by the RNN.  The differences in accuracy are also significant at the 0.05 level, as per the paired $t$-tests shown in Tables \ref{table_tt1} and \ref{table_tt2}.
\begin{table}[htbp]
	\caption{Paired $t$-test for the Different Methods Under the Accuracy Metric $1 - \epsilon_{\mathrm{MAE}}$ (Significance Level = 0.05)}
	\small
	\centering
	\begin{tabular}{|c|ccccc|}
		\hline
		Method & BCD & B-BCD & TS-GLM &SVR&RNN
		\\ \hline
		BCD          &0  &1&1&1&1                                                                     \\
		B-BCD   & 1 &0&1&1&0  \\
		TS-GLM             &1  &1&0&0&1                                                     \\
		SVR   &  1&1&0&0&1 \\
		RNN    &  1&0&1&1&0 \\
		\hline
	\end{tabular}
	\label{table_tt1}
\end{table}
\begin{table}[htbp]
	\caption{Paired $t$-test for the Different Methods Under the Accuracy Metric $1-d_{\mathrm{M}_1}$ (Significance Level = 0.05)}
	\small
	\centering
	\begin{tabular}{|c|ccccc|}
		\hline
		Method & BCD & B-BCD & TS-GLM &SVR&RNN
		\\ \hline
		BCD          &0  &1&1&1&1                                                                     \\
		B-BCD   & 1 &0&1&1&0  \\
		TS-GLM             &1  &1&0&1&1                                                     \\
		SVR   &  1&1&1&0&1 \\
		RNN    &  1&0&1&1&0 \\
		\hline
	\end{tabular}
	\label{table_tt2}
\end{table}
The two tables show that we can reject the hypothesis that the accuracy achieved by B-BCD or RNN can be achieved by any of the other methods.  It is not surprising that the RNN performs so well as a result of the sophisticated representations afforded by deep neural networks.  It is, however, notable that RNNs are not amenable to real-time implementation due to the heavy computational costs required to train them.  Moreover, interpreting the results is not straightforward with deep neural networks in general. These are the key advantages of our approach: B-BCD can achieve the performance accuracy of a deep neural network, but can be implemented in real-time, it is also easy to extract insights from the results.

\subsection{Real World Example}
\textit{Description of Experiments}: The experiments conducted in this section are meant to test the performance of the proposed approach on the larger real-world dataset described in Sec.~\ref{s_networks}.  We use seven weeks of high-resolution data, 49 days, from the beginning of the first week of December, 2018 to the end of the third week of January, 2019.  We do not distinguish workdays from weekends (as inputs) in our experiments as the proposed method is essentially a dynamic learning approach that is adaptive to time-varying changes and within-week patterns.  The horizons and lags tested in these experiments are $H \in \{1,10,60,120\}$ and $L \in \{10,30,60,120\}$ seconds, and the historical data used for boosting consists of $|\DD| = 20$ consecutive days in each experiment. The number of training and testing time steps are $T_{\tr} = 500$ and $T_{\te} = 100$ seconds, respectively.  We benchmark our approach against the same prediction techniques described in Sec.~\ref{ss_SUMO} at the end of this section, specifically, we compare both BCD and B-BCD against TS-GLM, SVR, and RNN.

\textit{Impact of Choice of Lag ($L$) and Horizon ($H$)}: Again, we measure accuracy using $1- \epsilon_{\mathrm{MAE}}$ and $1 - d_{\mathrm{M}_1}$, and we report mean accuracy and standard deviations (calculated over the 31 sensors) in  Table~\ref{table_4} and Table~\ref{table_5} using different lags and prediction horizons.  The entries in bold are those with the highest accuracy for each prediction horizon.  
\begin{table*}[h!]
	\caption{Testing Accuracy Using $\mathrm{MAE}$ for Real-World Example with 31 sensors ($1 - \epsilon_{\mathrm{MAE}}: \mathrm{Mean}\pm \mathrm{Std}$)}
	\small
	\centering
	\begin{tabular}{|c|cccc|}
		\hline
		$|\DD| \times L$ & $H=1$  & $H=10$ & $H=60$ & $H=120$ 
		\\
		\hline
		$20\times10$            & 0.8758 $\pm$ 0.0205            &  0.8609 $\pm$ 0.0225    &0.8475 $\pm$ 0.0255 & 0.8138 $\pm$ 0.0199 
		\\
		\textbf{$20\times30$}             &  \textbf{0.9091} $\pm$ \textbf{0.0211}        & \textbf{0.9012} $\pm$ \textbf{0.0269}   & \textbf{0.8539} $\pm$ \textbf{0.0215}  &\textbf{0.8432} $\pm$ \textbf{0.0197} 
		\\
		$20\times60$ &0.8764 $\pm$ 0.0252 &0.8631 $\pm$ 0.0215 & 0.8353 $\pm$ 0.0192&0.8155 $\pm$ 0.0197
		\\
		$20\times120$&0.8443 $\pm$ 0.0217&0.8401 $\pm$ 0.0145&0.8178 $\pm$ 0.0278&0.7901 $\pm$ 0.0201
		\\
		\hline
	\end{tabular}
	\label{table_4}
\end{table*}
\begin{table*}[h!]
	\caption{Testing Accuracy Using the Skorokhod $\mathrm{M}_1$ Metric for Real-World Example with 31 sensors ($1 - d_{\mathrm{M}_1}: \mathrm{Mean} \pm \mathrm{Std}$)}
	\small
	\centering
	\begin{tabular}{|c|cccc|}
		\hline
		$|\DD| \times L$ & $H=1$  & $H=10$ & $H=60$ & $H=120$ 
		\\
		\hline
		$20\times10$            & 0.8618 $\pm$ 0.0197               &  0.8509 $\pm$ 0.0215    &0.8453 $\pm$ 0.0221 & 0.8124 $\pm$ 0.0203      
		\\
		\textbf{$20\times30$}             &\textbf{0.8858} $\pm$ \textbf{0.0213}        & \textbf{0.8719} $\pm$ \textbf{0.0235}   & \textbf{0.8419} $\pm$ \textbf{0.0225}  &\textbf{0.8332} $\pm$ \textbf{0.0257}  
		\\
		$20\times60$ &0.8517 $\pm$ 0.0201 &0.8438 $\pm$ 0.0211 & 0.8359 $\pm$ 0.0197&0.8158 $\pm$ 0.0211
		\\
		$20\times120$&0.8313 $\pm$ 0.0227&0.8398 $\pm$ 0.0195&0.8268 $\pm$ 0.0218&0.8108 $\pm$ 0.0211
		\\
		\hline
	\end{tabular}
	\label{table_5}
\end{table*}
We see roughly the same pattern in both tables and similar results to those we observed in Sec.~\ref{ss_SUMO}, namely, that the accuracy decreases as $H$ increases and that longer lags ($L$) do not necessarily mean improved accuracy.  The mean accuracies (taken over all 31 sensors) is no lower than 79.01\% and reaches 90.91\%, while the corresponding standard deviations do not exceed 2.52\%.  The lowest accuracies under $1 - \epsilon_{\mathrm{MAE}}$ and $1 - d_{\mathrm{M}_1}$, respectively, exceed 75.07\% and 76.94\% in 97.5\% of the cases.

We also conduct pair $t$-tests to investigate the impact of the variables $H$ and $L$ and present the results in Table~\ref{table_ttest1} below.  The top part of the table summarizes the results of changing $L$, the bottom part summarizes the effect of changing $H$.  All results suggest that the differences between the different models are significant at the 0.05 level.
\begin{table*}[htbp!]
	\caption{Paired $t$-test for Different Lags and Horizons Using Real-World Data (Significance Level = 0.05)}
	\small
	\centering
	\begin{tabular}{|c|cccc|}
		\hline
		$(H,|\DD| \times L)$ & $(1,20\times10)$    & $(1,20\times30)$   & $(1,20\times60)$ & $(1,20\times120)$ 
		\\ 
		\hline
		$(1,20\times10)$               &0  &1 & 1&1
		\\
		$(1,20\times30)$               &1  &0 & 1 &1
		\\
		$(1,20\times60)$      &1& 1 &0 & 1
		\\
		$(1,20\times120)$    & 1&1  &1 & 0
		\\
		\hline
		$(H,|\DD| \times L)$  & $(1,20\times10)$ & $(10,20\times30)$  & $(60,20\times30)$   & $(120,20\times30)$
		\\ 
		\hline
		
		$(1,20\times30)$            &    0     & 1     &1 &1                                                
		\\
		$(10,20\times30)$            & 1       & 0 & 1&1
		\\
		$(60,20\times30)$  &     1   & 1 &0 &1
		\\
		$(120,20\times30)$               &1  &1 & 1&0
		\\ \hline
	\end{tabular}
	\label{table_ttest1}
\end{table*}

\textit{Choice of Rank Parameter ($r$) and Algorithm Convergence}: We select the case $L = 30$ and $H = 10$ seconds to test the impact of the rank parameter ($r$).  The objective values achieved when the algorithm converges under different values of $r$ are depicted in Fig.~\ref{F:objr}.  We see that the lowest objective value is achieved when $r = 95$ indicating a highly sparse matrix (two orders of magnitude smaller than both $nh \sim |\DD|Ln = 18,600$ and $|\DD|T_{\tr} + T_{\te}=10,100$).
\begin{figure}[h!]
	\centerline{\includegraphics[scale=0.5]{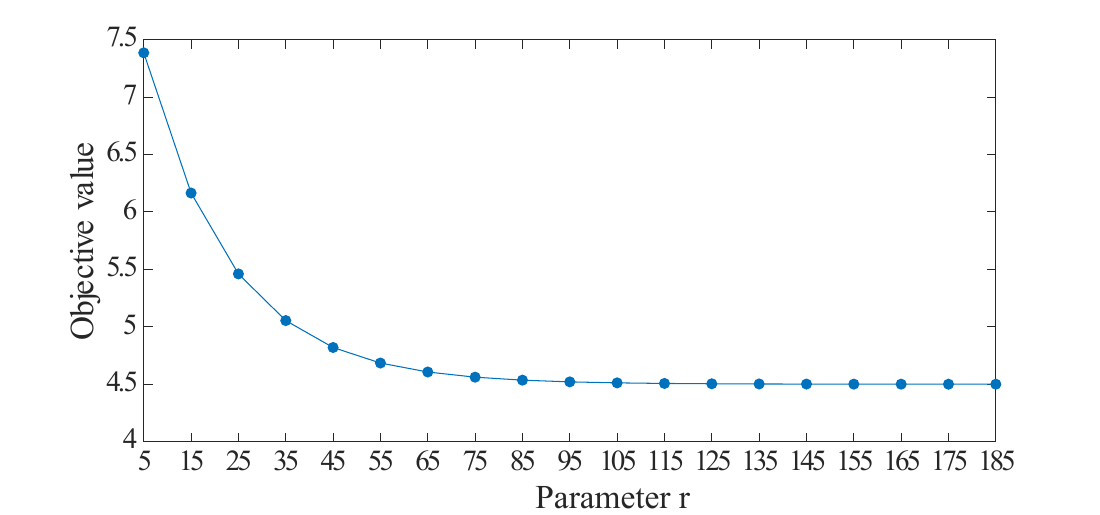}}
	\caption{Relationship Between $r \ge \rank(\ZZ)$ and Objective Function Value at Convergence}
	\label{F:objr}		
\end{figure}
Figures \ref{F:obj} and \ref{F:boost} illustrate the sub-linear convergence rate of the algorithm for the same inputs ($L = 30$ seconds, $H = 10$ seconds, and $r=95$). Fig.~\ref{F:boost} also includes the testing accuracy, which also converges fast.  Note the small gap between the two curves in Fig.~\ref{F:boost}, which suggests that the model generalizes well.
\begin{figure}[h!]
	\centerline{\includegraphics[scale=0.5]{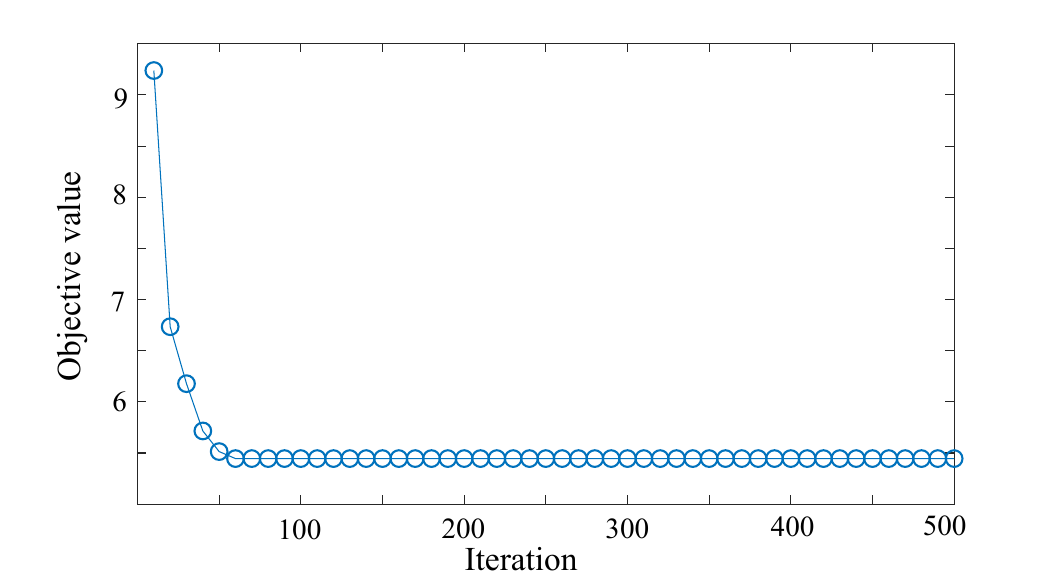}}
	\caption{Convergence Rate}
	\label{F:obj}		
\end{figure}
\begin{figure}[h!]
	\centerline{\includegraphics[scale=0.5]{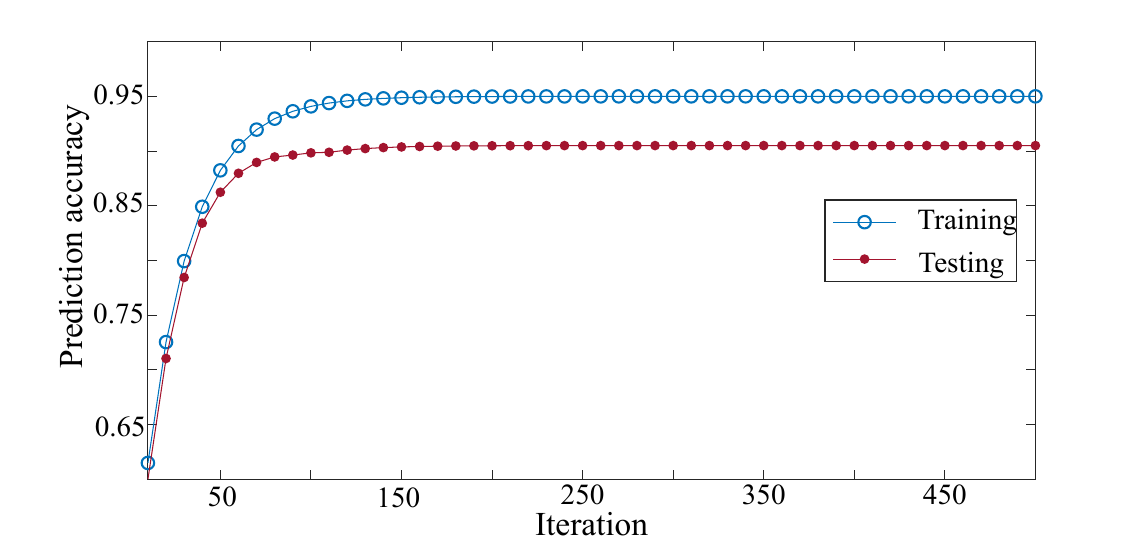}}
	\caption{Sub-Linear Convergence of both Training and Testing Accuracies}
	\label{F:boost}		
\end{figure}

\textit{Comparisons}: We now show the results of our comparisons.  We first provide the accuracy results for all of the methods that we apply, averaged over all 31 sensors along with standard deviations in Table~\ref{table_6}.  The results suggest that our B-BCD outperforms all other methods except for the RNN. Paired $t$-tests were carried out under both accuracy metrics, and the results are listed in Tables \ref{table_tt3} and \ref{table_tt4}.  Under both metrics, we see that the differences between B-BCD and all other approaches (excluding RNN) are significant.  Again, RNN is a good benchmark for comparison, but not implementation due to the heavy computational costs that come with fitting them.  Our B-BCD approach can achieve accuracy that is comparable to a well-trained RNN, but offers computational advantages.
\begin{table}[h!]
	\caption{Accuracy Comparisons for Real-World Problem}
	\small
	\centering
	\begin{tabular}{|c|cc|}
		\hline
		Method & $1-d_{\mathrm{M}_1}$ & $1 - \epsilon_{\mathrm{MAE}}$
		\\ \hline
		BCD            & 0.8101 $\pm$ 0.0219            &0.8589 $\pm$ 0.0215 \\
		B-BCD &0.8858 $\pm$ 0.0213&0.9091 $\pm$ 0.0211 \\
		TS-GLM & 0.8058 $\pm$ 0.0328 & 0.8215 $\pm$    0.0329 \\
		SVR & 0.7867 $\pm$ 0.0161& 0.8015 $\pm$ 0.0135\\
		RNN  &0.8703 $\pm$ 0.0305&0.9127 $\pm$ 0.0192\\
		\hline
	\end{tabular}
	\label{table_6}
\end{table}
\begin{table}[h!]
	\caption{Paired $t$-test for the Different Methods Under the Accuracy Metric $1 - \epsilon_{\mathrm{MAE}}$ (Significance Level = 0.05)}
	\small
	\centering
	\begin{tabular}{|c|ccccc|}
		\hline
		Method & BCD & B-BCD & TS-GLM &SVR&RNN
		\\  \hline
		BCD          &0  &1&1&1&1   \\
		B-BCD   & 1 &0&1&1&0  \\
		TS-GLM &1  &1&0&0&1  \\
		SVR   &  1&1&0&0&1 \\
		RNN    &  1&0&1&1&0 \\
		\hline
	\end{tabular}
	\label{table_tt3}
\end{table}
\begin{table}[h!]
	\caption{Paired $t$-test for the Different Methods Under the Accuracy Metric $1-d_{\mathrm{M}_1}$ (Significance Level = 0.05)}
	\small
	\centering
	\begin{tabular}{|c|ccccc|}
		\hline
		Method & BCD & B-BCD & TS-GLM &SVR&RNN
		\\ \hline
		BCD          &0  &1&1&1&1 \\
		B-BCD   & 1 &0&1&1&0  \\
		TS-GLM  &1  &1&0&1&1 \\
		SVR   &  1&1&1&0&1 \\
		RNN    &  1&0&1&1&0 \\
		\hline
	\end{tabular}
	\label{table_tt4}
\end{table}

\section{Conclusions}
\label{S:conc}
Our contribution can broadly be described as specializing contemporary convex optimization techniques to traffic prediction. These techniques have revolutionized a variety of applications that involve large volumes of data, from image processing to online recommender systems.  However, they seem to have found little or no application in transportation science and traffic management.  Our analysis demonstrates that one can solve large prediction problems in real-time (sub-linear convergence rate) and that the solutions obtained are block coordinate-wise minimizers.  We also demonstrated that training error can be made arbitrarily small with ensemble learning.  The latter are typically used to amalgamate results from heterogeneous techniques.  Our implementation uses historical data as predictors.  We elected to perform ensemble learning in this way for the sake of interpretability of our results.

The analysis in this paper culminates in a bound on the training error and a brief discussion on how this generalizes.  The training error bound is a specialization of a well known bound that comes with the AdaBoost algorithm.  Our bound allows for any type of thresholding. The generalization errors analysis is not complete without an analysis of the VC-dimensions of our models, although our empirical results suggest that the out-of-sample errors are similar to the training error, with differences in accuracy that do not exceed 10\%.  Future work can be conducted along two separate lines:  The first is an in-depth analysis of how these methods generalize, analyzing VC-dimension and other approaches. The VC-dimension gives worst-case bounds, which may or may not be useful in practice.  It would also be useful to consider kernels inspired by traffic physics and how this affects sample complexity. The second involves investigations of what should be considered an acceptable level of error.  This will depend on the application; for example, for signal timing optimization a lag of three seconds in a sensor actuation prediction can trigger a signal status decision that results in unwanted congestion.  We leave these questions to future research.

\section*{Acknowledgment}
This work was supported by the NYUAD Center for Interacting Urban Networks (CITIES), funded by Tamkeen under the NYUAD Research Institute Award CG001 and by the Swiss Re Institute under the Quantum Cities\textsuperscript{TM} initiative.  The authors would also like to acknowledge in-kind support received from the Abu Dhabi Department of Transportation, in the form of the high-resolution traffic data that were used in our experiments.  The opinions expressed in this article are those of the authors alone do not represent the opinions of CITIES or the Abu Dhabi Department of Transportation.

\section*{Supplemental Material}
	Supplemental material including source code and data can be found at \url{https://github.com/lwqangle123/EKMC}.

\appendix
\gdef\thesection{Appendix \Alph{section}}

	
	
{ \small
\bibliographystyle{plainnat}
\bibliography{refs}

\begin{thebibliography}{90}
\providecommand{\natexlab}[1]{#1}
\providecommand{\url}[1]{\texttt{#1}}
\expandafter\ifx\csname urlstyle\endcsname\relax
  \providecommand{\doi}[1]{doi: #1}\else
  \providecommand{\doi}{doi: \begingroup \urlstyle{rm}\Url}\fi

\bibitem[Balke et~al.(2005)Balke, Charara, and Parker]{balke2005development}
Kevin Balke, Hassan Charara, and Ricky Parker.
\newblock \emph{Development of a traffic signal performance measurement system
  ({TSPMS})}.
\newblock Texas Department of Transportation, Report No. FHWA/TX-05/0-4422-2,
  2005.

\bibitem[Beck and Tetruashvili(2013)]{beck2013convergence}
Amir Beck and Luba Tetruashvili.
\newblock On the convergence of block coordinate descent type methods.
\newblock \emph{SIAM journal on Optimization}, 23\penalty0 (4):\penalty0
  2037--2060, 2013.

\bibitem[Benjamin et~al.(2003)Benjamin, Rigby, and
  Stasinopoulos]{benjamin2003generalized}
Michael~A Benjamin, Robert~A Rigby, and D~Mikis Stasinopoulos.
\newblock Generalized autoregressive moving average models.
\newblock \emph{Journal of the American Statistical Association}, 98\penalty0
  (461):\penalty0 214--223, 2003.

\bibitem[Benkraouda et~al.(2020)Benkraouda, Thodi, Yeo, Menendez, and
  Jabari]{benkraouda2020traffic}
Ouafa Benkraouda, Bilal~Thonnam Thodi, Hwasoo Yeo, Monica Menendez, and
  Saif~Eddin Jabari.
\newblock Traffic data imputation using deep convolutional neural networks.
\newblock \emph{IEEE Access}, 8:\penalty0 104740--104752, 2020.

\bibitem[Broersen(1997)]{broersen1997abc}
Piet Broersen.
\newblock The {ABC} of autoregressive order selection criteria.
\newblock \emph{IFAC Proceedings Volumes}, 30\penalty0 (11):\penalty0 245--250,
  1997.

\bibitem[Brun et~al.(2014)Brun, Aleksandrova, and Boyer]{brun2014can}
Armelle Brun, Marharyta Aleksandrova, and Anne Boyer.
\newblock Can latent features be interpreted as users in matrix
  factorization-based recommender systems?
\newblock In \emph{2014 IEEE/WIC/ACM International Joint Conferences on Web
  Intelligence (WI) and Intelligent Agent Technologies (IAT)}, volume~2, pages
  226--233, 2014.

\bibitem[Cand{\`e}s and Recht(2009)]{candes2009exact}
Emmanuel~J Cand{\`e}s and Benjamin Recht.
\newblock Exact matrix completion via convex optimization.
\newblock \emph{Foundations of Computational Mathematics}, 9\penalty0
  (6):\penalty0 717--772, 2009.

\bibitem[Cand{\`e}s and Tao(2009)]{candes2009power}
Emmanuel~J. Cand{\`e}s and Terence Tao.
\newblock The power of convex relaxation: {N}ear-optimal matrix completion.
\newblock \emph{IEEE Transactions on Information Theory}, 56\penalty0
  (5):\penalty0 2053--2080, 2009.

\bibitem[Chen et~al.(2011)Chen, Hu, Meng, and Zhang]{chen2011short}
Chenyi Chen, Jianming Hu, Qiang Meng, and Yi~Zhang.
\newblock Short-time traffic flow prediction with {ARIMA}-{GARCH} model.
\newblock In \emph{2011 IEEE Intelligent Vehicles Symposium (IV)}, pages
  607--612, 2011.

\bibitem[Chen et~al.(2017)Chen, Yu, Wu, Ren, and Li]{chen2017estimation}
Peng Chen, Guizhen Yu, Xinkai Wu, Yilong Ren, and Yueguang Li.
\newblock Estimation of red-light running frequency using high-resolution
  traffic and signal data.
\newblock \emph{Accident Analysis \& Prevention}, 102:\penalty0 235--247, 2017.

\bibitem[Coifman and Cassidy(2002)]{coifman2002vehicle}
Benjamin Coifman and Michael Cassidy.
\newblock Vehicle reidentification and travel time measurement on congested
  freeways.
\newblock \emph{Transportation Research Part A: Policy and Practice},
  36\penalty0 (10):\penalty0 899--917, 2002.

\bibitem[Curtis(2017)]{asct}
Eddie Curtis.
\newblock \emph{Adaptive Signal Control Technology, September 8, 2017, Retrived
  October 30, 2020}.
\newblock Federal Highway Administration, 2017.
\newblock URL
  \url{https://www.fhwa.dot.gov/innovation/everydaycounts/edc-1/asct.cfm}.

\bibitem[Dilip et~al.(2017)Dilip, Freris, and Jabari]{dilip2017sparse}
Deepthi Dilip, Nikolaos Freris, and Saif~Eddin Jabari.
\newblock Sparse estimation of travel time distributions using {G}amma kernels.
\newblock In \emph{Proceedings of the 96th Annual Meeting of the Transportation
  Research Board, Washington D.C.}, pages No. 17--02971, 2017.

\bibitem[Fang et~al.(1994)Fang, Loparo, and Feng]{fang1994inequalities}
Yuguang Fang, Kenneth~A Loparo, and Xiangbo Feng.
\newblock Inequalities for the trace of matrix product.
\newblock \emph{IEEE Transactions on Automatic Control}, 39\penalty0
  (12):\penalty0 2489--2490, 1994.

\bibitem[Fokianos and Kedem(2004)]{fokianos2004partial}
Konstantinos Fokianos and Benjamin Kedem.
\newblock Partial likelihood inference for time series following generalized
  linear models.
\newblock \emph{Journal of Time Series Analysis}, 25\penalty0 (2):\penalty0
  173--197, 2004.

\bibitem[Freund and Schapire(1997)]{freund1997decision}
Yoav Freund and Robert~E Schapire.
\newblock A decision-theoretic generalization of on-line learning and an
  application to boosting.
\newblock \emph{Journal of Computer and System Sciences}, 55\penalty0
  (1):\penalty0 119--139, 1997.

\bibitem[Friedman et~al.(2010)Friedman, Hastie, and
  Tibshirani]{friedman2010regularization}
Jerome Friedman, Trevor Hastie, and Rob Tibshirani.
\newblock Regularization paths for generalized linear models via coordinate
  descent.
\newblock \emph{Journal of Statistical Software}, 33\penalty0 (1):\penalty0 1,
  2010.

\bibitem[Ghosh et~al.(2009)Ghosh, Basu, and O'Mahony]{ghosh2009multivariate}
Bidisha Ghosh, Biswajit Basu, and Margaret O'Mahony.
\newblock Multivariate short-term traffic flow forecasting using time-series
  analysis.
\newblock \emph{IEEE Transactions on Intelligent Transportation Systems},
  10\penalty0 (2):\penalty0 246--254, 2009.

\bibitem[Goldberg et~al.(2010)Goldberg, Recht, Xu, Nowak, and
  Zhu]{goldberg2010transduction}
Andrew Goldberg, Ben Recht, Junming Xu, Robert Nowak, and Jerry Zhu.
\newblock Transduction with matrix completion: {T}hree birds with one stone.
\newblock In \emph{Advances in neural information processing systems}, pages
  757--765, 2010.

\bibitem[Guo et~al.(2007)Guo, Williams, and Smith]{guo2007data}
Jianhua Guo, Billy~M Williams, and Brian~L Smith.
\newblock Data collection time intervals for stochastic short-term traffic flow
  forecasting.
\newblock \emph{Transportation Research Record}, 2024\penalty0 (1):\penalty0
  18--26, 2007.

\bibitem[Guo et~al.(2019)Guo, Li, and Ban]{guo2019urban}
Qiangqiang Guo, Li~Li, and Xuegang~Jeff Ban.
\newblock Urban traffic signal control with connected and automated vehicles:
  {A} survey.
\newblock \emph{Transportation Research Part C: Emerging Technologies},
  101:\penalty0 313--334, 2019.

\bibitem[Harrell~Jr(2015)]{harrell2015regression}
Frank~E Harrell~Jr.
\newblock \emph{Regression modeling strategies: with applications to linear
  models, logistic and ordinal regression, and survival analysis}.
\newblock Springer International Publishing, Cham, Switzerland, 2015.

\bibitem[Hochreiter and Schmidhuber(1997)]{hochreiter1997long}
Sepp Hochreiter and J{\"u}rgen Schmidhuber.
\newblock Long short-term memory.
\newblock \emph{Neural Computation}, 9\penalty0 (8):\penalty0 1735--1780, 1997.

\bibitem[Hu and Liu(2013)]{hu2013arterial}
Heng Hu and Henry~X Liu.
\newblock Arterial offset optimization using archived high-resolution traffic
  signal data.
\newblock \emph{Transportation Research Part C: Emerging Technologies},
  37:\penalty0 131--144, 2013.

\bibitem[Hu et~al.(2013)Hu, Wu, and Liu]{hu2013managing}
Heng Hu, Xinkai Wu, and Henry~X Liu.
\newblock Managing oversaturated signalized arterials: a maximum flow based
  approach.
\newblock \emph{Transportation Research Part C: Emerging Technologies},
  36:\penalty0 196--211, 2013.

\bibitem[Hu et~al.(2014)Hu, Gao, Yao, and Xie]{hu2014traffic}
Jianming Hu, Pan Gao, Yunfei Yao, and Xudong Xie.
\newblock Traffic flow forecasting with particle swarm optimization and support
  vector regression.
\newblock In \emph{17th International IEEE Conference on Intelligent
  Transportation Systems (ITSC)}, pages 2267--2268, 2014.

\bibitem[Ishak and Al-Deek(2002)]{ishak2002performance}
Sherif Ishak and Haitham Al-Deek.
\newblock Performance evaluation of short-term time-series traffic prediction
  model.
\newblock \emph{Journal of Transportation Engineering}, 128\penalty0
  (6):\penalty0 490--498, 2002.

\bibitem[Jabari(2012)]{jabari2012}
Saif~Eddin Jabari.
\newblock \emph{A stochastic model of macroscopic traffic flow: {T}heoretical
  foundations}.
\newblock Ph.D. Dissertation. The University of Minnesota Twin Cities,
  Minneapolis, MN, 2012.

\bibitem[Jabari and Liu(2012)]{jabari2012stochastic}
Saif~Eddin Jabari and Henry Liu.
\newblock A stochastic model of traffic flow: {T}heoretical foundations.
\newblock \emph{Transportation Research Part B: Methodological}, 46\penalty0
  (1):\penalty0 156--174, 2012.

\bibitem[Jabari and Liu(2013)]{jabari2013stochastic}
Saif~Eddin Jabari and Henry Liu.
\newblock A stochastic model of traffic flow: {G}aussian approximation and
  estimation.
\newblock \emph{Transportation Research Part B: Methodological}, 47:\penalty0
  15--41, 2013.

\bibitem[Jabari et~al.(2019)Jabari, Dilip, Lin, and
  Thonnam~Thodi]{jabari2019learning}
Saif~Eddin Jabari, Deepthi Dilip, DianChao Lin, and Bilal Thonnam~Thodi.
\newblock Learning traffic flow dynamics using random fields.
\newblock \emph{IEEE Access}, 7:\penalty0 130566--130577, 2019.

\bibitem[Jabari et~al.(2020)Jabari, Freris, and Dilip]{jabari2020sparse}
Saif~Eddin Jabari, Nikolaos Freris, and Deepthi Dilip.
\newblock Sparse travel time estimation from streaming data.
\newblock \emph{Transportation Science}, 54\penalty0 (1):\penalty0 1--20, 2020.

\bibitem[Jain et~al.(2013)Jain, Netrapalli, and Sanghavi]{jain2013low}
Prateek Jain, Praneeth Netrapalli, and Sujay Sanghavi.
\newblock Low-rank matrix completion using alternating minimization.
\newblock In \emph{Proceedings of the 45th Annual ACM Symposium on Theory of
  Computing}, pages 665--674, 2013.

\bibitem[Jeong et~al.(2013)Jeong, Byon, Castro-Neto, and
  Easa]{jeong2013supervised}
Young-Seon Jeong, Young-Ji Byon, Manoel~Mendonca Castro-Neto, and Said~M Easa.
\newblock Supervised weighting-online learning algorithm for short-term traffic
  flow prediction.
\newblock \emph{IEEE Transactions on Intelligent Transportation Systems},
  14\penalty0 (4):\penalty0 1700--1707, 2013.

\bibitem[Kamarianakis and Prastacos(2003)]{kamarianakis2003forecasting}
Yiannis Kamarianakis and Poulicos Prastacos.
\newblock Forecasting traffic flow conditions in an urban network: {C}omparison
  of multivariate and univariate approaches.
\newblock \emph{Transportation Research Record}, 1857\penalty0 (1):\penalty0
  74--84, 2003.

\bibitem[Kamarianakis and Prastacos(2005)]{kamarianakis2005space}
Yiannis Kamarianakis and Poulicos Prastacos.
\newblock Space--time modeling of traffic flow.
\newblock \emph{Computers \& Geosciences}, 31\penalty0 (2):\penalty0 119--133,
  2005.

\bibitem[Kamarianakis et~al.(2012)Kamarianakis, Shen, and
  Wynter]{kamarianakis2012real}
Yiannis Kamarianakis, Wei Shen, and Laura Wynter.
\newblock Real-time road traffic forecasting using regime-switching space-time
  models and adaptive {LASSO}.
\newblock \emph{Applied Stochastic Models in Business and Industry},
  28\penalty0 (4):\penalty0 297--315, 2012.

\bibitem[Kang et~al.(2017)Kang, Lv, and Chen]{kang2017short}
Danqing Kang, Yisheng Lv, and Yuan-yuan Chen.
\newblock Short-term traffic flow prediction with {LSTM} recurrent neural
  network.
\newblock In \emph{Proceedings of the 20th IEEE International Conference on
  Intelligent Transportation Systems (ITSC)}, pages 1--6, 2017.

\bibitem[Kim and Niemeier(2001)]{kim2001weighted}
Sung-Eun Kim and Debbie Niemeier.
\newblock A weighted autoregressive model to improve mobile emissions estimates
  for locations with spatial dependence.
\newblock \emph{Transportation Science}, 35\penalty0 (4):\penalty0 413--424,
  2001.

\bibitem[Krajzewicz et~al.(2012)Krajzewicz, Erdmann, Behrisch, and
  Bieker]{krajzewicz2012recent}
Daniel Krajzewicz, Jakob Erdmann, Michael Behrisch, and Laura Bieker.
\newblock Recent development and applications of {SUMO}-{S}imulation of {U}rban
  {MO}bility.
\newblock \emph{International Journal on Advances in Systems and Measurements},
  5\penalty0 (3\&4), 2012.

\bibitem[Kumar and Vanajakshi(2015)]{kumar2015short}
S~Vasantha Kumar and Lelitha Vanajakshi.
\newblock Short-term traffic flow prediction using seasonal {ARIMA} model with
  limited input data.
\newblock \emph{European Transport Research Review}, 7\penalty0 (3):\penalty0
  21, 2015.

\bibitem[Kwak and Geroliminis(2019)]{kwak2019traffic}
Semin Kwak and Nikolas Geroliminis.
\newblock Traffic forecasting for freeway networks by a localized linear
  regression time series model with a graph data dimensional reduction method.
\newblock In \emph{Proceedings of the 19th Swiss Transport Research
  Conference}, pages 1--6, 2019.

\bibitem[Li and Jabari(2019)]{li2019position}
Li~Li and Saif~Eddin Jabari.
\newblock Position weighted backpressure intersection control for urban
  networks.
\newblock \emph{Transportation Research Part B: Methodological}, 128:\penalty0
  435--461, 2019.

\bibitem[Li and Jabari(2020)]{li2020decentralized}
Li~Li and Saif~Eddin Jabari.
\newblock A decentralized network control approach based on continuum traffic
  flow modeling.
\newblock In \emph{99th TRB Annual Meeting. Washington, D.C. 2020}, 2020.

\bibitem[Li et~al.(2021)Li, Okoth, and Jabari]{li2020backpressure}
Li~Li, Victor Okoth, and Saif~Eddin Jabari.
\newblock Backpressure control with estimated queue lengths for urban network
  traffic.
\newblock \emph{IET Intelligent Transport Systems}, 15\penalty0 (2):\penalty0
  320--330, 2021.

\bibitem[Li et~al.(2017)Li, Zhao, and Gao]{li2017linearity}
Wenqing Li, Chunhui Zhao, and Furong Gao.
\newblock Linearity evaluation and variable subset partition based hierarchical
  process modeling and monitoring.
\newblock \emph{IEEE Transactions on Industrial Electronics}, 65\penalty0
  (3):\penalty0 2683--2692, 2017.

\bibitem[Lin and Jabari(2021)]{lin2021pay}
DianChao Lin and Saif~Eddin Jabari.
\newblock Pay for intersection priority: {A} free market mechanism for
  connected vehicles.
\newblock \emph{IEEE Transactions on Intelligent Transportation Systems}, Early
  Access:\penalty0 1--12, 2021.

\bibitem[Lippi et~al.(2013)Lippi, Bertini, and Frasconi]{lippi2013short}
Marco Lippi, Matteo Bertini, and Paolo Frasconi.
\newblock Short-term traffic flow forecasting: {A}n experimental comparison of
  time-series analysis and supervised learning.
\newblock \emph{IEEE Transactions on Intelligent Transportation Systems},
  14\penalty0 (2):\penalty0 871--882, 2013.

\bibitem[Liu and Sun(2014)]{liu2014length}
Henry~X Liu and Jie Sun.
\newblock Length-based vehicle classification using event-based loop detector
  data.
\newblock \emph{Transportation Research Part C: Emerging Technologies},
  38:\penalty0 156--166, 2014.

\bibitem[Liu et~al.(2008)Liu, Ma, Hu, Wu, and Yu]{liu2008smart}
Henry~X Liu, Wenteng Ma, Heng Hu, Xinkai Wu, and Guizhen Yu.
\newblock {SMART}-{SIGNAL}: {S}ystematic monitoring of arterial road traffic
  signals.
\newblock In \emph{2008 11th International IEEE Conference on Intelligent
  Transportation Systems}, pages 1061--1066, 2008.

\bibitem[Liu et~al.(2017)Liu, Zheng, Feng, and Chen]{liu2017short}
Yipeng Liu, Haifeng Zheng, Xinxin Feng, and Zhonghui Chen.
\newblock Short-term traffic flow prediction with {C}onv-{LSTM}.
\newblock In \emph{Proceedings of the 9th International Conference on Wireless
  Communications and Signal Processing (WCSP)}, pages 1--6, 2017.

\bibitem[Lu and Osorio(2018)]{lu2018probabilistic}
Jing Lu and Carolina Osorio.
\newblock A probabilistic traffic-theoretic network loading model suitable for
  large-scale network analysis.
\newblock \emph{Transportation Science}, 52\penalty0 (6):\penalty0 1509--1530,
  2018.

\bibitem[Lv et~al.(2014)Lv, Duan, Kang, Li, and Wang]{lv2014traffic}
Yisheng Lv, Yanjie Duan, Wenwen Kang, Zhengxi Li, and Fei-Yue Wang.
\newblock Traffic flow prediction with big data: {A} deep learning approach.
\newblock \emph{IEEE Transactions on Intelligent Transportation Systems},
  16\penalty0 (2):\penalty0 865--873, 2014.

\bibitem[Ma et~al.(2015)Ma, Tao, Wang, Yu, and Wang]{ma2015long}
Xiaolei Ma, Zhimin Tao, Yinhai Wang, Haiyang Yu, and Yunpeng Wang.
\newblock Long short-term memory neural network for traffic speed prediction
  using remote microwave sensor data.
\newblock \emph{Transportation Research Part C: Emerging Technologies},
  54:\penalty0 187--197, 2015.

\bibitem[Ma et~al.(2017)Ma, Dai, He, Ma, Wang, and Wang]{ma2017learning}
Xiaolei Ma, Zhuang Dai, Zhengbing He, Jihui Ma, Yong Wang, and Yunpeng Wang.
\newblock Learning traffic as images: {A} deep convolutional neural network for
  large-scale transportation network speed prediction.
\newblock \emph{Sensors}, 17\penalty0 (4):\penalty0 818, 2017.

\bibitem[Maas et~al.(2013)Maas, Hannun, and Ng]{maas2013rectifier}
Andrew~L Maas, Awni~Y Hannun, and Andrew~Y Ng.
\newblock Rectifier nonlinearities improve neural network acoustic models.
\newblock In \emph{Proceedings of the 30th International Conference on Machine
  Learning}, volume~28, pages 1--6, 2013.

\bibitem[Mackenzie et~al.(2018)Mackenzie, Roddick, and
  Zito]{mackenzie2018evaluation}
Jonathan Mackenzie, John~F Roddick, and Rocco Zito.
\newblock An evaluation of {HTM} and {LSTM} for short-term arterial traffic
  flow prediction.
\newblock \emph{IEEE Transactions on Intelligent Transportation Systems},
  20\penalty0 (5):\penalty0 1847--1857, 2018.

\bibitem[May et~al.(2004)May, Coifman, Cayford, and Merritt]{may2004automatic}
Adolf May, Benjamin Coifman, Randall Cayford, and Greg Merritt.
\newblock \emph{Automatic diagnostics of loop detectors and the data collection
  system in the {B}erkeley highway lab}.
\newblock California PATH Research Report UCB-ITS-PRR-2004-13, 2004.
\newblock URL \url{https://escholarship.org/uc/item/3qx70239}.

\bibitem[May et~al.(2003)May, Cayford, Coifman, and Merritt]{may2003loop}
Adolf~D May, Randall Cayford, Ben Coifman, and Greg Merritt.
\newblock \emph{Loop detector data collection and travel time measurement in
  the {B}erkeley highway laboratory}.
\newblock California PATH Research Report UCB-ITS-PRR-2003-17, 2003.
\newblock URL \url{https://escholarship.org/uc/item/0248v7w8}.

\bibitem[Oh et~al.(2005)Oh, Ritchie, and Oh]{oh2005exploring}
Cheol Oh, Stephen~G Ritchie, and Jun-Seok Oh.
\newblock Exploring the relationship between data aggregation and
  predictability to provide better predictive traffic information.
\newblock \emph{Transportation Research Record}, 1935\penalty0 (1):\penalty0
  28--36, 2005.

\bibitem[Osorio and Fl{\"o}tter{\"o}d(2014)]{osorio2014capturing}
Carolina Osorio and Gunnar Fl{\"o}tter{\"o}d.
\newblock Capturing dependency among link boundaries in a stochastic dynamic
  network loading model.
\newblock \emph{Transportation Science}, 49\penalty0 (2):\penalty0 420--431,
  2014.

\bibitem[Osorio and Punzo(2019)]{osorio2019efficient}
Carolina Osorio and Vincenzo Punzo.
\newblock Efficient calibration of microscopic car-following models for
  large-scale stochastic network simulators.
\newblock \emph{Transportation Research Part B: Methodological}, 119:\penalty0
  156--173, 2019.

\bibitem[Ramezani et~al.(2015)Ramezani, Haddad, and
  Geroliminis]{ramezani2015dynamics}
Mohsen Ramezani, Jack Haddad, and Nikolas Geroliminis.
\newblock Dynamics of heterogeneity in urban networks: aggregated traffic
  modeling and hierarchical control.
\newblock \emph{Transportation Research Part B: Methodological}, 74:\penalty0
  1--19, 2015.

\bibitem[Recht et~al.(2010)Recht, Fazel, and Parrilo]{recht2010guaranteed}
Benjamin Recht, Maryam Fazel, and Pablo~A Parrilo.
\newblock Guaranteed minimum-rank solutions of linear matrix equations via
  nuclear norm minimization.
\newblock \emph{SIAM Review}, 52\penalty0 (3):\penalty0 471--501, 2010.

\bibitem[Seo et~al.(2017)Seo, Bayen, Kusakabe, and Asakura]{seo2017traffic}
T.~Seo, A.~Bayen, T.~Kusakabe, and Y.~Asakura.
\newblock Traffic state estimation on highway: {A} comprehensive survey.
\newblock \emph{Annual Reviews in Control}, 43:\penalty0 128--151, 2017.

\bibitem[Shaaban et~al.(2016)Shaaban, Khan, and Hamila]{shaaban2016literature}
Khaled Shaaban, Muhammad~Asif Khan, and Rida Hamila.
\newblock Literature review of advancements in adaptive ramp metering.
\newblock \emph{Procedia Computer Science}, 83:\penalty0 203--211, 2016.

\bibitem[Skabardonis et~al.(1996)Skabardonis, Petty, Noeimi, Rydzewski, and
  Varaiya]{skabardonis1996880}
Alexander Skabardonis, Karl Petty, Hisham Noeimi, Daniel Rydzewski, and Pravin
  Varaiya.
\newblock {I}-880 field experiment: {D}ata-base development and incident delay
  estimation procedures.
\newblock \emph{Transportation Research Record}, 1554\penalty0 (1):\penalty0
  204--212, 1996.

\bibitem[Smaglik et~al.(2007)Smaglik, Sharma, Bullock, Sturdevant, and
  Duncan]{smaglik2007event}
Edward~J Smaglik, Anuj Sharma, Darcy~M Bullock, James~R Sturdevant, and Gary
  Duncan.
\newblock Event-based data collection for generating actuated controller
  performance measures.
\newblock \emph{Transportation Research Record}, 2035\penalty0 (1):\penalty0
  97--106, 2007.

\bibitem[Smola and Sch{\"o}lkopf(2004)]{smola2004tutorial}
Alex~J Smola and Bernhard Sch{\"o}lkopf.
\newblock A tutorial on support vector regression.
\newblock \emph{Statistics and Computing}, 14\penalty0 (3):\penalty0 199--222,
  2004.

\bibitem[Srebro et~al.(2004)Srebro, Alon, and
  Jaakkola]{srebro2004generalization}
Nathan Srebro, Noga Alon, and Tommi~S Jaakkola.
\newblock Generalization error bounds for collaborative prediction with
  low-rank matrices.
\newblock In \emph{NIPS}, volume~4, pages 5--27, 2004.

\bibitem[Stock and Watson(2001)]{stock2001vector}
James~H Stock and Mark~W Watson.
\newblock Vector autoregressions.
\newblock \emph{Journal of Economic Perspectives}, 15\penalty0 (4):\penalty0
  101--115, 2001.

\bibitem[Sun(2015)]{sun2015stochastic}
Lu~Sun.
\newblock Stochastic projection-factoring method based on piecewise stationary
  renewal processes for mid-and long-term traffic flow modeling and
  forecasting.
\newblock \emph{Transportation Science}, 50\penalty0 (3):\penalty0 998--1015,
  2015.

\bibitem[Suykens and Vandewalle(1999)]{suykens1999least}
Johan~AK Suykens and Joos Vandewalle.
\newblock Least squares support vector machine classifiers.
\newblock \emph{Neural Processing Letters}, 9\penalty0 (3):\penalty0 293--300,
  1999.

\bibitem[Tan et~al.(2016)Tan, Wu, Shen, Jin, and Ran]{tan2016short}
Huachun Tan, Yuankai Wu, Bin Shen, Peter~J Jin, and Bin Ran.
\newblock Short-term traffic prediction based on dynamic tensor completion.
\newblock \emph{IEEE Transactions on Intelligent Transportation Systems},
  17\penalty0 (8):\penalty0 2123--2133, 2016.

\bibitem[Thodi et~al.(2021{\natexlab{a}})Thodi, Khan, Jabari, and
  Menendez]{thodi2021incorporating}
Bilal~Thonnam Thodi, Zaid~Saeed Khan, Saif~Eddin Jabari, and Monica Menendez.
\newblock Incorporating kinematic wave theory into a deep learning method for
  high-resolution traffic speed estimation.
\newblock \emph{arXiv preprint arXiv:2102.02906}, 2021{\natexlab{a}}.
\newblock URL \url{https://arxiv.org/abs/2102.02906}.

\bibitem[Thodi et~al.(2021{\natexlab{b}})Thodi, Khan, Jabari, and
  Menendez]{thodi2021learning}
Bilal~Thonnam Thodi, Zaid~Saeed Khan, Saif~Eddin Jabari, and Monica Menendez.
\newblock Learning traffic speed dynamics from visualizations.
\newblock \emph{arXiv preprint arXiv:2105.01423}, 2021{\natexlab{b}}.
\newblock URL \url{https://arxiv.org/abs/2105.01423}.

\bibitem[van Wieringen(2015)]{van2015lecture}
Wessel~N van Wieringen.
\newblock Lecture notes on ridge regression.
\newblock \emph{arXiv preprint arXiv:1509.09169}, 2015.
\newblock URL \url{https://arxiv.org/abs/1509.09169}.

\bibitem[Vapnik(1995)]{vapniknature}
Vladimir Vapnik.
\newblock \emph{The nature of statistical learning theory}.
\newblock Springer, New York, 1995.

\bibitem[Vlahogianni and Karlaftis(2011)]{vlahogianni2011temporal}
Eleni Vlahogianni and Matthew Karlaftis.
\newblock Temporal aggregation in traffic data: {I}mplications for statistical
  characteristics and model choice.
\newblock \emph{Transportation Letters}, 3\penalty0 (1):\penalty0 37--49, 2011.

\bibitem[Wang and Wong(2002)]{wang2002autoregressive}
Wenyi Wang and Albert~K Wong.
\newblock Autoregressive model-based gear fault diagnosis.
\newblock \emph{Journal of Vibration and Acoustics}, 124\penalty0 (2):\penalty0
  172--179, 2002.

\bibitem[Wang and Nihan(2004)]{wang2004dynamic}
Yinhai Wang and Nancy~L Nihan.
\newblock Dynamic estimation of freeway large-truck volumes based on
  single-loop measurements.
\newblock \emph{Intelligent Transportation Systems}, 8\penalty0 (3):\penalty0
  133--141, 2004.

\bibitem[Whitt(2002)]{whitt2002stochastic}
Ward Whitt.
\newblock \emph{Stochastic-process limits: {A}n introduction to
  stochastic-process limits and their application to queues}.
\newblock Springer--Verlag, New York, NY, 2002.

\bibitem[Williams and Hoel(2003)]{williams2003modeling}
Billy~M Williams and Lester~A Hoel.
\newblock Modeling and forecasting vehicular traffic flow as a seasonal {ARIMA}
  process: {T}heoretical basis and empirical results.
\newblock \emph{Journal of Transportation Engineering}, 129\penalty0
  (6):\penalty0 664--672, 2003.

\bibitem[Wu et~al.(2004)Wu, Ho, and Lee]{wu2004travel}
Chun-Hsin Wu, Jan-Ming Ho, and Der-Tsai Lee.
\newblock Travel-time prediction with support vector regression.
\newblock \emph{IEEE Transactions on Intelligent Transportation Systems},
  5\penalty0 (4):\penalty0 276--281, 2004.

\bibitem[Wu and Liu(2014)]{wu2014using}
Xinkai Wu and Henry~X Liu.
\newblock Using high-resolution event-based data for traffic modeling and
  control: {A}n overview.
\newblock \emph{Transportation Research Part C: Emerging Technologies},
  42:\penalty0 28--43, 2014.

\bibitem[Wu et~al.(2010)Wu, Liu, and Gettman]{wu2010identification}
Xinkai Wu, Henry~X Liu, and Douglas Gettman.
\newblock Identification of oversaturated intersections using high-resolution
  traffic signal data.
\newblock \emph{Transportation Research Part C: Emerging Technologies},
  18\penalty0 (4):\penalty0 626--638, 2010.

\bibitem[Xu and Yin(2013)]{xu2013block}
Yangyang Xu and Wotao Yin.
\newblock A block coordinate descent method for regularized multiconvex
  optimization with applications to nonnegative tensor factorization and
  completion.
\newblock \emph{SIAM Journal on Imaging Sciences}, 6\penalty0 (3):\penalty0
  1758--1789, 2013.

\bibitem[Yildirimoglu and Geroliminis(2013)]{yildirimoglu2013experienced}
Mehmet Yildirimoglu and Nikolas Geroliminis.
\newblock Experienced travel time prediction for congested freeways.
\newblock \emph{Transportation Research Part B: Methodological}, 53:\penalty0
  45--63, 2013.

\bibitem[Zhang and Wang(2013)]{zhang2013gaussian}
Guohui Zhang and Yinhai Wang.
\newblock A gaussian kernel-based approach for modeling vehicle headway
  distributions.
\newblock \emph{Transportation Science}, 48\penalty0 (2):\penalty0 206--216,
  2013.

\bibitem[Zheng et~al.(2018)Zheng, Jabari, Liu, and Lin]{zheng2018traffic}
Fangfang Zheng, Saif~Eddin Jabari, Henry Liu, and DianChao Lin.
\newblock Traffic state estimation using stochastic {L}agrangian dynamics.
\newblock \emph{Transportation Research Part B: Methodological}, 115:\penalty0
  143--165, 2018.

\end{thebibliography}
}
	
	
	
	
	

\end{document}